\documentclass[opre,nonblindrev]{informs3}
\OneAndAHalfSpacedXI

\usepackage{natbib}
\bibpunct[, ]{(}{)}{,}{a}{}{,}%
\def\bibfont{\small}%

\TheoremsNumberedThrough
\EquationsNumberedThrough

\usepackage{amsmath, comment, tikz, graphicx, algorithm, algpseudocode, multirow, multicol, subcaption, bbm, epsfig, enumerate, amsfonts, enumitem, subeqnarray, url, color, array, xcolor, booktabs, endnotes, float, mathpazo, pdfpages, bm, pdflscape}

\usetikzlibrary{decorations.pathreplacing}
\usetikzlibrary{math}
\usepackage[skip=2pt, font=footnotesize, labelfont=bf, labelsep=period]{caption}

\usepackage{pgfplots}
    \pgfplotsset{
        compat=1.9,
    }
\usepackage[backref = false, bookmarks, breaklinks=true, colorlinks = true, plainpages = false, citecolor = DarkBlue , urlcolor = DarkBlue, filecolor = DarkBlue, linkcolor = DarkBlue]{hyperref}
\definecolor{DarkBlue}{rgb}{0,0.08,0.45}

\let\footnote=\endnote

\renewcommand{\theRRHFirstLine}{}
\renewcommand{\theRRHSecondLine}{}
\renewcommand{\theLRHFirstLine}{}
\renewcommand{\theLRHSecondLine}{}

\newcommand{\tabincell}[2]{\begin{tabular}{@{}#1@{}}#2\end{tabular}}

\usepackage[backref = false, bookmarks]{hyperref}
\hypersetup{hidelinks}

\renewcommand{\theRRHFirstLine}{}
\renewcommand{\theRRHSecondLine}{}
\renewcommand{\theLRHFirstLine}{}
\renewcommand{\theLRHSecondLine}{}

\algdef{SE}[DOWHILE]{Do}{doWhile}{\algorithmicdo}[1]{\algorithmicwhile\ #1}%

\ECRepeatTheorems
\EquationsNumberedThrough

\begin{document}

\TITLE{Harmonizing SAA and DRO}

\ARTICLEAUTHORS{

\AUTHOR{Ziliang Jin$^{\textup a}$, Jianqiang Cheng$^{\textup b}$, Daniel Zhuoyu Long$^{{\textup c}}$, Kai Pan$^{{\textup d}}$}
\vspace{0.1cm}
\AFF{$^{\textup a}$School of Economics and Management, Southeast University, Nanjing, China}
\AFF{$^{\textup b}$College of Engineering, University of Arizona, Tucson, AZ 85721, USA}
\AFF{$^{\textup c}$Faculty of Engineering, The Chinese University of Hong Kong, New Territories, Hong Kong}
\AFF{$^{\textup d}$Faculty of Business, The Hong Kong Polytechnic University, Kowloon, Hong Kong}
\vspace{0.1cm}
\AFF{Contact: \EMAIL{ziliang.jin@seu.edu.cn} (ZJ), \EMAIL{jqcheng@arizona.edu} (JC), \\ \EMAIL{zylong@se.cuhk.edu.hk} (DZL), \EMAIL{kai.pan@polyu.edu.hk} (KP)}

}

\ABSTRACT{Decision-makers often encounter uncertainty, and the distribution of uncertain parameters plays a crucial role in making reliable decisions. However, complete information is rarely available. The sample average approximation (SAA) approach utilizes historical data to address this, but struggles with insufficient data. Conversely, moment-based distributionally robust optimization (DRO) effectively employs partial distributional information but can yield conservative solutions even with ample data. To bridge these approaches, we propose a novel method called \emph{harmonizing optimization (HO)}, which integrates SAA and DRO by adaptively adjusting the weights of \textit{data} and \textit{information} based on sample size $N$. This allows HO to amplify data effects in large samples while emphasizing information in smaller ones. More importantly, HO performs well across varying data sizes without needing to classify them as large or small. We provide practical methods for determining these weights and demonstrate that HO offers finite-sample performance guarantees, proving asymptotic optimality when the weight of \textit{information} follows a $1/\sqrt{N}$-rate. In addition, HO can be applied to enhance scenario reduction, improving approximation quality and reducing completion time by retaining critical information from reduced scenarios. Numerical results show significant advantages of HO in solution quality compared to Wasserstein-based DRO, and highlight its effectiveness in scenario reduction.}


\maketitle

\vspace{-.8cm}

\section{Introduction} \label{sec:intro}





A major challenge in solving stochastic programming models \citep{birge2011introduction} for decision-making problems under uncertainty is that the probability distribution of the random parameter is rarely known in practice.
We can often collect \emph{data} from practices to support decision-making. 
Thus, extensive studies use historical realizations/observations (i.e., data) of the random parameter to estimate its unknown distribution and then obtain an optimal or near-optimal solution. 
For instance, the well-known sample average approximation (SAA) utilizes data to derive the empirical distribution, thereby approximating the unknown distribution. 
The SAA approach to stochastic programming problems has attractive performance guarantees.
First, the SAA model's size (the number of decision variables and constraints) scales linearly in the number of samples \citep{shapiro2003monte}.
Second, the SAA model exhibits asymptotic optimality under mild conditions, showcasing that the obtained optimal value and decisions are guaranteed to converge to those of the original models as the number of samples goes to infinity \citep{shapiro2021lectures}. 
With various theoretical studies \citep{kleywegt2002sample, hu2012sample, gotoh2025data}, SAA has also been applied in practice, including supply chain management \citep{schutz2009supply, cheung2019sampling, lin2022data}, power system operations \citep{takriti2000incorporating, porras2023tight, schindler2024planner}, and financial planning \citep{alexander2006minimizing, xu2009smooth}. 

Despite its considerable success, SAA has limitations.
As a purely data-driven paradigm, it relies solely on data and does not incorporate any additional information (e.g., domain knowledge). 
Consequently, it exhibits poor performance when a stochastic programming model has limited data.
It also remains challenging to determine whether a given amount of data is sufficient for SAA to yield a reliable solution, particularly in high-dimensional problems. 
This difficulty arises from the fact that while SAA is known to perform well with large datasets, the definition of ``large'' can vary significantly depending on the specific problem at hand. 
Thus, when faced with a dataset that may not meet the criteria for being considered large, it becomes difficult to determine whether SAA is the appropriate approach for solving the problem.

Besides data, we may also possess \emph{partial distributional information} about the random parameters (e.g., moment information) in practice.
This information setting is common across various industries.
In power systems, we can derive mean and correlation information about uncertain solar power generation using external physical factors like solar radiation and precipitation \citep{luo2021deep}.
Similarly, in transportation systems, we can obtain mean information about uncertain vehicle trips from external behavioral factors like vehicle velocity and acceleration \citep{bahari2021injecting}.
We can incorporate such partial distributional information to help address the uncertainty, where distributionally robust optimization (DRO) can serve this purpose.

The DRO approach provides a robust optimal solution that performs the best under the worst-case distribution in a predefined distributional ambiguity set \citep{scarf1958min}.
The ambiguity set, containing all relevant distributions, can be described using partial distributional information about the uncertainty, such as moment information \citep{rahimian2019distributionally}.
In particular, the moment-based ambiguity set considers distributions whose moments satisfy certain conditions, such as restricting their first and second moments to be close to nominal moments \citep{delage2010distributionally, zymler2013distributionally, wiesemann2014distributionally}.
By leveraging the partial distributional information, the moment-based DRO provides solutions that have superior performance compared to those obtained by SAA in the out-of-sample tests \citep{delage2010distributionally, liu2017distributionally, shehadeh2023distributionally}.
Thus, the moment-based DRO has received extensive attention, with proven performance guarantees \citep{delage2010distributionally, wiesemann2014distributionally, long2023supermodularity} and 
a wide range of applications, including transportation management \citep{ghosal2020distributionally, basciftci2021distributionally, shehadeh2023distributionally}, machine learning \citep{lanckriet2002robust, nguyen2020robust, li2022moment}, and finance \citep{ghaoui2003worst, popescu2007robust, rujeerapaiboon2016robust, liu2017distributionally}. 

Unlike the SAA approach, which exhibits asymptotic optimality, the moment-based DRO approach may yield a conservative solution when we have a large amount of data.
More specifically, the effectiveness of these two approaches generally depends on the data size; that is, the SAA approach performs well with a large data size, while the moment-based DRO excels with a small data size.
Thus, to adopt an appropriate approach to the problem, decision-makers may initially assess the size of available data, judging whether it is large or small.
However, assessing data size presents significant challenges for decision-makers for the following two reasons.
\begin{enumerate}[label={(\roman*)}, wide, labelwidth=!, labelindent=0pt]
\item Besides the amount of data, assessing data size also requires considering uncertain parameters and the model's dimensionality. 
The same amount of data may be sufficient (i.e., considered large) for some parameters and models, but insufficient (i.e., considered small) for others.
For example, estimating the distribution of uncertain parameters with clear features and low dimensionality requires a relatively small amount of data, whereas a relatively large amount of data may be needed otherwise.
Similarly, what is considered a large amount of data for a low-dimensional problem may be deemed small for a high-dimensional one.


\item Before solving the problem, a quantitative relationship between the required amount of data and the uncertain parameters and the model’s dimensionality may not be established.
Thus, decision-makers cannot determine the exact amount of required data and assess whether the given amount of data is ``large'' or ``small'' for the problem at hand.
\end{enumerate}

The significant challenges in assessing data size can easily lead to misjudgments. 
Such errors can result in selecting inappropriate approaches, thereby compromising solution quality.
When data is erroneously assessed as large (when in fact it is small) and the SAA approach is consequently employed, the obtained solution may lack robustness. 
Conversely, when data is mistakenly deemed small (when in fact it is large) and the DRO approach is consequently adopted, the obtained solution may be conservative, failing to fully use the value of the data.
To address the challenges of assessing data size, we propose an innovative approach that eliminates the need to evaluate data size and performs consistently well for any data size.

In this paper, we integrate both data and partial distributional information to address the uncertainty without incurring the aforementioned drawbacks, leading to an approach harmonizing the SAA and DRO approaches, namely \emph{harmonizing optimization (HO)}.
More importantly, we can adaptively adjust the weights of \textit{data} and \textit{information} (i.e., the significance of their roles in this approach) according to the available sample size.
Such an approach offers an attractive step toward bridging the data and information to address the uncertainty in stochastic programs and support data-driven decision-making.
Specifically, it works well for any data size, whether large or small, allowing decision-makers to use it directly without the need to evaluate the data size.
Notably, this approach can be extended to scenario reduction and significantly improve its performance by incorporating partial information.
Specifically, when we reduce the number of scenarios included in a stochastic program, the HO approach helps retain the information about the dropped scenarios, thereby enhancing the quality of approximations. 
Thus, HO can help decision-makers reduce the number of scenarios to consider, significantly alleviating the computational difficulty of decision-making under uncertainty.
We summarize our contributions as follows:
\begin{enumerate}[label={(\roman*)}, wide, labelwidth=!, labelindent=0pt]
\item We propose a novel approach, referred to as HO, aiming to obtain superior-quality solutions to decision-making problems under uncertainty.
HO utilizes both data and information by harmonizing SAA and moment-based DRO approaches.
In HO, the weights of data and information can be adaptively adjusted according to the sample size, amplifying the significance of data in large samples and emphasizing the influence of information in limited samples.
Consequently, HO works well for any data size, enabling direct use without evaluating the data size.
\item We show a finite-sample performance guarantee for our proposed HO model. 
The HO model also ensures asymptotic optimality, holding performance guarantees when the weight parameter is in a $1/\sqrt{N}$-rate, where $N$ denotes the number of given samples.
Moreover, the HO model can be reformulated as a computationally tractable model, such as a linear programming (LP) or semidefinite programming (SDP) model.
\item We show the applicability and strength of our HO method in scenario reduction.
Compared with existing approaches, it can obtain a superior approximation with greater efficiency for stochastic programming problems.
More importantly, it only needs to consider a few scenarios to maintain effectiveness, regardless of the original sample size. 
\item We conduct numerical experiments to reveal the significance of HO in addressing decision-making under uncertainty and scenario reduction. 
We compare HO against the Wasserstein-based DRO in the mean-risk portfolio optimization problem.
The HO consistently stands out in out-of-sample performance across all sample sizes, with particularly notable improvements when the size is limited.
We also compare HO against prevailing scenario reduction approaches in the lot sizing problem.
With the same number of reduced scenarios, the HO provides a more accurate approximation of the original problem with all the scenarios while significantly reducing computational time.
\end{enumerate}

Note that \cite{tsang2025tradeoff} independently propose a similar framework recently, called the tradeoff (TRO) approach, which combines SAA and DRO using a sample-size-dependent weight.
Our focus and application of HO differ from theirs in three aspects.
\textit{First}, concerning the challenges in assessing data size and ensuring consistently good performance across all data sizes, we focus on integrating data and partial distributional information. 
Specifically, we harmonize SAA and moment-based DRO, rather than Wasserstein-based DRO. 
Proposition \ref{prop: wasserstein_equivalence} in Section \ref{sec:intro-ho} shows that combining SAA and Wasserstein-based DRO, as studied in \cite{tsang2025tradeoff}, is essentially equivalent to using Wasserstein-based DRO solely.
\textit{Second}, the weight to balance \textit{data} and \textit{information} (denoted by $\lambda$ in Section \ref{sec:intro-ho}) is crucial in HO, and we discuss the selection of $\lambda$ in detail.
Specifically, we provide an explicit form $\lambda = C/\sqrt{N}$, along with multiple methods to estimate $C$ (see Section \ref{sec: estimate lambda}), whereas \cite{tsang2025tradeoff} do not specifically characterize the weight in their framework.
With this form, we estimate the constant $C$ only once and can then apply HO directly to the same problem across multiple instances with varying sample sizes $N$, which is common in practice (see Section \ref{sec: estimate lambda}).
\textit{Third}, we establish the practical significance of HO by applying it to substantially improve scenario reduction (see Section \ref{sec: Scenario Reduction}), which is not explored by \cite{tsang2025tradeoff}.
Scenario reduction is a crucial and widely used approach for addressing computational challenges of the SAA model with many scenarios.
We show that HO can significantly enhance scenario reduction, outperforming the existing approach by improving approximation quality and reducing computational time.
Moreover, \cite{wang2025learning} and \cite{tsang2025alg} propose similar frameworks combining SAA and DRO, applying them to specific problems in machine learning and facility location, respectively. 


The remainder of this paper is organized as follows.
Section \ref{sec: Existing Models} illustrates existing models, including the stochastic programming model and the general DRO model.
In Section \ref{sec: HO}, we propose the HO model, establish its theoretical performance guarantees, and provide its computationally tractable reformulation.
Section \ref{sec: Scenario Reduction} demonstrates the applicability and strength of HO in scenario reduction.
Section \ref{sec: Numerical Experiments} provides extensive numerical experiments to validate the theoretical results and present practical insights.
Section \ref{sec: Conclusion} concludes the paper.
Notations are introduced in Appendix \ref{sec: notations}.
All proofs are presented in the Appendix if not specified.


\section{Existing Models} \label{sec: Existing Models}

Given a nonempty, convex and compact set $\mathcal{X} \subseteq \mathbb{R}^n$, an uncertainty set $\mathcal{S} \subseteq \mathbb{R}^m$, a function $f: \mathbb{R}^n \times \mathbb{R}^m \rightarrow \mathbb{R}$, and the joint probability distribution $\mathbb{P}$ of a random vector $\boldsymbol{\xi} \in \mathcal{S}$, 
we introduce the following stochastic program that seeks an $\mathbf{x} \in \mathcal{X}$ to minimize the expectation of $f ( \mathbf{x}, \boldsymbol{\xi} )$:
\vspace{-2mm}
\begin{align} \label{model0}
    V^* = \min_{\mathbf{x} \in \mathcal{X}} \mathbb{E}_{\mathbb{P}} \left[ f \left( \mathbf{x}, \boldsymbol{\xi} \right) \right] = \min_{\mathbf{x} \in \mathcal{X}} \int_{\boldsymbol{\xi} \in \mathcal{S}} f \left( \mathbf{x}, \boldsymbol{\xi} \right) \mathbb{P} \left( \boldsymbol{\xi} \right).
\end{align}
We let $F ( \mathbf{x} ) = \mathbb{E}_{\mathbb{P}} [ f ( \mathbf{x}, \boldsymbol{\xi} ) ]$ and assume it 
is well-defined with any $\mathbb{P}$.
That is, for any $\mathbf{x} \in \mathcal{X}$, the function $f(\mathbf{x}, \cdot)$ is measurable and $\mathbb{E}_{\mathbb{P}} [ |f ( \mathbf{x}, \boldsymbol{\xi} )| ] < \infty$. 
We also assume that for any $\mathbf{x} \in \mathcal{X}$, $f(\mathbf{x}, \cdot)$ is convex and Lipschitz continuous.
We use $\sigma^2(\mathbf{x})$ to denote the variance of $f ( \mathbf{x}, \boldsymbol{\xi} )$ for any $\mathbf{x} \in \mathcal{X}$.
Model \eqref{model0} can represent either a single-stage or multi-stage stochastic program. 
When it represents a multi-stage stochastic program, $\mathbf{x}$ denotes the first-stage decision variables and $f (\mathbf{x}, \boldsymbol{\xi})$ is the total cost with a given $\mathbf{x}$ and a realized scenario path $\boldsymbol{\xi}$ over multiple stages.



The distribution $\mathbb{P}$ is generally unknown in practice, leading to difficulty solving model \eqref{model0}.
However, $\mathbb{P}$ is often partially observable through a finite number of historical realizations of the random vector $\boldsymbol{\xi}$.
Let $\tilde{\boldsymbol{\xi}}_1, \ldots, \tilde{\boldsymbol{\xi}}_N$ be $N$ independently and identically distributed (iid) samples of $\boldsymbol{\xi}$, and $\mathbb{P}_0 = (1/N) \times \sum_{j=1}^N \delta_{\tilde{\boldsymbol{\xi}}_j}$, where $\delta_{\boldsymbol{\xi}}$ is the Dirac measure concentrating unit mass at $\boldsymbol{\xi} \in \mathbb{R}^m$.
With these samples, we can naturally use the SAA approach to approximate model \eqref{model0} as
\vspace{-2mm}
\begin{align} \label{model: SAA}
    V_N = \min_{\mathbf{x} \in \mathcal{X}} F_N \left( \mathbf{x} \right) 
    = \min_{\mathbf{x} \in \mathcal{X}} \mathbb{E}_{\mathbb{P}_0} \left[ f \left( \mathbf{x}, \boldsymbol{\xi} \right) \right] 
    = \min_{\mathbf{x} \in \mathcal{X}} \frac{1}{N} \sum_{j=1}^N f \left( \mathbf{x}, \tilde{\boldsymbol{\xi}}_j \right).
\end{align}

\vspace{-2mm}

The optimal value of model \eqref{model: SAA} (i.e., $V_N$) can converge to its counterpart of the original model \eqref{model0} (i.e., $V^*$) with probability $1$ (w.p. 1) when $N$ grows to infinity (see the proposition below), exhibiting the asymptotic optimality of model \eqref{model: SAA}.

\begin{proposition}[Proposition 5.2, \citealt{shapiro2021lectures}] \label{prop: converge}
    If $F_N(\mathbf{x})$ converges to $F(\mathbf{x})$ w.p. 1 as $N \rightarrow \infty$, uniformly on $\mathcal{X}$, then $V_N \rightarrow V^*$ w.p. 1 as $N \rightarrow \infty$.
\end{proposition}


In addition, for any $\mathbf{x} \in \mathcal{X}$, we can use the value of $F_N(\mathbf{x})$ to estimate the range of the value of $F(\mathbf{x})$ in the following proposition.
\begin{proposition} \label{prop: interval}
Given any $\mathbf{x} \in \mathcal{X}$ and $\alpha \in [0,1]$, we have the following (approximate) $100 (1-\alpha) \%$ confidence interval for $F(\mathbf{x})$: 
$
[ F_N ( \mathbf{x} ) - z_{\frac{\alpha}{2}} \hat{\sigma}(\mathbf{x}) / \sqrt{N},
     \ F_N ( \mathbf{x} ) +  z_{\frac{\alpha}{2}} \hat{\sigma}(\mathbf{x}) / \sqrt{N}  ],
$
where $z_{\alpha/2} = \Phi^{-1}(1-\alpha/2)$, $\Phi$ denotes the cumulative distribution function (cdf) of the standard normal distribution, and $\hat{\sigma}^2(\mathbf{x}) = \sum_{j=1}^N ( f( \mathbf{x}, \tilde{\boldsymbol{\xi}}_j) - F_N(\mathbf{x}) )^2 / (N-1)$.
\end{proposition}



Propositions \ref{prop: converge} and \ref{prop: interval} highlight that SAA offers performance guarantees when $N$ is large. 
Note that determining what qualifies as a ``large'' $N$ may be challenging (see Section \ref{sec:intro}).
Moreover, when $N$ is small, SAA's performance may be poor because it solely relies on limited data samples, which may not well approximate the true distribution of the uncertainty.
Next, we introduce the moment-based DRO that utilizes partial distributional information about uncertain parameters.
By leveraging this additional information, DRO maintains stable and robust performance across all sample sizes $N$, which is especially advantageous when $N$ is small.


The DRO framework assumes that the true distribution $\mathbb{P}$ of the random vector $\boldsymbol{\xi} \in \mathcal{S}\subseteq \mathbb{R}^m$ is ambiguous in a distributional set $\mathcal{D}$, by which one optimizes decisions against the worst-case distribution in $\mathcal{D}$ \citep{scarf1958min}.
We can formulate the DRO counterpart of model \eqref{model0} as:
\vspace{-2mm}
\begin{align} \tag{DRO} \label{DRO}
\min_{\boldsymbol{x} \in \mathcal{X}} \ \max_{\mathbb{P} \in \mathcal{D}} \ \mathbb{E}_{\mathbb{P}} \left[ f \left( \boldsymbol{x},\boldsymbol{\xi} \right) \right].
\end{align}

\vspace{-2mm}


We consider a moment-based ambiguity set $\mathcal{D}$ in the standard form \citep{wiesemann2014distributionally}: \vspace{-0.2cm}
\begin{align} \label{ambiguity_set}
\mathcal{D}
=\left\{ \mathbb{P} \in \mathcal{D}_0 \left( \mathbb{R}^m \times \mathbb{R}^h \right) \ \middle| \
\begin{array}{l}
\mathbb{E}_{\mathbb{P}} \left[ \mathbf{A}\boldsymbol{\xi} + \mathbf{B}\mathbf{u} \right] = \mathbf{b},  \ \
\mathbb{P} \left[ \left( \boldsymbol{\xi}, \mathbf{u} \right) \in \mathcal{C}_i \right] \in \left[ \underline{p}_i, \overline{p}_i \right], \ \forall \, i \in [I]
\end{array}
\right\},
\end{align}

\vspace{-0.2cm}

\noindent
which is explained as follows.
First, considering an additional auxiliary random vector $\mathbf{u} \in \mathbb{R}^h$ in $\mathcal{D}$, we generalize the notation $\mathbb{P}$ to represent the joint probability distribution of $\boldsymbol{\xi}$ and $\mathbf{u}$.
Second, the set $\mathcal{D}$ contains all distributions with mean values lying in an affine manifold characterized by $\mathbf{A} \in \mathbb{R}^{s \times m}$, $\mathbf{B} \in \mathbb{R}^{s \times h}$, and $\mathbf{b} \in \mathbb{R}^{s}$ and with $I$ conic representable confidence sets $\mathcal{C}_i$ for any $i \in [I]$. 
Third, for each $i \in [I]$, we have $\overline{p}_i, \underline{p}_i \in [0,1]$ and $\overline{p}_i \geq \underline{p}_i$ and define $\mathcal{C}_i$ as
\vspace{-2mm}
\begin{align*}
    \mathcal{C}_i = \left\{ \left( \boldsymbol{\xi}, \mathbf{u} \right) \in \mathbb{R}^m \times \mathbb{R}^h \ | \  \mathbf{c}_i - \left( \mathbf{C}_i \boldsymbol{\xi} + \mathbf{D}_i \mathbf{u} \right) \in \mathcal{K}_i \right\},
\end{align*}
\vspace{-9mm}

\noindent
where $\mathbf{C}_i \in \mathbb{R}^{L_i \times m}$, $\mathbf{D}_i \in \mathbb{R}^{L_i \times h}$, $\mathbf{c}_i \in \mathbb{R}^{L_i}$, and $\mathcal{K}_i$ is a proper cone. 
Note that including the auxiliary random vector $\mathbf{u}$ helps model various structural information about the marginal distribution of $\boldsymbol{\xi}$ while ensuring all the information about the true marginal distribution of $\boldsymbol{\xi}$ (denoted by $\mathbb{P}_{\boldsymbol{\xi}}^*$) is included in $\mathcal{D}$, i.e., $\mathbb{P}_{\boldsymbol{\xi}}^* \in \Pi_{\boldsymbol{\xi}} \mathcal{D}$.
We can recognize several popular moment-based ambiguity sets in the literature as special cases of the ambiguity set $\mathcal{D}$ in \eqref{ambiguity_set} (see Appendix \ref{sec: Special Cases} for details).

Relying solely on partial distributional information, which may be collected from domain knowledge or inferred from other information sources, DRO maintains stable and robust performance for any sample size $N$, making it especially advantageous when $N$ is limited.
However, unlike the SAA approach that has asymptotic optimality, its advantages diminish as $N$ grows.
Note that determining what qualifies as a ``small'' $N$ may be challenging (see Section \ref{sec:intro}).
In the following section, we harmonize the SAA and DRO approaches to maintain the benefits of both approaches without worrying whether $N$ is large or small.

\section{Harmonizing Optimization} \label{sec: HO}

In this section, we propose a novel approach (denoted by the HO approach) that integrates \textit{data} and \textit{partial distributional information} (e.g., domain knowledge) by harmonizing the SAA and DRO approaches. 
This ensures consistent and significant performance across any possible values of $N$ (i.e., sample size), thereby allowing the HO approach to be used directly with any data size.

\vspace{-2mm}

\subsection{Introduction of HO} \label{sec:intro-ho}

In our HO approach, which integrates data and partial distributional information, we use a parameter $\lambda \in [0, 1]$ to measure the weight of \textit{information} and $1-\lambda$ to measure the weight of \textit{data}.
Intuitively, when $N$ is small, $\lambda$ should be relatively large to amplify the influence of information and mitigate the impact of data.
Conversely, when $N$ is large, $\lambda$ should remain relatively small to emphasize the significance of data and limit the influence of information.
Thus, we set $\lambda = C/\sqrt{N}$ in alignment with this rationale, ensuring harmony between data and information.
Here, $C$ is a predetermined fixed constant, and we will discuss how to determine it in detail in Section \ref{sec: estimate lambda}.

Given $\lambda \in [0, 1]$ and $N$ iid samples of $\boldsymbol{\xi}$ defined in Section \ref{sec: Existing Models}, we formulate our HO model as
\begin{align} \tag{HO} \label{model_harmonizing}
\Gamma(\lambda) 
= \min_{\mathbf{x} \in \mathcal{X}} F_{\lambda}\left( \mathbf{x} \right)
= \min_{\mathbf{x} \in \mathcal{X}} \left\{ \left( 1 - \lambda \right) \mathbb{E}_{\mathbb{P}_0} \left[ f \left( \mathbf{x}, \boldsymbol{\xi} \right) \right] + \lambda \max_{\mathbb{P} \in \mathcal{D}} \mathbb{E}_{\mathbb{P}}\left[ f \left( \mathbf{x}, \boldsymbol{\xi} \right) \right] \right\},
\end{align}
where $ \mathbb{P}_0 $ and $ \mathcal{D} $ are defined in Section \ref{sec: Existing Models}. 
The following proposition shows that we can equivalently transform model \eqref{model_harmonizing} into a DRO model, where the decision is optimized against the worst-case distribution within a parameterized ambiguity set.


\begin{proposition} \label{prop: reform_to_dro}
Model \eqref{model_harmonizing} can be reformulated as
$\min_{\boldsymbol{x} \in \mathcal{X}} \ \max_{\mathbb{P}_{\textup H} \in \mathcal{D}_{\textup H} ( \lambda )} \ \mathbb{E}_{\mathbb{P}_{\textup H}} [ f ( \boldsymbol{x},\boldsymbol{\xi} ) ]$,
where $\mathcal{D}_{\textup H}(\lambda) = \{ \mathbb{P}_{\textup H} \ | \ \mathbb{P}_{\textup H} = (1-\lambda)\mathbb{P}_{0}  + \lambda \mathbb{P}_{\boldsymbol{\xi}}, \ \mathbb{P}_{\boldsymbol{\xi}} \in \Pi_{\boldsymbol{\xi}}\mathcal{D} \}$.
\end{proposition}

Clearly, when $\lambda$ varies, the size of the ambiguity set $\mathcal{D}_{\textup H} ( \lambda )$ varies accordingly. 
We have the following proposition.

\begin{proposition} \label{prop: ambiguity_set_size}
If $\mathbb{P}_0 \in \Pi_{\boldsymbol{\xi}}\mathcal{D}$, then $\mathcal{D}_{\textup H}(\lambda_2) \subseteq \mathcal{D}_{\textup H}(\lambda_1)$ for any $0 \leq \lambda_2 \leq \lambda_1 \leq 1$.
\end{proposition}

Proposition \ref{prop: ambiguity_set_size} offers decision-makers guidance on determining the weights of \emph{data} and \emph{information} in different cases of historical sample sizes.
Specifically, Proposition \ref{prop: ambiguity_set_size} offers a new perspective on the intuition behind the decrease in $\lambda$ as $N$ increases.
When $N$ grows, we have more available data to approximate $\mathbb{P}$, enabling us to make a more accurate decision.
For such a case, we need a small $\lambda$ to focus on the significance of \textit{data} and decrease the influence of \textit{information}. 
It follows that the parameterized set $\mathcal{D}_{\textup H}(\lambda)$ shrinks, thereby diminishing the conservatism of model \eqref{model_harmonizing} and leading to a more accurate decision.

Unlike \cite{tsang2025tradeoff}, we do not consider a Wasserstein ambiguity set $\mathcal{D}$ because the following proposition shows that combining SAA and Wasserstein-based DRO is equivalent to using Wasserstein-based DRO solely.
Specifically, we define $\mathcal{D}_{\textup W}(r_{\textup H}) = \{\mathbb{P} \ | \ W(\mathbb{P}, \mathbb{P}_0) \leq r_{\textup H} \}$, where $W: \mathcal{D}_0 (\mathbb{R}^m) \times \mathcal{D}_0 (\mathbb{R}^m) \rightarrow \mathbb{R}_+$ denotes the 1-Wasserstein metric, and $r_{\textup H} \in \mathbb{R}_+$ is the radius.
\begin{proposition} \label{prop: wasserstein_equivalence}
    For any $\lambda \in [0, 1]$ and $r_{\textup H} \in \mathbb{R}_+$, setting $r_{\textup W} = \lambda r_{\textup H}$, we then have
    \begin{align*}
        \left( 1 - \lambda \right) \mathbb{E}_{\mathbb{P}_0} \left[ f \left( \mathbf{x}, \boldsymbol{\xi} \right) \right] + \lambda \max_{\mathbb{P} \in \mathcal{D}_{\textup W}(r_{\textup H})} \mathbb{E}_{\mathbb{P}}\left[ f \left( \mathbf{x}, \boldsymbol{\xi} \right) \right] = \max_{\mathbb{P} \in \mathcal{D}_{\textup W}(r_{\textup W})} \mathbb{E}_{\mathbb{P}}\left[ f \left( \mathbf{x}, \boldsymbol{\xi} \right) \right], \ \forall x \in \mathcal{X}.
    \end{align*}
\end{proposition}


\vspace{-2mm}

\subsection{Finite-sample Performance Guarantee}

Proposition \ref{prop: ambiguity_set_size} reveals the impact of the weight parameter $\lambda$ on the size of the ambiguity set $\mathcal{D}_{\textup H}(\lambda)$, which in turn affects the performance of model \eqref{model_harmonizing}. 
On the one hand, if the weight $\lambda$ is too large, then the ambiguity set $\mathcal{D}_{\textup H}(\lambda)$ becomes very large, potentially leading to an overly conservative solution.
On the other hand, if the weight $\lambda$ is too small, then the model loses the value of information, potentially failing to overcome the limitations of SAA.
Therefore, it is crucial to determine an appropriate value for $\lambda$ so that an optimal solution with a good performance guarantee can be obtained. 
Since we typically have a finite number of historical samples in practice, finding the appropriate value for $\lambda$ in the finite-sample case becomes even more important.
In this section, from a statistical point of view, we estimate $\lambda$ with respect to any finite sample size $N$ to ensure a performance guarantee for model \eqref{model_harmonizing}.

Recall that the ambiguity set $\mathcal{D}$ defined in \eqref{ambiguity_set} is constructed based on moment information (e.g., mean vector and covariance matrix) about uncertainties.
All the distributions within this set satisfy the same prescribed conditions on their mean vector and covariance matrix, but differences still exist between these distributions.
To quantify such differences, we use the following Gelbrich distance, calculated based on the distributions' mean vectors and covariance matrices.

\begin{definition}[Gelbrich distance]
The Gelbrich distance $\mathcal{G}$ between two mean-covariance pairs $(\boldsymbol{\mu}_1, \boldsymbol{\Sigma}_1)$ and $(\boldsymbol{\mu}_2, \boldsymbol{\Sigma}_2)$ is calculated by
\vspace{-0.4cm}
\begin{align*}
    \mathcal{G}\left( (\boldsymbol{\mu}_1, \boldsymbol{\Sigma}_1), (\boldsymbol{\mu}_2, \boldsymbol{\Sigma}_2) \right) 
    = \bigg(  \| \boldsymbol{\mu}_1-\boldsymbol{\mu}_2 \|^2 + \textup{Tr}\Big( \boldsymbol{\Sigma}_1 + \boldsymbol{\Sigma}_2 - 2 \left( \boldsymbol{\Sigma}_2^{\frac{1}{2}} \boldsymbol{\Sigma}_1 \boldsymbol{\Sigma}_2^{\frac{1}{2}} \right)^{\frac{1}{2}} \Big) \bigg)^{\frac{1}{2}}.
\end{align*}
\end{definition}

The Gelbrich distance is a metric on $\mathbb{R}^m \times \mathbb{S}_+^{m}$;
that is, $\mathcal{G}$ is non-negative, symmetric and subadditive, and equals $0$ if and only if $(\boldsymbol{\mu}_1, \boldsymbol{\Sigma}_1) = (\boldsymbol{\mu}_2, \boldsymbol{\Sigma}_2)$ \citep{givens1984class}.
Let $\boldsymbol{\mu}_0$ and $\boldsymbol{\Sigma}_0$ denote the mean value and covariance matrix of $\boldsymbol{\xi}$ under the empirical distribution $\mathbb{P}_0$, respectively; that is, $\boldsymbol{\mu}_0 = \mathbb{E}_{\mathbb{P}_0}[\boldsymbol{\xi}]$ and $\boldsymbol{\Sigma}_0 = \mathbb{E}_{\mathbb{P}_0}[( \boldsymbol{\xi} - \boldsymbol{\mu}_0)( \boldsymbol{\xi} - \boldsymbol{\mu}_0)^{\top} ]$.
Let $\boldsymbol{\mu}({\mathbb{P}}_{\textup H})$ and $\boldsymbol{\Sigma}({\mathbb{P}}_{\textup H})$ denote the mean value and covariance matrix of $\boldsymbol{\xi}$ under any distribution $\mathbb{P}_{\textup H}$, respectively.
With any distance $\epsilon > 0$, we
define
\begin{align}
    \lambda^* = \argmin \left\{ \lambda \ \middle| \ \min_{\mathbb{P}_{\textup H} \in \partial\mathcal{D}_{\textup H}(\lambda)} \mathcal{G}\left(
        (\boldsymbol{\mu}_0, \boldsymbol{\Sigma}_0), (\boldsymbol{\mu}(\mathbb{P}_{\textup H}), \boldsymbol{\Sigma}(\mathbb{P}_{\textup H}))
    \right) \geq \epsilon \right\}, \label{best_lambda}
\end{align}
where $\partial\mathcal{D}_{\textup H}(\lambda)$ denotes the boundary of $\mathcal{D}_{\textup H}(\lambda)$.
Under a common assumption below on the true distribution $\mathbb{P}$, the ambiguity set $\mathcal{D}_{\textup H}(\lambda^*)$ provides attractive performance guarantees.

\begin{assumption} \label{ass: light tail}
    We assume $\mathbb{P}$ is a light-tailed distribution; that is, there exist an exponent $a > 2$ and $b > 0$ such that $E = \mathbb{E}_{\mathbb{P}} \left[ \exp( b \| \boldsymbol{\xi} \|^a) \right] < \infty$.
\end{assumption}

Assumption \ref{ass: light tail}, which trivially holds because the support set $\mathcal{S}$ is compact \citep{esfahani2018data}, requires that the tail of $\mathbb{P}$ decays at an exponential rate.
Let $\mathcal{P}$ denote an $m$-fold product of the true distribution $\mathbb{P}$ on $\mathcal{S}$. 
We show a finite-sample performance guarantee in the form of including $\mathbb{P}$ within the ambiguity set $\mathcal{D}_{\textup H}(\lambda^*)$ below.

\begin{proposition} \label{prop: non_asymptotic}
If $\mathbb{P}_0 \in \Pi_{\boldsymbol{\xi}}\mathcal{D}$, then for all $N \geq 1$, $m \neq 4$, and $\epsilon > 0$, the true probability distribution $\mathbb{P}$ is included in $\mathcal{D}_{\textup H}(\lambda^*)$ with a confidence at $1-\beta$; that is,
\vspace{-2mm}
\begin{align}
    \mathcal{P} \left( \mathbb{P} \in \mathcal{D}_{\textup H}(\lambda^*) \right) \geq 1-\beta, 
    \ \text{where} \ 
    \beta = \begin{cases} 
        c_1\exp \left( -c_2N\epsilon^{\max\{\frac{m}{2},2\}} \right), \ & \epsilon \leq 1 \\
        c_1\exp \left( -c_2N\epsilon^{\frac{a}{2}} \right), \ & \epsilon > 1
    \end{cases},  \label{probability_in}
\end{align}
where $c_1$ and $c_2$ are positive constants depending on $m$ and $a$, $b$, and $E$ introduced in Assumption \ref{ass: light tail}. 
Moreover, for any $\lambda \geq \lambda^*$, we have $    \mathcal{P} ( \mathbb{P} \in \mathcal{D}_{\textup H}(\lambda) )
    \geq
    \mathcal{P} ( \mathbb{P} \in \mathcal{D}_{\textup H}(\lambda^*) )$.
For any $\lambda \in [1-\beta, 1]$, we also have $\mathcal{P} ( \mathbb{P} \in \mathcal{D}_{\textup H}(\lambda) ) \geq 1-\beta$.
\end{proposition}

By \eqref{probability_in} in Proposition \ref{prop: non_asymptotic}, we can calculate
\begin{align}
\epsilon = \left( \frac{\log \left( c_1 \beta^{-1} \right) }{ c_2 N } \right)^{\frac{1}{\max\{\frac{m}{2},2\}}}, \ \textup{if } N \geq \frac{\log \left( c_1 \beta^{-1} \right)}{c_2}; \ \textup{ and } \ \epsilon = \left( \frac{\log \left( c_1 \beta^{-1} \right) }{ c_2 N } \right)^{\frac{2}{a}}, \ \textup{otherwie.}
    \label{eqn: epsilon_value}
\end{align}
With a given $\epsilon$ calculated by \eqref{eqn: epsilon_value}, we design a bisection search algorithm to determine $\lambda^*$ efficiently (see Algorithm \ref{alg: bisection} in Appendix \ref{app: Bisection Search Algorithm}).
Given the obtained $\lambda^*$, Proposition \ref{prop: non_asymptotic} ensures that the true distribution $\mathbb{P}$ is in the ambiguity set $\mathcal{D}_{\textup H}(\lambda^*)$ with a confidence at $1-\beta$.

\subsection{Asymptotic Optimality}

Proposition \ref{prop: non_asymptotic} provides a performance guarantee for model \eqref{model_harmonizing} when the sample size $N$ is finite.
In this section, we further investigate the performance of the model as $N$ tends to infinity.
It is clear from \eqref{eqn: epsilon_value} that $\epsilon$ tends to $0$ as $N$ grows sufficiently large.
In additional, Proposition \ref{prop: ambiguity_set_size} suggests that as $N$ grows, $\lambda = C/\sqrt{N}$ decreases, causing $\mathcal{D}_{\textup H}(\lambda)$ to shrink.
These trends indicate that the optimal value of model \eqref{model_harmonizing} may converge as $N$ grows sufficiently large.
To that end, we prove that the optimal value of model \eqref{model_harmonizing} converges to $V^*$ with probability (w.p.) 1 as $N$ tends to infinity, showcasing the asymptotic optimality of model \eqref{model_harmonizing}. 
More importantly, the corresponding error, e.g., the gap between the optimal value of model \eqref{model_harmonizing} and $V^*$, shrinks quickly in the $1/\sqrt{N}$-rate, achieving a good performance guarantee. 
Such a result provides one with confidence to use the HO approach for decision-making in practice, as the performance of HO improves with the duration of operations and the accumulation of more data samples.
We first present the asymptotic optimality of model \eqref{model_harmonizing} in the following proposition.

\begin{proposition} \label{prop: harmonizing_converge}
    If $F_N(\mathbf{x})$ converges to $F(\mathbf{x})$ w.p. 1 as $N \rightarrow \infty$, uniformly on $\mathcal{X}$, then $\Gamma(\lambda) \rightarrow V^*$ w.p. 1 as $N \rightarrow \infty$.
\end{proposition}


Next, we investigate the gap between the optimal value of model \eqref{model_harmonizing} and $V^*$.
Let $\mathcal{X}^*$ denote the set of optimal solutions of model \eqref{model0}. 
We then have the following proposition.

\begin{proposition} \label{prop: harmonizing_bias}
Assume there exists a measurable function $W: \mathcal{S} \rightarrow \mathbb{R}_+$ such that $\mathbb{E}[W(\boldsymbol{\xi})^2]$ is finite and $| f(\mathbf{x}, \boldsymbol{\xi}) - f(\mathbf{x}^{\prime}, \boldsymbol{\xi}) | \leq W(\boldsymbol{\xi}) \| \mathbf{x}-\mathbf{x}^{\prime} \|$ for any $\mathbf{x}, \mathbf{x}^{\prime} \in \mathcal{X}$ and a.e. $\boldsymbol{\xi} \in \mathcal{S}$. Then the following holds:
\begin{align}
    &\Gamma \left( \lambda \right) = \inf_{\mathbf{x} \in \mathcal{X}^*} F_{\lambda}(\mathbf{x}) + O \left( \frac{1}{\sqrt{N}} \right), \ \ \ \sqrt{N} \left( \Gamma \left( \lambda \right) - V^* \right) \xrightarrow{\mathcal{D}}  \inf_{\mathbf{x} \in \mathcal{X}^*} Y \left( \mathbf{x} \right), \label{eqn: asymptotic_order_2}
\end{align}
where $Y(\mathbf{x}) \sim \mathcal{N}(0, \sigma^2(\mathbf{x}) )$ for any $\mathbf{x} \in \mathcal{X}$.
Furthermore, if $\mathcal{X}^* = \{\mathbf{x}^*\}$ is a singleton, then
\vspace{-2mm}
\begin{align}
    \sqrt{N} \left( \Gamma \left( \lambda \right) - V^* \right) \xrightarrow{\mathcal{D}}  \mathcal{N} \left(0, \sigma^2\left( \mathbf{x}^* \right) \right). \label{eqn: asymptotic_order_3}
\end{align}
\end{proposition}

\vspace{-2mm}

By the second part of \eqref{eqn: asymptotic_order_2} in Proposition \ref{prop: harmonizing_bias} and Remark 57 in \cite{shapiro2021lectures}, we have $\sqrt{N} \mathbb{E}[\Gamma (\lambda) - V^*]$ tends to $\mathbb{E}[\inf_{\mathbf{x} \in \mathcal{X}^*} Y ( \mathbf{x} )]$ as $N \rightarrow \infty$; that is,
\vspace{-2mm}
\begin{align}
    \mathbb{E}[\Gamma(\lambda)] - V^* = \frac{1}{\sqrt{N}} \mathbb{E}\left[ \inf_{\mathbf{x} \in \mathcal{X}^*} Y \left( \mathbf{x} \right) \right] + o\left( \frac{1}{\sqrt{N}} \right), \label{eqn: optimal_value_gap}
\end{align}
where $ o (\cdot) $ refers to convergence to 0. Equation \eqref{eqn: optimal_value_gap} reflects the gap between the optimal value of model \eqref{model_harmonizing} and $V^*$, which diminishes as $N$ grows sufficiently large. 
Thus, given that we set the weight $\lambda$ in a $1/\sqrt{N}$-rate, i.e., $C / \sqrt{N}$, we can obtain a good performance guarantee for model \eqref{model_harmonizing}. 
The performance of model \eqref{model_harmonizing} is particularly significant when $\mathcal{X}^* = \{\mathbf{x}^*\}$ is a singleton, which leads to $\mathbb{E}[\inf_{\mathbf{x} \in \mathcal{X}^*} Y (\mathbf{x})] = \mathbb{E}[Y (\mathbf{x}^*)] = 0$ and $\mathbb{E}[\Gamma(\lambda)] - V^* = o( 1/\sqrt{N} )$. 
When $\mathcal{X}^*$ has more than one elements, $\inf_{\mathbf{x} \in \mathcal{X}^*} Y (\mathbf{x})$ may have a negative mean, i.e., $\mathbb{E}[\inf_{\mathbf{x} \in \mathcal{X}^*} Y (\mathbf{x})] < 0$.
Then, the gap, i.e., $\mathbb{E}[\Gamma(\lambda)] - V^*$, may be negative and in the order of $1/\sqrt{N}$, i.e., $O(1/\sqrt{N})$.

Model \eqref{model_harmonizing} achieves the above significant performance by integrating data with information and diminishing the influence of information while enlarging the significance of data as the sample size becomes large.
More importantly, we provide a \textit{simple yet effective} framework for integrating data and information while attaining a significant theoretical performance.
Such a framework is comparable to existing DRO frameworks with theoretical guarantees, such as the Wasserstein DRO with a 1 or 2-Wasserstein distance \citep{gao2023finite}.
Specifically, \cite{gao2023finite} shows that 1 or 2-Wasserstein DRO can achieve its performance guarantees by using the Wasserstein ball radius in a $1/\sqrt{N}$-rate (i.e., a rate similar to the weight $\lambda$ in this paper) to effectively avoid the curse of dimensionality.
To show the performance guarantees, \cite{gao2023finite} employs several advanced techniques, including Kantorovich’s duality, Markov’s inequality, and Young's inequality in several steps:
(i) the variation-based concentration holds if the true distribution satisfies the transportation-information inequality, by which performance guarantees for Wasserstein DRO can be proved for one loss function when the radius is in $1/\sqrt{N}$-rate;
(ii) leverages Local Rademacher Complexity Arguments to extend these results to encompass a wider range of loss functions. 
Clearly, the process of proving the performance guarantees for our proposed framework with the weight $\lambda$ set in a $1/\sqrt{N}$-rate is more straightforward to comprehend than that for Wasserstein DRO.

\vspace{-2mm}

\subsection{Parameter Estimation} \label{sec: estimate lambda}

In this section, we detail the estimation of $\lambda$, which plays a crucial role in HO.
Specifically, setting $\lambda = C / \sqrt{N}$ guarantees the asymptotic optimality of our HO model and its optimal value error of order $O(1/\sqrt{N})$. 
Choices of the constant $C$ do not affect the theoretical guarantees but may result in decisions with various performances in practice.
We propose three different methods of choosing $C$: (i) $K$-fold cross-validation, (ii) Tightening the confidence interval in Proposition \ref{prop: interval}, and (iii) Straightforward estimation.
Method (i) divides data samples into $K$ sets, using each set once to obtain optimal solutions with various $C$ candidates and the remaining sets to validate their performance, to identify the best candidate $C$.
Method (ii) identifies the best candidate $C$ that minimizes the confidence interval in Proposition \ref{prop: interval}, i.e., $z_{\alpha/2} \hat{\sigma}(\mathbf{x}) / \sqrt{N}$.
Method (iii) sets $C = \sqrt{M_0}$, where $M_0$ denotes the smallest number of samples we may have.
The details of each method are presented in Appendix \ref{sec: details of parameter estimation}.

As opposed to some existing DRO models (e.g., Wasserstein DRO), which require estimating the size parameter of the ambiguity set whenever $N$ samples change, our proposed HO model only requires estimating $C$ once, regardless of sample changes.
In particular, once we complete the estimation of $C$, we have $\lambda = C/\sqrt{N}$ for any $N$, by which we can apply the HO model directly for any sample size.
This highlights the significance of our proposed approach when solving the same problem multiple times with a varying number of given samples, which is common in real-world applications.
For example, consider a case where an operations manager is responsible for inventory management across thousands of convenience stores (e.g., 7-Eleven), which face uncertain demands.
The manager is tasked with solving the same stochastic newsvendor problem multiple times, one for each store.
Given the stores' diverse locations, the amount of historical demand samples varies from one store to another.
In this case, our proposed HO model can prove its specific advantage: we only need to estimate the size parameter $C$ once.
After this initial estimation, the model can be applied to efficiently address the stochastic inventory challenges for all stores.

\vspace{-2mm}

\subsection{Equivalent Reformulation} \label{sec: Equivalent Reformulation of HO}

First, to ensure the tractability of model \eqref{model_harmonizing}, we require the following common and practical conditions on the ambiguity set $\mathcal{D}$ and function $f(\mathbf{x}, \boldsymbol{\xi})$ \citep{wiesemann2014distributionally}.

\begin{enumerate}[label=(\roman*)]
\item The confidence set $\mathcal{C}_{I}$ is bounded and owns probability 1, i.e., $\underline{p}_I = \overline{p}_I = 1$.
This condition ensures that the confidence set with the largest index, i.e., $\mathcal{C}_I$, contains the support of $(\boldsymbol{\xi}, \mathbf{u})$. 

\item There exists a distribution $\mathbb{P} \in \mathcal{D}$ such that $\mathbb{P}((\boldsymbol{\xi}, \mathbf{u}) \in \mathcal{C}_i) \in (\underline{p}_i, \overline{p}_i)$, whenever $\underline{p}_i < \overline{p}_i$ for some $i \in [I]$.
This condition guarantees that there exists a distribution $\mathbb{P} \in \mathcal{D}$ satisfying the probability bounds as strict inequalities.

\item The function $f(\mathbf{x}, \boldsymbol{\xi})$ is piecewise linear convex in $\boldsymbol{\xi}$, i.e., $f(\mathbf{x}, \boldsymbol{\xi}) = \max_{k \in [K]}  f_k(\mathbf{x}, \boldsymbol{\xi}) = \max_{k \in [K]} 
\{ \alpha_k(\mathbf{x})^{\top} \boldsymbol{\xi} + \beta_k(\mathbf{x}) \} $ with both $\alpha_k: \mathbb{R}^n \rightarrow \mathbb{R}^m$ and $\beta_k: \mathbb{R}^n \rightarrow \mathbb{R}$ affine in $\mathbf{x}$ for any $k \in [K]$. 
This condition enables us to use robust optimization techniques to reformulate the semi-infinite constraints that arise from a dual reformulation of $\max_{\mathbb{P} \in \mathcal{D}} \ \mathbb{E}_{\mathbb{P}} [ f ( \boldsymbol{x},\boldsymbol{\xi} ) ]$.

\item For any $i, j \in [I], i \neq j$, we have either $\mathcal{C}_i \subsetneq \mathcal{C}_j$, $\mathcal{C}_j \subsetneq \mathcal{C}_i$, or $\mathcal{C}_i \cap \mathcal{C}_j = \emptyset$.
This condition implies a strict partial order on $\mathcal{C}_1, \ldots, \mathcal{C}_I$ in terms of the $\subsetneq$-relation.
This enables us to split the support of $(\boldsymbol{\xi}, \mathbf{u})$ into several disjoint and nonempty sets in the reformulation of $\max_{\mathbb{P} \in \mathcal{D}} \ \mathbb{E}_{\mathbb{P}} [ f ( \boldsymbol{x},\boldsymbol{\xi} ) ]$.
\end{enumerate}

\begin{theorem} \label{theorem: reformulation}
Assume conditions (i)--(iv) hold. Model \eqref{model_harmonizing} can be equivalently reformulated as
\vspace{-2mm}
\begin{align} \tag{$\textup{H}_1$} \label{model_harmonizing_reform}
    \min_{\mathbf{x} \in \mathcal{X}; \ \mathbf{w}, \boldsymbol{\pi}, \boldsymbol{\tau}, \boldsymbol{\kappa}, \boldsymbol{\theta}} \ \ \ & \left( 1-\lambda \right) \frac{1}{N} \sum_{j=1}^N w_j + \lambda \Big( \mathbf{b}^{\top} \boldsymbol{\pi} + \sum_{i \in [I]} \left( \overline{p}_i \kappa_i - \underline{p}_i \tau_i \right) \Big) \\
    {\normalfont \text{s.t.}} \ \ \ 
    & w_j \geq \alpha_k(\mathbf{x})^{\top} \tilde{\boldsymbol{\xi}}_j + \beta_k(\mathbf{x}), \ \forall \, j \in [N], \, k \in [K], \nonumber \\
    & \mathbf{c}_i^{\top} \boldsymbol{\theta}_{i,k} + \beta_k \left( \mathbf{x} \right) \leq \sum_{j \in \mathcal{A}_i} \left( \kappa_j - \tau_j \right), \ \forall \, i \in [I], k \in [K], \nonumber \\
    & \mathbf{C}_i^{\top} \boldsymbol{\theta}_{i,k} + \mathbf{A}^{\top} \boldsymbol{\pi} = \alpha_k \left( \mathbf{x} \right), \ \forall \, i \in [I], k \in [K], \nonumber \\
    & \mathbf{D}_i^{\top} \boldsymbol{\theta}_{i,k} + \mathbf{B}^{\top} \boldsymbol{\pi} = 0, \ \forall \, i \in [I], k \in [K], \nonumber \\
    & \boldsymbol{\pi} \in \mathbb{R}^m, \ \boldsymbol{\tau}, \ \boldsymbol{\kappa} \in \mathbb{R}_+^I; \ \ \boldsymbol{\theta}_{i,k} \in \mathcal{K}_i^*, \ \forall \, i \in [I], k \in [K]. \nonumber
\end{align}
\end{theorem}
\vspace{-2mm}
\begin{proof}{Proof.}
The result is deduced from Theorem 1 in \cite{wiesemann2014distributionally}.
\end{proof}

Model \eqref{model_harmonizing_reform} is a computationally tractable program for several ambiguity sets of practical interests, and we provide the details in Appendix \ref{sec: tractable reform}.

\section{Scenario Reduction} \label{sec: Scenario Reduction}

Propositions \ref{prop: converge} and \ref{prop: interval} indicate that the SAA model \eqref{model: SAA} performs notably well with a substantial number of samples.
However, such a large volume of samples makes the model hard to solve, posing significant challenges for making decisions under uncertainty in practice.
More generally, stochastic models with discrete distributions over a large volume of scenarios are hard to solve.
To address this computational challenge, scenario reduction emerges as an effective approach.
That is, for the model considering the $N$ given samples $\tilde{\boldsymbol{\xi}}_1, \ldots, \tilde{\boldsymbol{\xi}}_N$ in Section \ref{sec: Existing Models}, we identify $M < N$ samples from these $N$ samples, along with a corresponding probability distribution, to build an SAA model with these $M$ samples, 
while this small-sized model can generate an optimal value to closely approximate the SAA model \eqref{model: SAA} with the initial $N$ samples. 

With any $M \leq N$, we let $\mathcal{S}_0(M)$ denote the set that contains all subsets of $\{\tilde{\boldsymbol{\xi}}_1, \ldots, \tilde{\boldsymbol{\xi}}_N \}$, each with a size $M$, i.e., $\mathcal{S}_0(M) = \{ \tilde{\mathcal{S}} \subseteq \{\tilde{\boldsymbol{\xi}}_1, \ldots, \tilde{\boldsymbol{\xi}}_N \} \ | \ |\tilde{\mathcal{S}}| = M \}$.
The scenario reduction problem that helps approximate model \eqref{model: SAA} and reduce the number of scenarios from $N$ to $M$ can be formulated as
\begin{align}
    \min_{\tilde{\mathcal{S}} \in \mathcal{S}_0(M)} \min_{\mathbb{P} \in \mathcal{D}_0(\tilde{\mathcal{S}})} \left| \min_{\mathbf{x} \in \mathcal{X}} \mathbb{E}_{\mathbb{P}} \left[ f \left( \mathbf{x}, \boldsymbol{\xi} \right) \right] 
    - \min_{\mathbf{x} \in \mathcal{X}} \mathbb{E}_{\mathbb{P}_0} \left[ f \left( \mathbf{x}, \boldsymbol{\xi} \right) \right] 
    \right|. \label{model: scenario reduction}
\end{align}
Model \eqref{model: scenario reduction} identifies an optimal subset $\tilde{\mathcal{S}}^*$ with size $M$ and an optimal distribution on $\tilde{\mathcal{S}}^*$ to build the small-sized SAA model that yields the optimal value closest to the one obtained with $N$ samples.
Note that this model can be intractable, exhibiting a significant challenge to solve.
Nevertheless, we can quickly obtain high-quality feasible solutions by employing the HO method.

Specifically, we select $M$ scenarios randomly from $\{\tilde{\boldsymbol{\xi}}_1, \ldots, \tilde{\boldsymbol{\xi}}_N \}$, 
denoted by $\tilde{\mathcal{S}}^{\prime} = \{ \tilde{\boldsymbol{\zeta}}_j^{\prime}, \ j\in [M] \}$, 
and establish the empirical distribution on $\tilde{\mathcal{S}}^{\prime}$, 
denoted by $\tilde{\mathbb{P}}_0 (\tilde{\mathcal{S}}^{\prime})= \sum_{j \in [M]} \delta_{\tilde{\boldsymbol{\zeta}}_j^{\prime}} / M$.
Then, we use $\min_{\mathbf{x} \in \mathcal{X}} \mathbb{E}_{\tilde{\mathbb{P}}_0 (\tilde{\mathcal{S}}^{\prime})} [ f ( \mathbf{x}, \boldsymbol{\xi} ) ]$, 
which considers $M$ scenarios, 
to approximate $\min_{\mathbf{x} \in \mathcal{X}} \mathbb{E}_{\mathbb{P}_0} [ f ( \mathbf{x}, \boldsymbol{\xi} ) ]$, 
which considers $N$ scenarios.
While obtaining $\tilde{\mathbb{P}}_0 (\tilde{\mathcal{S}}^{\prime})$ is straightforward, it may not yield a satisfactory approximation because it may fail to leverage certain information contained in the initial $N$ scenarios. 
To enhance the approximation quality, we resort to our proposed HO framework, which helps incorporate certain distributional information (i.e., $\mathcal{D}$ in model \eqref{model_harmonizing}), 
highlighting the effectiveness of our HO method in scenario reduction. 

We first establish $\tilde{\mathbb{P}}_0 (\tilde{\mathcal{S}}^{\prime})$ and extract the partial distributional information from the $N$ samples to construct the ambiguity set $\mathcal{D}$. 
Then, we use the following HO model
\begin{align}
\min_{\mathbf{x} \in \mathcal{X}} \left\{ \left( 1 - \lambda \right) \mathbb{E}_{\tilde{\mathbb{P}}_0 ( \tilde{\mathcal{S}}^{\prime} ) } \left[ f \left( \mathbf{x}, \boldsymbol{\xi} \right) \right] + \lambda \max_{\mathbb{P} \in \mathcal{D}} \mathbb{E}_{\mathbb{P}}\left[ f \left( \mathbf{x}, \boldsymbol{\xi} \right) \right] \right\} \label{hro_scenario_reduction} 
\end{align}
to approximate the original SAA model $\min_{\mathbf{x} \in \mathcal{X}} \mathbb{E}_{\mathbb{P}_0} \left[ f \left( \mathbf{x}, \boldsymbol{\xi} \right) \right]$, where we set $\lambda = 1 - \sqrt{M} / \sqrt{N}$.
Note that when $M=N$, i.e., no scenario reduction, the above HO model \eqref{hro_scenario_reduction} recovers the original SAA model \eqref{model: SAA}.
As $M$ decreases, which implies fewer scenarios are considered, $\lambda$ correspondingly increases. 
This results in a larger-size ambiguity set $\mathcal{D}_{\textup H}(\lambda)$, as suggested by Proposition \ref{prop: ambiguity_set_size}. 
The expansion of $\mathcal{D}_{\textup H}(\lambda)$ ensures that the obtained solution can effectively hedge against the increased uncertainty induced by scenario reduction, thereby maintaining the solution's quality. 
Note that this method of estimating $\lambda$ is different from those introduced in Section \ref{sec: estimate lambda} and we name it as the estimation method (iv) \textbf{scenario reduction estimation}.

Different from existing scenario reduction approaches, such as the approach in \cite{rujeerapaiboon2022scenario} that needs to evaluate the model's performance with respect to each of $N$ scenarios iteratively, our proposed HO uses the partial distributional information from the $N$ scenarios, thereby maintaining its efficiency even when $N$ is very large. 
Specifically, an existing approach considers the initial $N$ samples $\{\tilde{\boldsymbol{\xi}}_1, \ldots, \tilde{\boldsymbol{\xi}}_N \}$ with its corresponding distribution $\mathbb{P}_{\textup N} = \sum_{i \in [N]} \eta_i \delta_{\tilde{\boldsymbol{\xi}}_i}$, where $\eta_i \in [0,1]$ for any $i \in [N]$ represents the probability of the $i$-th sample, and aims to identify a subset of samples with a distribution closest to $\mathbb{P}_{\textup N}$ \citep{dupavcova2003scenario, rujeerapaiboon2022scenario}.
For example, \cite{rujeerapaiboon2022scenario} perform scenario reduction by identifying a subset $\{\tilde{\boldsymbol{\zeta}}_j, \ j\in [M]\} \subseteq \{\tilde{\boldsymbol{\xi}}_i, \ i\in [N]\}$ that has a distribution $\mathbb{Q}^* = \sum_{j \in [M]} \omega_j \delta_{\tilde{\boldsymbol{\zeta}}_j}$ closest to $\mathbb{P}_{\textup N}$, in terms of type-$l$ Wasserstein distance.
Here $\omega_j \in [0,1]$ for any $j \in [M]$ stands for the probability of the $j$-th sample. 
The type-$l$ Wasserstein distance between $\mathbb{Q}^*$ and $\mathbb{P}_{\textup N}$ is calculated by
\vspace{-2mm}
\begin{align*}
    d_l(\mathbb{P}_{\textup N}, \mathbb{Q}^*) = \left\{ \min_{\boldsymbol{\gamma} \in \mathbb{R}_+^{N \times M}} 
    \left\{ 
    \sum_{i \in [N]} \sum_{j \in [M]} \gamma_{i,j} \| \tilde{\boldsymbol{\xi}}_i - \tilde{\boldsymbol{\zeta}}_j \|^l \  \middle|   \
    \sum_{j \in [M]} \gamma_{i,j} = \eta_i, \forall i \in [N], 
    \sum_{i \in [N]} \gamma_{i,j} = \omega_j, \forall j \in [M]
    \right\}
    \right\}^{\frac{1}{l}}.
\end{align*}
They further solve the following problem to obtain $\mathbb{Q}^*$:
\vspace{-2mm}
\begin{align}
    G_l \left( \mathbb{P}_{\textup N}, M \right) = \min _{\mathbb{Q}} \left\{ d_l \left( \mathbb{P}_{\textup N}, \mathbb{Q} \right) \ \middle| \ \mathbb{Q} \in \mathcal{D}_0(\tilde{\mathcal{S}}), \ \tilde{\mathcal{S}} \in \mathcal{S}_0(M)
    \right\}. \label{scenario_reduction_problem}
\end{align}

\vspace{-2mm}

To solve problem \eqref{scenario_reduction_problem}  efficiently, \cite{rujeerapaiboon2022scenario} propose a polynomial-time constant-factor approximation algorithm based on a local search algorithm in \cite{arya2004local} (see Algorithm \ref{alg:local search} in Appendix \ref{app: lot sizing}).
While this algorithm serves as an approximation technique to determine an upper bound (denoted by $\overline{G}_l (\mathbb{P}_{\textup N}, M)$) for $G_l (\mathbb{P}_{\textup N}, M)$, it can attain a satisfactory bound for $\overline{G}_l (\mathbb{P}_{\textup N}, M) / G_l (\mathbb{P}_{\textup N}, M)$. 
However, this algorithm needs to evaluate the model's performance with respect to each of $N$ scenarios in each iteration, resulting in an obvious computational time, especially when $N$ is large. 
Moreover, even if we can obtain $\mathbb{Q}^*$ successfully, it may not yield the optimal value that is closest to the one obtained under $\mathbb{P}_{\textup N}$. 
In contrast, our HO framework in scenario reduction can maintain its efficiency for any size of $N$.
HO can achieve strong performance with only a few scenarios by retaining information about the dropped ones. Even when $N$ is very large, making the original problem extremely difficult to solve, HO requires only a few scenarios while maintaining effectiveness. Consequently, when $N$ is large, our method's advantages are more pronounced.
Meanwhile, both $\tilde{\mathbb{P}}_0 (\tilde{\mathcal{S}}^{\prime})$ and partial distributional information can be quickly and easily identified from the initial $N$ scenarios, helping us to establish an optimization-based approach that incorporates the knowledge of the problem.
Such results highlight the significance of our HO framework in helping decision-makers reduce the number of scenarios to consider, thereby simplifying decision-making under uncertainty and ensuring high-quality solutions.
We demonstrate these advantages with numerical results in Section \ref{sec: lot sizing}.
Establishing this practical significance of HO also differentiates our study from \cite{tsang2025tradeoff}.

\section{Numerical Experiments} \label{sec: Numerical Experiments}

We conduct numerical experiments to provide insights into the performance of our proposed model \eqref{model_harmonizing}. 
The model is implemented in MATLAB R2023a by the modeling language CVX with the Mosek solver on a PC with an Intel(R) Core(TM) i9-13900K @ 3.00 GHz processor.
We apply our methodologies, including model \eqref{model_harmonizing} and parameter estimation methods (i)--(iv), to two industrial applications: mean-risk portfolio optimization and lot sizing on a network.
We examine the significance of HO by comparing its out-of-sample performance against the performance of other solution approaches.
In our experiments, we evaluate the out-of-sample performance of the solution obtained by any approach using $10^6$ test samples, which are separate from the $N$ training samples used to compute the solution.
Parameter settings are detailed in Appendix \ref{app: Parameter Settings in Numerical Experiments}.

\vspace{-2mm}

\subsection{Mean-risk Portfolio Optimization} \label{sec: Portfolio Optimization}



Consider a capital market consisting of $m$ assets whose returns are captured by random parameters $\boldsymbol{\xi} = (\xi_1, \ldots, \xi_m)^{\top} \in \mathbb{R}^m$. 
With a fixed capital, one invests a percentage $x_i$ in the $i$-th asset, leading to a portfolio investment decision $\mathbf{x} = (x_1, \ldots, x_m)^{\top} \in \mathbb{R}^m$. 
We formulate the HO counterpart of the mean-risk portfolio optimization problem as
\vspace{-2mm}
\begin{align} \label{model: portfolio}
\min_{\mathbf{x} \in \mathcal{X}} \max_{\mathbb{P} \in \mathcal{D}_{\textup H}\left(\lambda\right)} 
\left\{
\mathbb{E}_{\mathbb{P}} \left[ -\mathbf{x}^{\top} \boldsymbol{\xi} \right] + \rho \mathbb{P} \text{-CVaR}_{a}\left( -\mathbf{x}^{\top} \boldsymbol{\xi} \right)
\right\},
\end{align}
\vspace{-6mm}

\noindent
where $\mathcal{X} = \{ \mathbf{x} \in \mathbb{R}_+^m \ | \ \mathbf{1}^{\top} \mathbf{x} = 1 \}$, $\rho \in \mathbb{R}_+$ reflects the decision maker’s risk-aversion preference, and $\mathbb{P} \text{-CVaR}_{a}( -\mathbf{x}^{\top} \boldsymbol{\xi} )$ quantifies conditional value-at-risk, i.e., the average of the $a \times 100\%$ worst portfolio losses under the distribution $\mathbb{P}$ \citep{rockafellar2000optimization}. 
Similarly, we can formulate the Wasserstein-based DRO counterpart of this problem \citep{esfahani2018data}.

Following similar steps as in \cite{esfahani2018data}, we can replace the CVaR in \eqref{model: portfolio} with its formal definition and further rewrite \eqref{model: portfolio} as 
\vspace{-2mm}
\begin{align} \label{model: portfolio_2}
\min_{\mathbf{x} \in \mathcal{X}, \tau \in \mathbb{R}} \max_{\mathbb{P} \in \mathcal{D}_{\textup H}\left(\lambda\right)} \mathbb{E}_{\mathbb{P}} \left[ \max_{k \leq K} 
\left\{ \alpha_k \mathbf{x}^{\top} \boldsymbol{\xi} + \beta_k \tau \right\}
\right],
\end{align}
\vspace{-8mm}

\noindent
where $K=2$, $\alpha_1 = -1$, $\alpha_2 = -1 - \rho/a$, $\beta_1 = \rho$, and $\beta_2 = \rho (1 - 1/a)$. 
We can then reformulate \eqref{model: portfolio_2} to a computationally tractable form by Theorem \ref{theorem: reformulation}, 
which can be applied here because $\max_{k \leq K} \{ \alpha_k \mathbf{x}^{\top} \boldsymbol{\xi} + \beta_k \tau \}$ is piecewise affine convex in $\boldsymbol{\xi}$ (see details in Appendix \ref{app: setup portfolio}).

We compare out-of-sample performances of our model \eqref{model_harmonizing} with two Wasserstein-based DRO models: (i) the model in \cite{esfahani2018data} that uses only data samples (denoted by ``Wasserstein'') and (ii) the model in \cite{gao2017distributionally} that uses both data samples and moment information (denoted by ``W+M'').
Specifically, ``Wasserstein'' and ``W+M'' incorporate their ambiguity sets $\mathcal{D}_{\textup W}$ and $\mathcal{D}_{\textup C}$, respectively, as follows:
\vspace{-3mm}
\begin{align*}
& \mathcal{D}_{\textup W} = \left\{\mathbb{P} \ \middle| \ W \left( \mathbb{P}, \mathbb{P}_0 \right) \leq r_{\textup W} \right\}, \\
& \mathcal{D}_{\textup C} = \left\{\mathbb{P} \ \middle|  \  \left( \mathbb{E}_{\mathbb{P}}[ \boldsymbol{\xi} ] -\boldsymbol{\mu} \right)^{\top} \boldsymbol{\Sigma}^{-1} \left( \mathbb{E}_{\mathbb{P}}[ \boldsymbol{\xi} ] - \boldsymbol{\mu} \right) \leq \gamma_1, \mathbb{E}_{\mathbb{P}} [ \left( \boldsymbol{\xi} - \boldsymbol{\mu} \right)  \left( \boldsymbol{\xi} - \boldsymbol{\mu} \right)^{\top} ] \preceq \gamma_2\boldsymbol{\Sigma}, \ W \left( \mathbb{P}, \mathbb{P}_0 \right) \leq r_{\textup C} \right\},
\end{align*}
\vspace{-9mm}

\noindent
where $W: \mathcal{D}_0 (\mathbb{R}^m) \times \mathcal{D}_0 (\mathbb{R}^m) \rightarrow \mathbb{R}_+$ denotes the Wasserstein metric.
Clearly, $\mathcal{D}_{\textup C}$ is the intersection of $\mathcal{D}_{\textup W}$ and the moment-based ambiguity set  $\mathcal{D}_{\textup M} = 
\{\mathbb{P} \ | \  ( \mathbb{E}_{\mathbb{P}}[ \boldsymbol{\xi} ] -\boldsymbol{\mu} )^{\top} \boldsymbol{\Sigma}^{-1} (\mathbb{E}_{\mathbb{P}}[ \boldsymbol{\xi} ] - \boldsymbol{\mu}) \leq \gamma_1, \ \mathbb{E}_{\mathbb{P}} [ ( \boldsymbol{\xi} - \boldsymbol{\mu} ) ( \boldsymbol{\xi} - \boldsymbol{\mu} )^{\top} ] \preceq \gamma_2\boldsymbol{\Sigma} \}$, i.e., $\mathcal{D}_{\textup C} = \mathcal{D}_{\textup W} \cap \mathcal{D}_{\textup M}$.

To ensure a fair comparison, we keep the same parameter settings as in \cite{esfahani2018data}.
We determine the Wasserstein radii: $r_{\textup W}$ for ``Wasserstein'' and  $r_{\textup C}$ for ``W+M,'' using the same approach of $K$-fold cross-validation as described in \cite{esfahani2018data}.
We assess methods (i)--(iii) of estimating $C$ as introduced in Section \ref{sec: estimate lambda},
and assess model \eqref{model_harmonizing} with $\mathcal{D}$ being $\mathcal{D}_{\textup D}$, constructed based on a given support $\mathcal{S} \subseteq \mathbb{R}^m$, mean $\boldsymbol{\mu} \in \mathbb{R}^m$ and deviation $\boldsymbol{\delta} \in \mathbb{R}^m$:
\vspace{-2mm}
\begin{align*}
\mathcal{D}_{\textup D} = \left\{ \mathbb{P} \ | \  \mathbb{E}_{\mathbb{P}} \left[ \boldsymbol{\xi} \right] = \boldsymbol{\mu}, \  \mathbb{E}_{\mathbb{P}} \left[ \mathbf{u} \right] = \boldsymbol{\delta}, \ \mathbb{P} \left(
\boldsymbol{\xi} \in \mathcal{S}
\right) = 1, \ \mathbb{P} \left(
   \mathbf{u} \geq \boldsymbol{\xi} - \boldsymbol{\mu}, \ \mathbf{u} \geq \boldsymbol{\mu} - \boldsymbol{\xi}
    \right)
= 1 \right\}.
\end{align*}
\vspace{-9mm}

Figure \ref{fig:out_portfolio_mad} shows the performance of model \eqref{model_harmonizing} when the number of samples, i.e., $N$, varies.
Specifically, we vary $N \in \{25, 50, 75, 100, 150, 200, 300, 400, 500\}$, and accordingly $M_0 = 25$.
For each instance, we perform 200 independent runs and report the average result.
We use ``MAD'' to denote model \eqref{model_harmonizing} with $\mathcal{D}_{\textup D}$, and use ``Cross,'' ``Gap,'' and ``$\sqrt{M_0}$'' to denote estimation methods (i), (ii), and (iii), respectively.
For instance, ``MAD\_Gap'' in Figure \ref{fig:out_portfolio_mad} indicates model \eqref{model_harmonizing} with an ambiguity set $\mathcal{D}_{\textup D}$, where we use method (ii) to estimate the value of $C$.
We use the true information of $\boldsymbol{\mu}$ and $\boldsymbol{\Sigma}$ to construct $\mathcal{D}_{\textup D}$.
Besides, we only estimate $C$ when $N = 25$, irrespective of the estimation methods used. 
Once $C$ is determined, we calculate $\lambda$ as $C/\sqrt{N}$ when $N$ varies.

\vspace{-4mm}

\begin{figure}[!htb]
    \begin{minipage}{0.4\textwidth}
     \centering
    \includegraphics[scale=0.4]{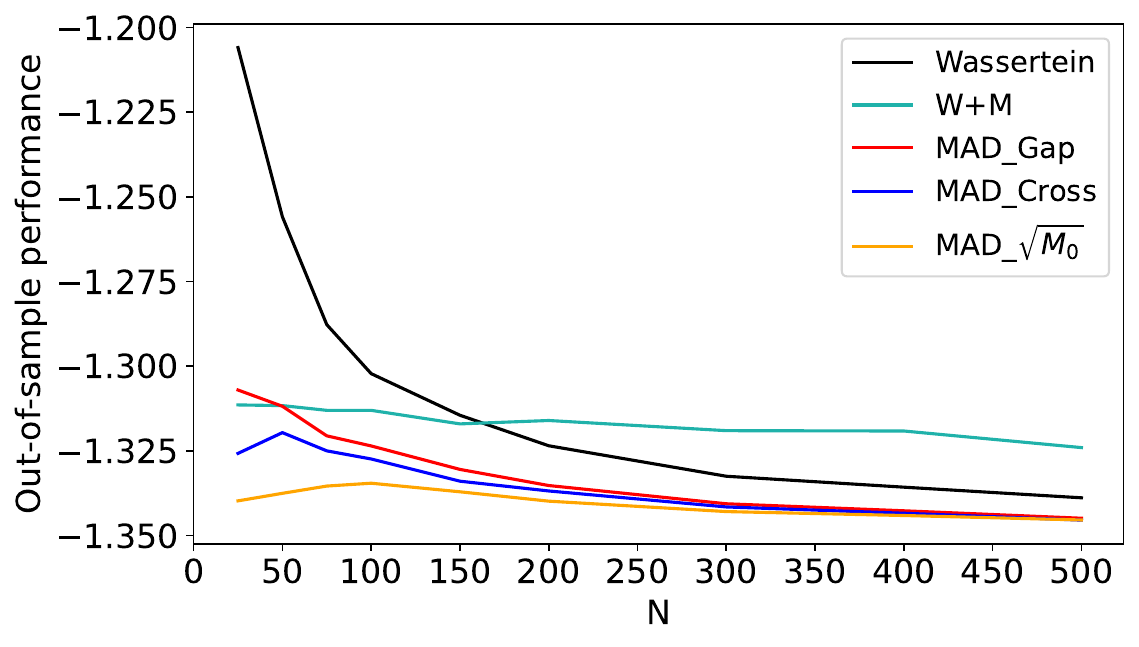}
    \caption{\centering Out-of-sample Performance of Model \eqref{model_harmonizing} with $\mathcal{D}_{\textup D}$}\label{fig:out_portfolio_mad}
    \end{minipage}
    \hfill
    \begin{minipage}{0.5\textwidth}
     \centering
    \includegraphics[scale=0.4]{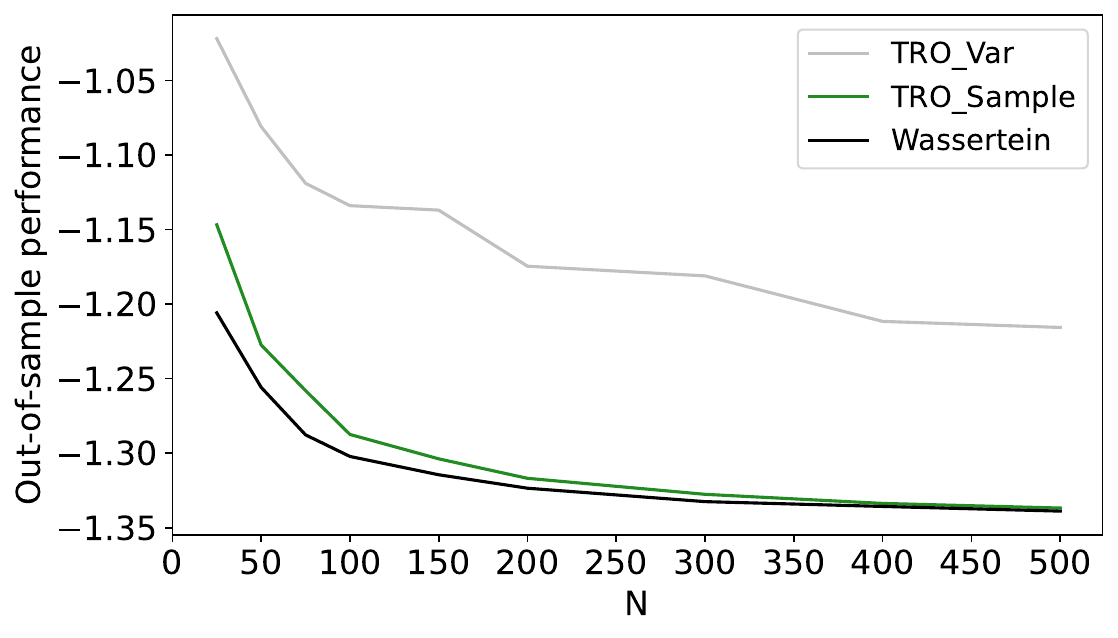}
    \vspace{0mm}
    \caption{\centering Out-of-sample Performance of TRO Model in \cite{tsang2025tradeoff}}\label{fig:out_portfolio_tsang}
  \end{minipage}
\end{figure}

\vspace{-9mm}

First, we compare model \eqref{model_harmonizing} with the ``Wasserstein'' model, i.e., the Wasserstein-based DRO model using only the empirical distribution derived from data samples.
Clearly, model \eqref{model_harmonizing} consistently outperforms ``Wasserstein,'' regardless of the value of $N$ and the methods used to estimate $C$.
The advantages of model \eqref{model_harmonizing} are more pronounced when $N$ is smaller, in terms of the out-of-sample results.
As $N$ increases, the out-of-sample results of all models turn to converge, affirming the asymptotic consistency across these models.
We also observe consistently reliable results obtained from different methods of estimating $C$, demonstrating their effectiveness.
Note that the ``Wasserstein'' model utilizes the empirical distribution derived from data only. 
When $N$ is small, i.e., data is limited, the empirical distribution may largely deviate from the true distribution, potentially leading to an unreliable or excessively large ambiguity set.
Specifically, if the Wasserstein radius is small, it may generate an ambiguity set where distributions are close to the empirical distribution but far from the true distribution. Conversely, a large radius may result in an excessively large ambiguity set, leading to an overly conservative solution.
Our HO approach overcomes these drawbacks and demonstrates superior performance by integrating data and information and adjusting their weights adaptively based on data size.
When data is limited, its weight becomes small, and its impact is mitigated, thereby reducing the effect of data scarcity.
Meanwhile, the weight of information becomes large, and its influence is amplified, guiding the model to capture the true distribution. 
More importantly, since the weight of information $\lambda$ is adaptively adjusted as the amount of data 
$N$ changes, we can confidently apply our HO approach without concern for whether the data is limited or sufficient.

Second, we compare model \eqref{model_harmonizing} with ``W+M'' model, i.e., Wasserstein-based DRO model using both data and information.
Model \eqref{model_harmonizing} demonstrates comparable out-of-sample performance to ``W+M'' when data is limited (e.g., $N \leq 50$).
These two models overcome the drawback of data scarcity and demonstrate superior performance when data is limited because they use moment information, which is particularly beneficial in such situations.
However, as $N$ increases, model \eqref{model_harmonizing} exhibits superior out-of-sample performance compared to ``W+M,'' with its advantages becoming more pronounced as $N$ grows.
This indicates that, despite both models using the same data, model \eqref{model_harmonizing} owns a 
stronger ability to leverage data to enhance solution quality than ``W+M.''
These results confirm the effectiveness of our methods for determining $C$ and $\lambda$, which shape the ability of model \eqref{model_harmonizing} to leverage data.
They also imply that the radius $r_{\textup C}$ obtained by cross-validation may not be ideal, limiting ``W+M'' model's ability to leverage data effectively.
Comparatively, cross-validation can achieve a better radius $r_{\textup W}$ for ``Wasserstein'' model, as evidenced by its superior out-of-sample performance over ``W+M'' when $N$ is large.
Note that this does not imply that ``Wasserstein'' outperforms ``W+M,'' because they use different radii.
For example, when $N=500$, the best-estimated radius is $r_{\textup W}=0.01$ for ``Wasserstein'' but $r_{\textup C}=0.09$ for ``W+M.''
However, a radius $r_{\textup C} = 0.01$ leads to $\mathcal{D}_{\textup C} = \emptyset$ for ``W+M.''
This occurs because the empirical distribution $\mathbb{P}_0$ derived from $N=500$ samples does not satisfy the moment conditions, i.e., $\mathbb{P}_0 \notin \mathcal{D}_{\textup M}$, and the radius $r_{\textup C}=0.01$ is too small for $\mathcal{D}_{\textup W}$ to intersect with $\mathcal{D}_{\textup M}$, resulting $\mathcal{D}_{\textup C} = \mathcal{D}_{\textup M} \cap \mathcal{D}_{\textup W} = \emptyset$.
We also check that ``W+M'' exhibits better out-of-sample performance than ``Wasserstein'' when $r_{\textup W} = r_{\textup C} = 0.09$.
Moreover, when data is limited, ``W+M'' exhibits better out-of-sample performance than ``Wasserstein,'' aligning with the findings in \cite{gao2017distributionally}.

Figure \ref{fig:out_portfolio_tsang} shows the performance of the TRO model in \cite{tsang2025tradeoff}, with the Wasserstein-based DRO model serving as the benchmark.
We test the TRO model with two types of ambiguity sets: a mean-variance ambiguity set, referred to as ``TRO\_Sample,'' and a $\phi$-divergence ball based on total variation distance, referred to as ``TRO\_Var.''
We use the same parameter settings for the ambiguity set in the TRO model as those in \cite{tsang2025tradeoff}, and determine the weight parameter using the same cross-validation approach as described therein.
By their settings, the mean and variance used in ``TRO\_Sample" are obtained from the $N$ samples.
Therefore, these TRO models (``TRO\_Sample" and ``TRO\_Var") do not incorporate partial distributional information.
Figure \ref{fig:out_portfolio_tsang} shows that ``Wasserstein'' consistently outperforms the TRO model, regardless of the ambiguity set adopted in the TRO model or the value of $N$.
This indicates that the proposed ``TRO\_Sample" and ``TRO\_Var" in \cite{tsang2025tradeoff} are less effective than ``Wasserstein," highlighting the need to carefully choose the ambiguity set. A poor choice may yield worse results than simply using the Wasserstein-based DRO.

\vspace{-6mm}

\begin{table}[htbp]
    \hspace{-1cm}
    \caption{Time (s) of Model \eqref{model_harmonizing} with $\mathcal{D}_{\textup D}$} 
    \centering
    \scriptsize
    \scalebox{1}{
    \begin{tabular}{ccccccccccccccc}
    \toprule
    \multirow{2}[1]{*}{$N$} & \multicolumn{2}{c}{\textbf{MAD\_Gap}} &       & \multicolumn{2}{c}{\textbf{MAD\_Cross}} &       & \multicolumn{2}{c}{\textbf{MAD\_$\sqrt{M_0}$}} &       & \multicolumn{2}{c}{\textbf{Wasserstein}} &       & \multicolumn{2}{c}{\textbf{W+M}} \\
\cmidrule{2-3}\cmidrule{5-6}\cmidrule{8-9}\cmidrule{11-12}\cmidrule{14-15}          & PREP   & COMP  &       & PREP   & COMP  &       & PREP   & COMP  &       & PREP   & COMP  &       & PREP   & COMP \\
    \midrule
    25    & 13.52 & 0.51  &       & 12.32 & 0.52  &       & 0     & 0.55  &       & 66.50  & 0.50  &       & 507.46  & 4.78  \\
    50    & 0     & 0.60  &       & 0     & 0.64  &       & 0     & 0.66  &       & 79.43 & 0.62  &       & 1,061.87  & 9.79  \\
    75    & 0     & 0.72  &       & 0     & 0.75  &       & 0     & 0.81  &       & 93.03 & 0.74  &       & 1,566.39  & 14.68  \\
    100   & 0     & 0.85  &       & 0     & 0.84  &       & 0     & 0.94  &       & 111.73 & 0.90  &       & 2,116.07  & 20.15  \\
    150   & 0     & 1.19  &       & 0     & 1.28  &       & 0     & 1.24  &       & 138.99 & 1.23  &       & 3,353.33  & 32.01  \\
    200   & 0     & 1.55  &       & 0     & 1.58  &       & 0     & 1.44  &       & 172.62 & 1.47  &       & 4,494.62  & 49.60  \\
    300   & 0     & 2.26  &       & 0     & 2.32  &       & 0     & 2.14  &       & 251.85 & 2.21  &       & 8,876.05  & 90.69  \\
    400   & 0     & 2.92  &       & 0     & 3.08  &       & 0     & 3.00  &       & 319.88 & 2.76  &       & 13,114.74  & 136.26  \\
    500   & 0     & 3.60  &       & 0     & 3.61  &       & 0     & 3.51  &       & 401.99 & 3.21  &       & 16,748.81  & 166.65  \\
    \midrule
    Average & 1.50  & 1.58  &       & 1.37  & 1.62  &       & 0     & 1.59  &       & 181.78  & 1.52  &       & 5,759.93  & 58.29  \\
    \bottomrule
    \end{tabular}
	}
	\label{tab:time_portfolio_mad}
\end{table}

\vspace{-3mm}

In addition, Table \ref{tab:time_portfolio_mad} presents the preparation time (column ``PREP'') that each model takes for parameter estimation, as well as the computational time (column ``COMP'') needed for the solving process.
Specifically, the preparation time of model \eqref{model_harmonizing} refers to the time for estimating $C$, while the preparation time of both ``Wasserstein'' and ``W+M'' refers to the time for estimating their radii.
Since $C$ is only estimated once when $N=25$ and we set $\lambda = C/\sqrt{N}$ as $N$ increases, model \eqref{model_harmonizing} can be applied directly when $N > 25$, leading to a preparation time of $0$ for $N > 25$.
Clearly, model \eqref{model_harmonizing} requires significantly less preparation time than ``Wasserstein'' and ``W+M'' models, regardless of the methods used to estimate $C$.
In terms of the computational time, model \eqref{model_harmonizing} is comparable to ``Wasserstein,''  whereas ``W+M'' requires significantly more time.
Alongside Figure \ref{fig:out_portfolio_mad}, it is evident that \textit{an appropriate value of $\lambda$ in model \eqref{model_harmonizing} can be easily and rapidly estimated, enabling the model to quickly obtain a solution with strong out-of-sample performance for any data size}.

We further examine the performance of model \eqref{model_harmonizing} with $\mathcal{D}$ being $\mathcal{D}_{\textup T}$, which exhibits a trend similar to that observed when $\mathcal{D}$ being $\mathcal{D}_{\textup D}$, as shown in Figure \ref{fig:out_portfolio_moment} and Table \ref{tab:time_portfolio_moment}.

\vspace{-5mm}

\begin{figure}[!htb]
\centering
  \includegraphics[scale=0.4]{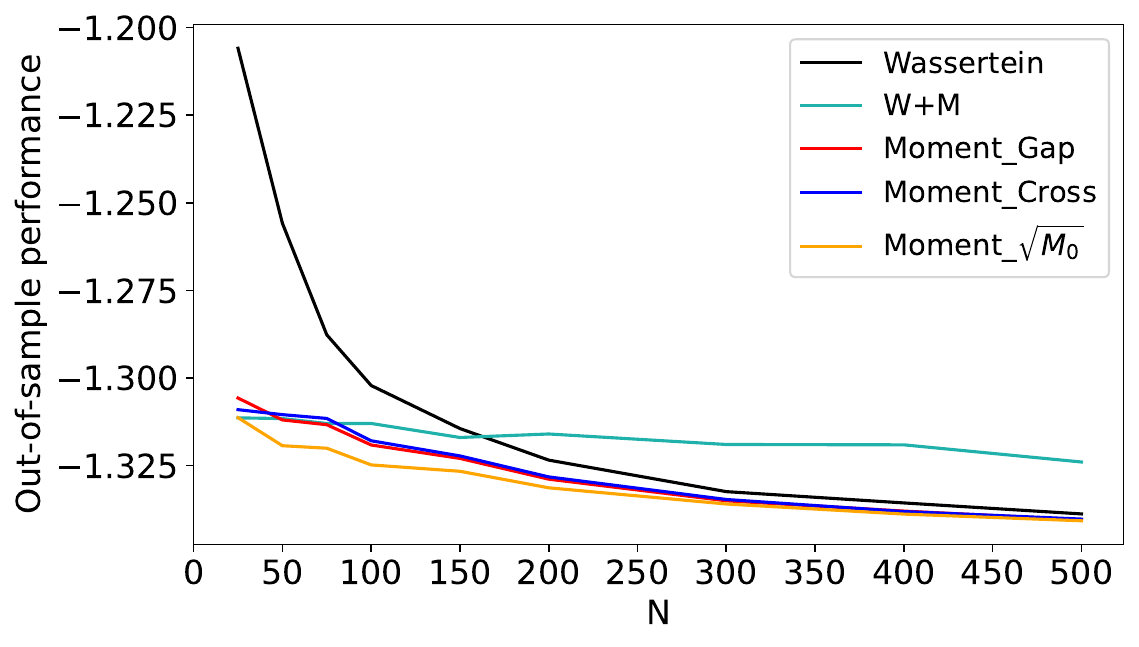}
  \vspace{-2mm}
\caption{\centering Out-of-sample Performance of Model \eqref{model_harmonizing} with $\mathcal{D}_{\textup T}$}
\label{fig:out_portfolio_moment}
\end{figure}

\vspace{-12mm}

\begin{table}[htbp]
    \hspace{-1cm}
    \caption{Time (s) of Model \eqref{model_harmonizing} with $\mathcal{D}_{\textup T}$} 
    \centering
    \scriptsize
    \scalebox{1}{
    \begin{tabular}{ccccccccccccccc}
    \toprule
    \multirow{2}[1]{*}{$N$} & \multicolumn{2}{c}{\textbf{Moment\_Gap}} &       & \multicolumn{2}{c}{\textbf{Moment\_Cross}} &       & \multicolumn{2}{c}{\textbf{Moment\_$\sqrt{M_0}$}} &       & \multicolumn{2}{c}{\textbf{Wasserstein}} &       & \multicolumn{2}{c}{\textbf{W+M}} \\
\cmidrule{2-3}\cmidrule{5-6}\cmidrule{8-9}\cmidrule{11-12}\cmidrule{14-15}          & PREP & COMP &       & PREP & COMP &       & PREP & COMP &       & PREP & COMP &       & PREP & COMP \\
    \midrule
    25    & 11.07 & 0.29  &       & 5.89  & 0.26  &       & 0     & 0.24  &       & 66.50  & 0.50  &       & 507.46  & 4.78  \\
    50    & 0     & 0.37  &       & 0     & 0.38  &       & 0     & 0.37  &       & 79.43 & 0.62  &       & 1,061.87  & 9.79  \\
    75    & 0     & 0.52  &       & 0     & 0.51  &       & 0     & 0.51  &       & 93.03 & 0.74  &       & 1,566.39  & 14.68  \\
    100   & 0     & 0.65  &       & 0     & 0.63  &       & 0     & 0.65  &       & 111.73 & 0.90  &       & 2,116.07  & 20.15  \\
    150   & 0     & 0.96  &       & 0     & 0.94  &       & 0     & 0.94  &       & 138.99 & 1.23  &       & 3,353.33  & 32.01  \\
    200   & 0     & 1.28  &       & 0     & 1.35  &       & 0     & 1.23  &       & 172.62 & 1.47  &       & 4,494.62  & 49.60  \\
    300   & 0     & 2.00  &       & 0     & 1.96  &       & 0     & 1.90  &       & 251.85 & 2.21  &       & 8,876.05  & 90.69  \\
    400   & 0     & 2.70  &       & 0     & 2.75  &       & 0     & 2.50  &       & 319.88 & 2.76  &       & 13,114.74  & 136.26  \\
    500   & 0     & 3.49  &       & 0     & 3.37  &       & 0     & 3.18  &       & 401.99 & 3.21  &       & 16,748.81  & 166.65  \\
    \midrule
    Average & 1.23  & 1.36  &       & 0.65  & 1.35  &       & 0     & 1.28  &       & 181.78  & 1.52  &       & 5,759.93  & 58.29  \\
    \bottomrule
    \end{tabular}%
	}
	\label{tab:time_portfolio_moment}
\end{table}

\vspace{-8mm}

\subsection{Lot Sizing on a Network} \label{sec: lot sizing}

Lot sizing is one of the most significant and difficult problems in production planning \citep{bertsimas2016duality, long2023supermodularity}.
It focuses on a network with a total of $m$ stores, with each store $i \in [m]$ facing a random demand $\xi_i$.
In the first stage where the uncertain demands $\boldsymbol{\xi}$ are not realized yet, we determine a positive allocation $x_i$ for each store $i \in [m]$, which is limited by an upper bound $K_i$.
The unit storage cost for the allocation at store $i \in [m]$ is $a_i$.
In the second stage, after realizing $\boldsymbol{\xi}$, we transport stock $y_{i,j}$ from store $i \in [m]$ to $j \in [m]$ at a unit cost $b_{i,j}$, and the transport amount is bounded by $Y_{i,j}$.
The demand shortage at any store $i \in [m]$, denoted by $z_i$, incurs a penalty of $c_i z_i$, where $c_i$ is the unit penalty at store $i$.
We formulate the HO counterpart of the lot sizing problem as
\begin{align} \label{model: lot sizing}
    \min_{\mathbf{x}} \ \left\{ \mathbf{a}^{\top} \mathbf{x} + \max_{\mathbb{P} \in \mathcal{D}_{\textup H}\left(\lambda\right) } \mathbb{E}_{\mathbb{P}} \left[ f \left( \mathbf{x}, \boldsymbol{\xi} \right)  \right] \ | \
    0 \leq x_i \leq K_i, \ \forall \, i \in [m] \right\},
\end{align}
where
\vspace{-5mm}
\begin{align}
    f \left( \mathbf{x}, \boldsymbol{\xi} \right) = \min_{\mathbf{y}, \mathbf{z}} \ & \sum_{i \in [m]}\sum_{j \in [m]} b_{i,j} y_{i,j} + \sum_{i \in [m]} c_i z_i \nonumber \\
    \mbox{s.t.} \ & \sum_{j \in [m]} y_{j,i} - \sum_{j \in [m]} y_{i,j} + z_{i} \geq \xi_i - x_i, \ \forall \, i \in [m], \label{cons: balance} \\
    & 0 \leq y_{i,j} \leq Y_{i,j}, \ \forall \, i \in [m], j \in [m]; \ \ z_i \geq 0, \ \forall \, i \in [m]. \nonumber 
\end{align}

\vspace{-4mm}

\noindent
Constraints \eqref{cons: balance} enforce the balance among the shift stock to and from store $i \in [m]$, shortage, demand, and allocation at store $i$. 

Model \eqref{model: lot sizing} is hard to solve in general because of its two-stage nature.
To enhance the solving process, we apply Algorithm 1 in \cite{long2023supermodularity} to solve the two-stage HO model \eqref{model: lot sizing} with $\mathcal{D}_{\textup D}$ (see details in Appendix \ref{app: lot sizing}).
We investigate the performance of model \eqref{model: lot sizing} for scenario reduction, where $N$ scenarios are reduced to $M$.
Specifically, we compare our model with the SAA model using $M$ scenarios, which are reduced from $N$ scenarios by two approaches: (i) ``Random:" selecting $M$ scenarios randomly from $N$, and 
(ii) ``Local Search:" selecting $M$ scenarios using the approximation algorithm based on the local search algorithm proposed in \cite{rujeerapaiboon2022scenario}. Further details about these two approaches are included in Appendix \ref{app: lot sizing}.

We conduct experiments for $N \in \{100, 500, 1000\}$ and $M \in \{10, 20, 30, 40, 50\}$.
For each instance, we conduct five independent runs and report the average result.
Note that the model with a small number of scenarios essentially approximates the original model with a large number of scenarios.
We define the approximation error as $|\text{opt}(M)- \text{opt}^*| / |\text{opt}^*| \times 100\%$, where $\text{opt}(M)$ represents the out-of-sample result of the solution obtained by the approach using $M$ samples and $\text{opt}^*$ represents the out-of-sample result of the solution obtained by SAA model using $N$ samples, 
to measure the quality of a solution obtained by any approach using $M$ samples.
We use ``MAD\_$\sqrt{M_0}$" to denote the HO model \eqref{model: lot sizing} with $\mathcal{D}_{\textup D}$, where the estimation method (iv) is used.

\vspace{-3mm}

\begin{table}[htbp]
    \hspace{-1cm}
  \begin{minipage}{0.5\textwidth}
    \centering
    \scriptsize
    \caption{\tabincell{c}{ Computational Time (s) \\ Without Reduction}
    } 
    \scalebox{0.95}{
    \begin{tabular}{cccc}
    \toprule
    $N$ & 100   & 500   & 1000 \\
    \midrule
    Time & \tabincell{c}{970.53 \\ ($\approx$0.27h)} & \tabincell{c}{21,523.79 \\ ($\approx$5.98h)}  & \tabincell{c}{76,148.79 \\ ($\approx$21.15h)} \\
    \bottomrule
    \end{tabular}
	}
	\label{tab:time_no_reduction}
  \end{minipage}
  %
  \hspace{-0cm}
  \begin{minipage}{0.5\textwidth}
    \centering
    \scriptsize
    \captionof{table}{Approximation Error (\%) When $N=100$} 
    \scalebox{0.95}{
    \begin{tabular}{cccccc}
    \toprule
    $M$ & \textbf{MAD\_$\sqrt{M_0}$} & \textbf{MAD\_Gap} & \textbf{MAD\_Cross} & \textbf{Random} & \textbf{ \tabincell{c}{Local \\ Search} } \\
    \midrule
    10    & 1.05  & 1.07  & 6.33  & 224.09  & 273.65  \\
    20    & 1.04  & 1.04  & 5.17  & 62.28  & 108.94  \\
    30    & 1.02  & 1.04  & 0.55  & 25.05  & 30.33  \\
    40    & 0.86  & 0.99  & 0.50  & 14.95  & 12.52  \\
    50    & 0.66  & 0.90  & 0.68  & 11.05  & 9.31  \\
    \bottomrule
    \end{tabular}
	}
	\label{tab:objgap_n_30_N_100}
  \end{minipage}
\end{table}
\vspace{-10mm}
\begin{table}[htbp]
  \begin{minipage}{0.5\textwidth}
    \centering
    \scriptsize
    \caption{Computational Time (s) When $N=100$} 
    \scalebox{0.95}{
    \begin{tabular}{cccccc}
    \toprule
    $M$ & \textbf{MAD\_$\sqrt{M_0}$} & \textbf{MAD\_Gap} & \textbf{MAD\_Cross} & \textbf{Random} & \textbf{ \tabincell{c}{Local \\ Search} } \\
    \midrule
    10    & 420.27  & 465.96  & 474.25  & 7.62  & 16.29  \\
    20    & 566.44  & 630.77  & 498.80  & 18.13  & 45.65  \\
    30    & 551.64  & 607.84  & 534.43  & 49.40  & 115.44  \\
    40    & 807.50  & 864.86  & 774.68  & 190.10  & 164.17  \\
    50    & 911.89  & 977.78  & 796.39  & 297.38  & 287.19  \\
    \bottomrule
    \end{tabular}%
	}
	\label{tab:time_n_30_N_100}
  \end{minipage}
  %
  \begin{minipage}{0.5\textwidth}
    \centering
    \scriptsize
    \captionof{table}{Preparation Time (s) When $N=100$} 
    \scalebox{0.95}{
    \begin{tabular}{cccccc}
    \toprule
    $M$ & \textbf{MAD\_$\sqrt{M_0}$} & \textbf{MAD\_Gap} & \textbf{MAD\_Cross} & \textbf{Random} & \textbf{ \tabincell{c}{Local \\ Search} } \\
    \midrule
    10    & 0     & 4,336.19  & 10,176.17  & 0     & 402.99  \\
    20    & 0     & 0     & 0     & 0     & 564.78  \\
    30    & 0     & 0     & 0     & 0     & 1,921.93  \\
    40    & 0     & 0     & 0     & 0     & 2,523.40  \\
    50    & 0     & 0     & 0     & 0     & 2,572.29  \\
    \bottomrule
    \end{tabular}%
    }
  \label{tab:preparetime_n_30_N_100}%
  \end{minipage}
\end{table}
\vspace{-3mm}

Table \ref{tab:time_no_reduction} reports the computational time taken by the SAA model to solve instances without scenario reduction.
Tables \ref{tab:objgap_n_30_N_100}--\ref{tab:preparetime_n_30_N_100} report the performance of different scenario reduction approaches for various $M$ when $N=100$.
Table \ref{tab:objgap_n_30_N_100} shows that irrespective of the method used to estimate $\lambda$, our HO model dominates the existing scenario reduction approaches across different values of $M$, in terms of the approximation error.
When $M$ is smaller, i.e., more scenarios are reduced, the advantages of our HO model over the existing scenario reduction approaches become more significant. 
The approximation error of our HO model when $M=10$ is even smaller than that of the existing approaches, including both ``Random" and ``Local Search," when $M=50$.
When using either method (ii) or method (iv) to estimate $\lambda$, our proposed HO approach consistently yields a very low and stable approximation error of around $1\%$ across varying $M$.
When using estimation method (i), the approximation error diminishes rapidly with increasing $M$.
With a small $M$, estimation methods (ii) or (iv) yield lower approximation errors, but as $M$ increases, estimation method (i) exhibits superior performance. 
Interestingly, the ``Local Search" approach does not always perform better than the ``Random" approach.
The former yields a lower approximation error than the latter when $M$ becomes large.

Table \ref{tab:time_n_30_N_100} shows the computational time each approach takes in the solving process.
Table \ref{tab:preparetime_n_30_N_100} shows the time consumed by each approach during the preparatory phase before solving the model, such as the time used for $K$-fold cross-validation to estimate $\lambda$ and the time taken by ``Local Search" to select $M$ scenarios.
Recall that when we use method (i) or method (ii) to estimate $\lambda$, we only need to estimate $C$ once when $M=10$.
As $M$ increases, we set $\lambda = C/\sqrt{M}$.
Thus, the preparation time of ``MAD\_Gap" and ``MAD\_Cross" is $0$ when $M > 10$.
As Tables \ref{tab:time_n_30_N_100}--\ref{tab:preparetime_n_30_N_100} show, the superior performance of our model comes with a computational cost as solving HO model \eqref{model: lot sizing} is more time-consuming, when compared to the ``Random" approach.
Nevertheless, the total time of our model, including both computational and preparation times, remains significantly lower than that of both the ``Local Search" approach and the SAA model considering $N=100$ scenarios (i.e., 970.53s as presented in Table \ref{tab:time_no_reduction}).

Similarly, Tables \ref{tab:objgap_n_30_N_500}--\ref{tab:preparetime_n_30_N_500} provide the performance of various scenario reduction approaches when $N=500$.
We observe similar trends as $N=100$ and $N=1000$ (see details in Appendix \ref{app:computational_performance}).
Specifically, HO outperforms the existing approaches for any $M$, yielding the lowest approximation error.
When $M$ is small, using method (ii) or method (iv) to estimate $\lambda$ can yield a lower approximation error, while as $M$ grows, estimation method (i) shows better performance.
In addition, HO using $M$ scenarios takes significantly shorter computational time than the SAA model using $N$ scenarios.
Note that the ``Local Search" approach requires extensive preparation time to select $M$ scenarios, with the time becoming significantly longer when $N$ is large.
This is because it evaluates the model's performance with respect to each of $N$ scenarios iteratively during the selection process.
Different from the ``Local Search" approach, our HO approach utilizes the partial distributional information calculated from the $N$ scenarios, thereby taking a shorter preparation time and maintaining its efficiency even when $N$ is large.
Moreover, our HO approach achieves a lower approximation error than other approaches by retaining information about the reduced scenarios.

\vspace{-2mm}
\begin{table}[htbp]
  \begin{minipage}{0.5\textwidth}
    \centering
    \scriptsize
    \caption{Approximation Error (\%) When $N=500$} 
    \scalebox{0.95}{
    \begin{tabular}{cccccc}
    \toprule
    $M$ & \textbf{MAD\_$\sqrt{M_0}$} & \textbf{MAD\_Gap} & \textbf{MAD\_Cross} & \textbf{Random} & \textbf{ \tabincell{c}{Local \\ Search} } \\
    \midrule
    10    & 4.49  & 4.49  & 13.00  & 196.36  & 380.91  \\
    20    & 4.48  & 4.47  & 10.69  & 68.51  & 230.39  \\
    30    & 4.48  & 4.46  & 5.39  & 27.86  & 105.46  \\
    40    & 4.47  & 4.41  & 4.04  & 13.95  & 54.16  \\
    50    & 4.47  & 4.32  & 3.39  & 9.87  & 36.57  \\
    \bottomrule
    \end{tabular}
	}
	\label{tab:objgap_n_30_N_500}
  \end{minipage}
  %
  \begin{minipage}{0.5\textwidth}
    \centering
    \scriptsize
    \caption{Computational Time (s) When $N=500$} 
    \scalebox{0.95}{
    \begin{tabular}{cccccc}
    \toprule
    $M$ & \textbf{MAD\_$\sqrt{M_0}$} & \textbf{MAD\_Gap} & \textbf{MAD\_Cross} & \textbf{Random} & \textbf{ \tabincell{c}{Local \\ Search} } \\
    \midrule
    10    & 372.13  & 346.85  & 338.93  & 11.25  & 12.22  \\
    20    & 483.66  & 444.19  & 433.41  & 31.98  & 28.17  \\
    30    & 594.51  & 544.17  & 531.64  & 84.59  & 74.38  \\
    40    & 696.43  & 659.24  & 639.52  & 159.69  & 142.65  \\
    50    & 785.40  & 784.21  & 760.87  & 229.64  & 203.51  \\
    \bottomrule
    \end{tabular}%
    }
  \label{tab:time_n_30_N_500}%
  \end{minipage}
\end{table}

\vspace{-6mm}

\begin{table}[htbp]
    \centering
    \scriptsize
    \caption{Preparation Time (s) When $N=500$} 
    \scalebox{0.95}{
    \begin{tabular}{cccccc}
    \toprule
    $M$ & \textbf{MAD\_$\sqrt{M_0}$} & \textbf{MAD\_Gap} & \textbf{MAD\_Cross} & \textbf{Random} & \textbf{ \tabincell{c}{Local \\ Search} } \\
    \midrule
    10    & 0     & 3,515.04  & 10,311.50  & 0     & 3,600  \\
    20    & 0     & 0     & 0     & 0     & 3,600  \\
    30    & 0     & 0     & 0     & 0     & 3,600  \\
    40    & 0     & 0     & 0     & 0     & 3,600  \\
    50    & 0     & 0     & 0     & 0     & 3,600  \\
    \bottomrule
    \end{tabular}%
	}
	\label{tab:preparetime_n_30_N_500}
\end{table}

\vspace{-6mm}

\section{Conclusion} \label{sec: Conclusion}

Decision-makers often face significant future uncertainties in their decision-making processes, compelling them to address problems under uncertainty.
To solve such problems, stochastic programming is a prominent approach to optimizing the expected performance under a given probability distribution.
However, such a distribution is rarely known to decision-makers in practice.
Extensive studies use historical data to approximate it with the empirical distribution, leading to the well-known SAA approach.
This approach offers strong performance guarantees, demonstrating asymptotic optimality (Proposition \ref{prop: converge}) and providing a confidence interval that includes the expectation under the true distribution (Proposition \ref{prop: interval}).
Despite its success when the sample size $N$ is large,  the SAA's performance may be poor when $N$ is limited because it solely relies on data, which may not approximate the true distribution.
More importantly, determining what qualifies as a ``large'' $N$ may be challenging in practice depending on the uncertain parameters and the model's dimensionality.
Besides data, one may also have partial distributional information about the uncertainty (e.g., moment information) to help them obtain a reliable solution.
Moment-based DRO is a popular approach that utilizes such information.
Unlike the SAA approach, it performs well when $N$ is limited and its model size does not depend on $N$.
However, when $N$ becomes large, its advantages may diminish and it may even provide a conservative solution.
Moreover, determining what qualifies as a ``small'' $N$ may also be challenging.
Therefore, we harmonize the SAA and DRO approaches to maintain the benefits of both of them by integrating data and partial distributional information, leading to a novel approach denoted by HO (see Model \eqref{model_harmonizing}), which works well for any data size without assessing the data size to be large or small.

In HO, the weights for \textit{data} (i.e., $1-\lambda$) and \textit{information} (i.e., $\lambda$) are adaptively adjusted based on $N$.
We achieve this by setting $\lambda=C/\sqrt{N}$, where $C$ is a predetermined fixed constant that one can easily identify (Section \ref{sec: estimate lambda}).
When $N$ is small, $\lambda$ remains large to amplify the influence of information and mitigate the impact of data.
In contrast, when $N$ is large, $\lambda$ decreases to shift the focus to data.
We explain this intuition from an alternative perspective by reformulating the HO model into a DRO model, whose ambiguity set shrinks as $\lambda$ decreases due to the growth of $N$ (Propositions \ref{prop: reform_to_dro} and \ref{prop: ambiguity_set_size}).
Our HO approach exhibits impressive performance guarantees.
In addition to providing a finite-sample performance guarantee (Proposition \ref{prop: non_asymptotic}),
it is also provably asymptotically optimal under mild conditions (Proposition \ref{prop: harmonizing_converge}) and delivers performance guarantees when $\lambda$ is in a $1/\sqrt{N}$-rate (Proposition \ref{prop: harmonizing_bias}), comparable to the Wasserstein-based DRO \citep{gao2023finite}.
More importantly, it can be reformulated into tractable forms easily solved by commercial solvers (Theorem \ref{theorem: reformulation} and Propositions \ref{prop: reform_moment} and \ref{prop: reform_mad}), thereby facilitating significant practical applications.


We further show the applicability and strength of HO in scenario reduction for stochastic programming by incorporating partial distributional information from initial samples.
Compared to the existing scenario reduction approach by \cite{rujeerapaiboon2022scenario}, which struggles with complexities from a large number of initial samples, HO remains effective for any number of initial samples by retaining information of dropped scenarios (Section \ref{sec: Scenario Reduction}).
It offers decision-makers a new approach to reducing the number of scenarios to consider, simplifying decision-making under uncertainty.
We further demonstrate the effectiveness of our HO approach in solving mean-risk portfolio optimization and lot sizing problems.
Numerical results show that HO significantly outperforms the Wasserstein-based DRO in out-of-sample performance (Section \ref{sec: Portfolio Optimization}).
In addition, it dominates the existing scenario reduction approach, achieving rapid completion and low approximation error (all within 4.5\% and some within 1\%), even when the number of scenarios is significantly reduced (Section \ref{sec: lot sizing}). 
Finally, this research can be extended in various directions.
For example, regarding the hyperparameter $C$ for computing the weight parameter $\lambda = C/\sqrt{N}$, it would be intriguing to investigate whether $C$ could be expressed as a function of factors related to uncertainties and the studied problem, which may help better determine the value of $C$. 
It would also be interesting to consider using the Wasserstein-based ambiguity set to represent the partial distributional information in our HO approach.
We leave them for further research.



{\SingleSpacedXI
\bibliographystyle{apalike}
\bibliography{Reference}
}

\newpage

\begin{APPENDICES}

\SingleSpacedXI
\small

\setcounter{table}{0}
\renewcommand{\thetable}{\Alph{section}\arabic{table}}
\setcounter{page}{1}

\section{Notations} \label{sec: notations}

\begin{table}[!htb]
\begin{center}
\caption{Summary of Key Notations}\label{tab:notations}
\centering
\renewcommand{\arraystretch}{1}
{\scriptsize
\begin{tabular}{ c  p{375pt} }
\toprule
{\bf Notation} &  {\bf Description}\\
\midrule

$\textbf{Parameters}$:\\
$N$ & Sample size \\
$M$ & Number of remaining scenarios after scenario reduction \\
$m$ & Dimension of the random vector $\boldsymbol{\xi}$ \\
$n$ & Dimension of the vector of decision variables $\mathbf{x}$ \\
$\lambda$ & Weight of information \\
$\boldsymbol{\mu}_0$, $\boldsymbol{\Sigma}_0$ & Mean value and covariance matrix of the uncertainty under the empirical distribution, respectively \\

$\textbf{Sets}$: \\
$\mathcal{X}$ & Feasibility set of decision variables \\
$\mathcal{S}$ & Uncertainty set \\
$\mathcal{S}_0(M)$ & Set of sample sets, each with a size $M$ \\
$\mathcal{D}$ & Distributional ambiguity set \\
$\mathcal{D}_{\textup T}$ & Distributional ambiguity set with mean and covariance information \\
$\mathcal{D}_{\textup D}$ & Distributional ambiguity set with mean absolute deviation information \\
$\mathcal{D}_{\textup H}(\lambda)$ & Distributional ambiguity set of the combined distribution with the weight $\lambda$ \\
$\mathcal{D}_0(\mathcal{S})$ & Set of all distributions on $\mathcal{S}$ \\

$\textbf{Distributions}$:\\
$\mathbb{P}$ & True distribution of $\boldsymbol{\xi}$ \\
$\mathbb{P}_0$ & Empirical distribution of $\boldsymbol{\xi}$ based on all $N$ samples \\
$\tilde{\mathbb{P}}_0 (\tilde{\mathcal{S}})$ & Empirical distribution of $\boldsymbol{\xi}$ based on a reduced sample set $\tilde{\mathcal{S}}$  \\
$\mathbb{P}_{\textup H}$ & Combined distribution of $\mathbb{P}_0$ and the distribution in a distributional ambiguity set \\

$\textbf{Functions}$:\\
$f(\mathbf{x}, \boldsymbol{\xi})$ & General function returning a real number \\
$F ( \mathbf{x} )$ & Objective function of the original stochastic model, i.e., $F ( \mathbf{x} ) = \mathbb{E}_{\mathbb{P}} [ f ( \mathbf{x}, \boldsymbol{\xi} ) ]$\\
$F_N ( \mathbf{x} )$ & Objective function of the SAA model with $N$ samples, i.e., $F ( \mathbf{x} ) = \mathbb{E}_{\mathbb{P}_0} [ f ( \mathbf{x}, \boldsymbol{\xi} ) ]$ \\
$F_{\lambda}( \mathbf{x} )$ & Objective function of the HO model with weight $\lambda$ \\
$\textbf{Optimal objective values}$:\\
$V^*$ & Optimal value of the original stochastic model \\
$V_N$ & Optimal value of the SAA model with $N$ samples \\
$\Gamma(\lambda) $ & Optimal value of the HO model with weight $\lambda$ \\

$\textbf{Abbreviations}$:\\
SAA & Sample average approximation \\
DRO  & Distributionally robust optimization \\
HO & Harmonizing optimization \\
MAD & Mean absolute deviation \\
PSD & Positive semi-definite \\
\bottomrule
\end{tabular}
}	 
\end{center}
\end{table}

For any integer $N \geq 1$, we use $[N] = \{1, \ldots, N\}$ to denote the set of running indices from $1$ to $N$.
We let $[\underline{a}, \overline{a}]_{\mathbb{Z}}$ denote the set of all integers between any two nonnegative integers $\underline{a}$ and $\overline{a}$; that is, $[\underline{a}, \overline{a}]_{\mathbb{Z}} = \{\underline{a}, \underline{a} + 1, \ldots, \overline{a}\}$ if $\underline{a} \leq \overline{a}$, and $[\underline{a}, \overline{a}]_{\mathbb{Z}} = \emptyset$ if $\underline{a} > \overline{a}$.
We denote scalar values, column vectors, and matrices by non-bold symbols, e.g., $\lambda$, lowercase bold symbols, e.g., $\mathbf{x} = (x_1, \ldots, x_n)^{\top}$, and uppercase characteristics, e.g., $\boldsymbol{\Sigma}$.
If a matrix $\boldsymbol{\Sigma}$ is positive semi-definite (PSD), then we use $\boldsymbol{\Sigma} \succeq 0$.
For multiple matrices or vectors with compatible sizes, we use square brackets to join them together, e.g., $\begin{bmatrix} \mathbf{A} \ \mathbf{B} \end{bmatrix}$ or $\begin{bmatrix} \mathbf{A} \\ \mathbf{B} \end{bmatrix}$.
For a proper cone $\mathcal{K}$ (a closed, convex, and pointed cone with nonempty interior), we let $\mathcal{K}^*$ denote the dual cone of a proper cone $\mathcal{K}$.
We let $\mathbb{S}_+^m \subseteq \mathbb{R}^{m \times m}$ denote the set of all PSD matrices in $\mathbb{R}^{m \times m}$. 
We use $\mathcal{D}_0(\mathbb{R}^m)$ to denote the set of all probability distributions on $\mathbb{R}^m$.
If $\mathbb{P} \in \mathcal{D}_0(\mathbb{R}^m \times \mathbb{R}^h)$ is a joint probability distribution of two random vectors $\boldsymbol{\xi} \in \mathbb{R}^m$ and $\mathbf{u} \in \mathbb{R}^h$, then $\Pi_{\boldsymbol{\xi}} \mathbb{P} \in \mathcal{D}_0(\mathbb{R}^m)$ denotes the marginal distribution of $\boldsymbol{\xi}$ under $\mathbb{P}$.
We extend this definition to any ambiguity set $\mathcal{D} \subseteq \mathcal{D}_0(\mathbb{R}^m \times \mathbb{R}^h)$ by setting $\Pi_{\boldsymbol{\xi}} \mathcal{D} = \cup_{\mathbb{P} \in \mathcal{D}} \{ \Pi_{\boldsymbol{\xi}} \mathbb{P} \}$.
We let $\mathbf{0}$ and $\mathbf{1}$ denote the vectors with all entries being $0$ and $1$, respectively, and $\mathbf{I}$ denote the identity matrix.
We use ``$\bullet$" to denote the inner product defined by $\mathbf{A} \bullet \mathbf{B} = \sum_{i,j}A_{ij}B_{ij}$, where $A_{ij}$ (resp. $B_{ij}$) denotes the entry of $\mathbf{A}$ (resp. $\mathbf{B}$) in row $i$ and column $j$.
We use $\xrightarrow{\mathcal{D}}$ to denote convergence in distribution.
We use $\mathcal{N}(\mu, \sigma)$ to denote the normal distribution with mean $\mu$ and standard deviation $\sigma$, and $\mathcal{U}(\underline{b},\overline{b})$ to denote the uniform distribution with lower bound $\underline{b}$ and upper bound $\overline{b}$.

\section{Supplement to Section \ref{sec: Existing Models}}

\subsection{Proof of Proposition \ref{prop: interval}}
Given any $\mathbf{x} \in \mathcal{X}$, the sample average estimator $F_N ( \mathbf{x} )$ of $F(\mathbf{x})$ is unbiased because the samples $\tilde{\boldsymbol{\xi}}_1, \ldots, \tilde{\boldsymbol{\xi}}_N$ are iid. 
For any $\mathbf{x} \in \mathcal{X}$, by the central limit theorem, we have $\sqrt{N}(F_N ( \mathbf{x} ) - F(\mathbf{x}))$ converges in distribution to a normal distribution with mean $0$ and variance $\sigma^2(\mathbf{x})$, $\mathcal{N}(0, \sigma^2(\mathbf{x}) )$;
that is, $F_N ( \mathbf{x} ) \sim \mathcal{N}(F(\mathbf{x}), \sigma^2(\mathbf{x})/N )$ asymptotically for large $N$. 
Using the $N$ iid samples, we can compute $\hat{\sigma}^2(\mathbf{x})$ as the sample variance estimator of $\sigma^2(\mathbf{x})$.
Thus, for any $\mathbf{x} \in \mathcal{X}$, we have an approximate $100 (1-\alpha) \%$ confidence interval for $\sqrt{N}(F_N ( \mathbf{x} ) - F(\mathbf{x}))$ as 
$[-z_{\frac{\alpha}{2}} \hat{\sigma}(\mathbf{x}), z_{\frac{\alpha}{2}} \hat{\sigma}(\mathbf{x})]$, which completes the proof. 
\Halmos 

\subsection{Special Cases of Moment-Based Ambiguity Sets} \label{sec: Special Cases}

We can recognize several popular moment-based ambiguity sets in the literature as special cases of the ambiguity set $\mathcal{D}$ in \eqref{ambiguity_set}.
For example, with given support $\mathcal{S} \subseteq \mathbb{R}^m$, mean $\boldsymbol{\mu} \in \mathbb{R}^m$, covariance matrix $\boldsymbol{\Sigma} \in \mathbb{R}^{m \times m}$, $\gamma_1 \geq 0$, $\gamma_2 \geq 1$, and $ \boldsymbol{\Sigma} \succ 0 $, we can set 
\begin{align*}
\mathcal{D}_{\textup T}
=\left\{ \mathbb{P} \in \mathcal{D}_0 \left( \mathbb{R}^m \times \mathbb{R}^{\left( m+1 \right) \times m} \right) \middle|
\begin{array}{l}
\mathbb{E}_{\mathbb{P}} \left[ 
\begin{bmatrix} 
    & 1
    \hspace{0.15in} 
    \mathbf{0}^{\top} \\ 
    & \hspace{-2.5mm} \mathbf{0}
    \hspace{0.15in} 
    \mathbf{I}
\end{bmatrix}
\begin{bmatrix} 
    \mathbf{u}_1^{\top} \\
    \mathbf{U}_2
\end{bmatrix}
-
\begin{bmatrix} 
    1 \\
    \mathbf{0}
\end{bmatrix}
\boldsymbol{\xi}^{\top}
\right] 
= 
\begin{bmatrix} 
    \mathbf{0}^{\top}& \\
    \gamma_2 \boldsymbol{\Sigma}
\end{bmatrix} 
\vspace{5mm}\\
\mathbb{P} \left(
\boldsymbol{\xi} \in \mathcal{S}
\right) = 1
\vspace{5mm}\\
\mathbb{P} \left(
    \begin{bmatrix} 
    \boldsymbol{\Sigma}
    & \hspace{0.1 in} 
    \left(
    \begin{bmatrix} 
        1 \hspace{0.15in} \mathbf{0}^{\top}
    \end{bmatrix}
    \begin{bmatrix} 
        \mathbf{u}_1^{\top} \\
        \mathbf{U}_2
    \end{bmatrix} 
    -
    \boldsymbol{\mu}^{\top}
    \right)^{\top}
    \\ 
    \left(
    \begin{bmatrix} 
        1 \hspace{0.15in} \mathbf{0}^{\top}
    \end{bmatrix}
    \begin{bmatrix} 
        \mathbf{u}_1^{\top} \\
        \mathbf{U}_2
    \end{bmatrix} 
    -
    \boldsymbol{\mu}^{\top}
    \right)
    & \gamma_1
    \end{bmatrix} \succeq 0
    \right)
= 1
\vspace{5mm}\\
\mathbb{P} \left(
    \begin{bmatrix} 
    1
    & \hspace{0.4in}  \left( \boldsymbol{\xi} - \boldsymbol{\mu} \right)^{\top} 
    \vspace{4mm} \\ 
    \left( \boldsymbol{\xi} - \boldsymbol{\mu} \right)
    & 
    \hspace{7mm}
    \begin{bmatrix} 
        \mathbf{0} \hspace{0.15in} \mathbf{I}
    \end{bmatrix}
    \begin{bmatrix} 
        \mathbf{u}_1^{\top} \\
        \mathbf{U}_2
    \end{bmatrix} 
    \end{bmatrix} \succeq 0
    \right)
= 1
\end{array}
\right\},
\end{align*}
where the auxiliary random parameters $\mathbf{u}_1 \in \mathbb{R}^m$ and $\mathbf{U}_2 \in \mathbb{R}^{m \times m}$.
It follows that
\begin{align*}
    \mathbb{P}_{\boldsymbol{\xi}} \in \Pi_{\boldsymbol{\xi}} \mathcal{D}_{\textup T} = 
    \left\{ 
    \mathbb{P}_{\boldsymbol{\xi}} \in \mathcal{D}_0 \left( \mathbb{R}^m \right) \ \middle| \
    \begin{array}{l}
        \mathbb{P}_{\boldsymbol{\xi}} \left( \boldsymbol{\xi} \in \mathcal{S} \right) = 1 \\
        \left( \mathbb{E}_{\mathbb{P}_{\boldsymbol{\xi}}}\left[ \boldsymbol{\xi} \right] -\boldsymbol{\mu} \right)^{\top} \boldsymbol{\Sigma}^{-1} 
        \left( \mathbb{E}_{\mathbb{P}_{\boldsymbol{\xi}}}\left[ \boldsymbol{\xi} \right] - 
        \boldsymbol{\mu} 
        \right)
        \leq \gamma_1 \\
        \mathbb{E}_{\mathbb{P}_{\boldsymbol{\xi}}} \left[ \left( \boldsymbol{\xi} - \boldsymbol{\mu} \right) \left( \boldsymbol{\xi} - \boldsymbol{\mu} \right)^{\top} \right]  
        \preceq \gamma_2\boldsymbol{\Sigma}
    \end{array}
    \right\},
\end{align*}
which describes that the support of $\boldsymbol{\xi}$ is $\mathcal{S}$,
the mean of $\boldsymbol{\xi}$ lies in an ellipsoid of size $\gamma_1$ centered at $\boldsymbol{\mu}$, 
and the covariance of $\boldsymbol{\xi}$ is bounded from above by $\gamma_2 \boldsymbol{\Sigma}$.
In addition, with given support $\mathcal{S} \subseteq \mathbb{R}^m$, mean $\boldsymbol{\mu} \in \mathbb{R}^m$ and deviation $\boldsymbol{\delta} \in \mathbb{R}^m$, we can set 
\begin{align*}
\mathcal{D}_{\textup D}
=\left\{ \mathbb{P} \in \mathcal{D}_0 \left( \mathbb{R}^m \times \mathbb{R}^{m} \right) \ \middle| \
\begin{array}{l}
\mathbb{E}_{\mathbb{P}} \left[ \boldsymbol{\xi} \right] = \boldsymbol{\mu} \\
\mathbb{E}_{\mathbb{P}} \left[ \mathbf{u} \right] = \boldsymbol{\delta} \\
\mathbb{P} \left(
\boldsymbol{\xi} \in \mathcal{S}
\right) = 1 \\
\mathbb{P} \left(
   \mathbf{u} \geq \boldsymbol{\xi} - \boldsymbol{\mu}, \ \mathbf{u} \geq \boldsymbol{\mu} - \boldsymbol{\xi}
    \right)
= 1
\end{array}
\right\},
\end{align*}
where the auxiliary random vector $\mathbf{u} \in \mathbb{R}^{m}$. 
It follows that
\begin{align*}
    \mathbb{P}_{\boldsymbol{\xi}} \in \Pi_{\boldsymbol{\xi}} \mathcal{D}_{\textup D}
    = 
    \left\{ 
    \mathbb{P}_{\boldsymbol{\xi}} \in \mathcal{D}_0 \left( \mathbb{R}^m \right) \ \middle| \
    \begin{array}{l}
        \mathbb{P}_{\boldsymbol{\xi}} \left( \boldsymbol{\xi} \in \mathcal{S} \right) = 1 \\ \mathbb{E}_{\mathbb{P}_{\boldsymbol{\xi}}}\left[ \boldsymbol{\xi} \right] = \boldsymbol{\mu}\\
        \mathbb{E}_{\mathbb{P}_{\boldsymbol{\xi}}} \left[ \left| \boldsymbol{\xi} - \boldsymbol{\mu} \right| \right] \leq \boldsymbol{\delta}
    \end{array}
    \right\},
\end{align*}
which specifies the support, mean, and mean absolute deviation (MAD) information of $\boldsymbol{\xi}$.

\section{Supplement to Section \ref{sec: HO}}

\subsection{Proof of Proposition \ref{prop: reform_to_dro}}
We have
\begin{align}
& \min_{\mathbf{x} \in \mathcal{X}} \left\{ \left( 1 - \lambda \right) \mathbb{E}_{\mathbb{P}_0} \left[ f \left( \mathbf{x}, \boldsymbol{\xi} \right) \right] + \lambda \max_{\mathbb{P} \in \mathcal{D}} \mathbb{E}_{\mathbb{P}}[f(\mathbf{x}, \boldsymbol{\xi})] \right\} \nonumber \\
= & \min_{\mathbf{x} \in \mathcal{X}} \max_{\mathbb{P} \in \mathcal{D}} \Big\{  \left( 1 - \lambda \right) \mathbb{E}_{\mathbb{P}_0} \left[ f \left( \mathbf{x}, \boldsymbol{\xi} \right) \right] + \lambda \mathbb{E}_{\mathbb{P}}[f(\mathbf{x}, \boldsymbol{\xi})] \Big\} \nonumber \\
= & \min_{\mathbf{x} \in \mathcal{X}} \max_{\mathbb{P}_{\boldsymbol{\xi}} \in \Pi_{\boldsymbol{\xi}}\mathcal{D}} \left\{  \left( 1 - \lambda \right) \mathbb{E}_{\mathbb{P}_0} \left[ f \left( \mathbf{x}, \boldsymbol{\xi} \right) \right] + \lambda \mathbb{E}_{\mathbb{P}_{\boldsymbol{\xi}}}[f(\mathbf{x}, \boldsymbol{\xi})] \right\} \nonumber \\
= & \min_{\mathbf{x} \in \mathcal{X}} \max_{\mathbb{P}_{\boldsymbol{\xi}} \in \Pi_{\boldsymbol{\xi}}\mathcal{D}} \mathbb{E}_{\left( 1-\lambda \right) \mathbb{P}_0 + \lambda\mathbb{P}_{\boldsymbol{\xi}}} \left[ f \left( \boldsymbol{x},\boldsymbol{\xi} \right) \right] \nonumber \\
= & \min_{\boldsymbol{x} \in \mathcal{X}} \ \max_{\mathbb{P}_{\textup H} \in \mathcal{D}_{\textup H} \left( \lambda \right)} \ \mathbb{E}_{\mathbb{P}_{\textup H}} \left[ f \left( \boldsymbol{x},\boldsymbol{\xi} \right) \right], \nonumber
\end{align}
which completes the proof.
\Halmos 

\subsection{Proof of Proposition \ref{prop: ambiguity_set_size}}
For any $\mathbb{P}_{\textup H} = (1-\lambda_2)\mathbb{P}_0 + \lambda_2\mathbb{P}_{\boldsymbol{\xi}} \in \mathcal{D}_{\textup H}(\lambda_2)$, by the definition of $\mathcal{D}_{\textup H}(\lambda)$, we have $\mathbb{P}_{\boldsymbol{\xi}} \in \Pi_{\boldsymbol{\xi}}\mathcal{D}$.
We define $\mathbb{P}_{\textup H}^{\prime} = (1 - \lambda_2/\lambda_1) \mathbb{P}_0 + (\lambda_2/\lambda_1) \mathbb{P}_{\boldsymbol{\xi}}$.
Since $0 \leq \lambda_2 \leq \lambda_1$, we have $\lambda_2/\lambda_1 \in [0,1]$.
Note that $\Pi_{\boldsymbol{\xi}}\mathcal{D}$ is convex.
Therefore, if $\mathbb{P}_0 \in \Pi_{\boldsymbol{\xi}}\mathcal{D}$, then we have $\mathbb{P}_{\textup H}^{\prime} \in \Pi_{\boldsymbol{\xi}}\mathcal{D}$.
By the definition of $\mathcal{D}_{\textup H}(\lambda)$, we then have $\mathbb{P}_{\textup H} 
= (1-\lambda_2)\mathbb{P}_0 + \lambda_2\mathbb{P}_{\boldsymbol{\xi}}
= (1-\lambda_1)\mathbb{P}_0 + \lambda_1 \mathbb{P}_{\textup H}^{\prime} \in \mathcal{D}_{\textup H}(\lambda_1)$, indicating that $\mathcal{D}_{\textup H}(\lambda_2) \subseteq \mathcal{D}_{\textup H}(\lambda_1)$.
This completes the proof.
\Halmos 

\subsection{{Proof of Proposition \ref{prop: wasserstein_equivalence}}}
For any $x \in \mathcal{X}$, given any $\lambda \in [0, 1]$ and $r_{\textup H} \in \mathbb{R}_+$, we have
\begin{align*}
    \left( 1 - \lambda \right) \mathbb{E}_{\mathbb{P}_0} \left[ f \left( \mathbf{x}, \boldsymbol{\xi} \right) \right] + \lambda \max_{\mathbb{P} \in \mathcal{D}_{\textup W}(r_{\textup H})} \mathbb{E}_{\mathbb{P}}\left[ f \left( \mathbf{x}, \boldsymbol{\xi} \right) \right] & = \left( 1 - \lambda \right) \mathbb{E}_{\mathbb{P}_0} \left[ f \left( \mathbf{x}, \boldsymbol{\xi} \right) \right] + \lambda \left( \mathbb{E}_{\mathbb{P}_0} \left[ f \left( \mathbf{x}, \boldsymbol{\xi} \right) \right] + r_{\textup H} \text{lip} \left( f\left( \mathbf{x}, \cdot \right) \right) \right) \\
    & = \mathbb{E}_{\mathbb{P}_0} \left[ f \left( \mathbf{x}, \boldsymbol{\xi} \right) \right] + \lambda r_{\textup H} \text{lip} \left( f\left( \mathbf{x}, \cdot \right) \right) \\
    & = \max_{\mathbb{P} \in \mathcal{D}_{\textup W}(r_{\textup W})} \mathbb{E}_{\mathbb{P}}\left[ f \left( \mathbf{x}, \boldsymbol{\xi} \right) \right],
\end{align*}
where $r_{\textup W} = \lambda r_{\textup H}$, $\text{lip} ( f( \mathbf{x}, \cdot ))$ denotes the Lipschitz constant of $f( \mathbf{x}, \cdot )$, and the first and last equalities hold by Proposition 6.17 in \cite{kuhn2025dro}.
\Halmos 

\subsection{Proof of Proposition \ref{prop: non_asymptotic}}
For any $\mathbb{P}^{\prime} \in \Pi_{\boldsymbol{\xi}}\mathcal{D}$, let $\underline{\lambda}(\mathbb{P}^{\prime}) = \argmin \{\lambda \ | \ \mathbb{P}^{\prime} \in \mathcal{D}_{\textup H}(\lambda) \}$.
We have $\mathbb{P}^{\prime}$ is on the boundary of $\mathcal{D}_{\textup H}(\underline{\lambda}(\mathbb{P}^{\prime}))$, i.e., $\mathbb{P}^{\prime} \in \partial \mathcal{D}_{\textup H}(\underline{\lambda}(\mathbb{P}^{\prime}))$; 
otherwise, we can always find a $\lambda^{\prime}$ that is smaller than $\underline{\lambda}(\mathbb{P}^{\prime})$ and satisfies $\mathbb{P}^{\prime} \in \mathcal{D}_{\textup H}(\lambda^{\prime})$.
We define set $\mathcal{B}_{\epsilon}^{\textup G}(\boldsymbol{\mu}_0, \boldsymbol{\Sigma}_0) = \{ \mathbb{P}_{\textup H} \in \mathcal{D}_{\textup H}(1) \ | \ \mathcal{G}((\boldsymbol{\mu}_0, \boldsymbol{\Sigma}_0), (\boldsymbol{\mu}(\mathbb{P}_{\textup H}), \boldsymbol{\Sigma}(\mathbb{P}_{\textup H}))) \leq \epsilon \}$, which contains all the distributions in $\mathcal{D}_{\textup H}(1)$ whose mean-covariance pairs have a Gelbrich distance of at most $\epsilon$ from the pair $(\boldsymbol{\mu}_0, \boldsymbol{\Sigma}_0)$.
We define set $\mathcal{B}_{\epsilon}^{\textup W}(\mathbb{P}_0) = \{ \mathbb{P}_{\textup H} \in \mathcal{D}_{\textup H}(1) \ | \ \mathcal{W}_2(\mathbb{P}_0, \mathbb{P}_{\textup H}) \leq \epsilon \}$, which contains all the distributions in $\mathcal{D}_{\textup H}(1)$ that have a type-2 Wasserstein distance of at most $\epsilon$ from $\mathbb{P}_0$.

First, we show that $\mathcal{B}_{\epsilon}^{\textup G}(\boldsymbol{\mu}_0, \boldsymbol{\Sigma}_0) \subseteq \mathcal{D}_{\textup H}(\lambda^*)$ by contradiction.
Suppose there exists $\mathbb{P}^{\prime} \in \mathcal{B}_{\epsilon}^{\textup G}(\boldsymbol{\mu}_0, \boldsymbol{\Sigma}_0)$ such that $\mathbb{P}^{\prime} \notin \mathcal{D}_{\textup H}(\lambda^*)$.
By Proposition \ref{prop: ambiguity_set_size}, we have $\underline{\lambda}(\mathbb{P}^{\prime}) > \lambda^*$.
By the definition of $\mathcal{D}_{\textup H}(\underline{\lambda}(\mathbb{P}^{\prime}))$, we have $\mathbb{P}^{\prime} = (1-\underline{\lambda}(\mathbb{P}^{\prime})) \mathbb{P}_0 + \underline{\lambda}(\mathbb{P}^{\prime}) \overline{\mathbb{P}}$, where $\overline{\mathbb{P}} \in \Pi_{\boldsymbol{\xi}} \mathcal{D}$.
More precisely, we have $\overline{\mathbb{P}} \in \partial \Pi_{\boldsymbol{\xi}} \mathcal{D}$ because otherwise, i.e., $\overline{\mathbb{P}} \notin \partial \Pi_{\boldsymbol{\xi}} \mathcal{D}$, we can always find a $\lambda^{\prime}$ that is smaller than $\underline{\lambda}(\mathbb{P}^{\prime})$ and satisfies $\mathbb{P}^{\prime} \in \mathcal{D}_{\textup H}(\lambda^{\prime})$.
Given $\mathbb{P}_0, \mathbb{P}^{\prime} \in \Pi_{\boldsymbol{\xi}} \mathcal{D}$, we define a new distribution $\mathbb{P}_{\textup H}$ as a convex combination of $\mathbb{P}_0$ and $\mathbb{P}^{\prime}$:
\begin{align*}
\mathbb{P}_{\textup H} 
& = \left( 1- \frac{\lambda^*}{\underline{\lambda}(\mathbb{P}^{\prime})} \right) \mathbb{P}_0 + \frac{\lambda^*}{\underline{\lambda}(\mathbb{P}^{\prime})} \mathbb{P}^{\prime} \\
& = \left( 1- \frac{\lambda^*}{\underline{\lambda}(\mathbb{P}^{\prime})} \right) \mathbb{P}_0 + \frac{\lambda^*}{\underline{\lambda}(\mathbb{P}^{\prime})} \left( \left( 1-\underline{\lambda}(\mathbb{P}^{\prime})) \mathbb{P}_0 + \underline{\lambda}(\mathbb{P}^{\prime} \right) \overline{\mathbb{P}} \right) \\
& = \left( 1-\lambda^* \right) \mathbb{P}_0 + \lambda^* \overline{\mathbb{P}}.
\end{align*}
Since $\overline{\mathbb{P}} \in \partial \Pi_{\boldsymbol{\xi}} \mathcal{D}$, by the definition of $\mathcal{D}_{\textup H}(\lambda^*)$, we have $\mathbb{P}_{\textup H} \in \partial\mathcal{D}_{\textup H}(\lambda^*)$.
As a result, by the definition of $\lambda^*$ in \eqref{best_lambda}, we have 
\[\mathcal{G} ((\boldsymbol{\mu}_0, \boldsymbol{\Sigma}_0), (\boldsymbol{\mu} (\mathbb{P}_{\textup H}), \boldsymbol{\Sigma}(\mathbb{P}_{\textup H}))) > \epsilon.\]

\noindent
Note that Corollary 3 in \cite{nguyen2021mean} suggests that $\mathcal{B}_{\epsilon}^{\textup G}(\boldsymbol{\mu}_0, \boldsymbol{\Sigma}_0)$ is convex. 
Thus, as $\mathbb{P}_0, \mathbb{P}^{\prime} \in \mathcal{B}_{\epsilon}^{\textup G}(\boldsymbol{\mu}_0, \boldsymbol{\Sigma}_0)$, we have $\mathbb{P}_{\textup H} \in \mathcal{B}_{\epsilon}^{\textup G}(\boldsymbol{\mu}_0, \boldsymbol{\Sigma}_0)$, which shows that $\mathcal{G} ((\boldsymbol{\mu}_0, \boldsymbol{\Sigma}_0), (\boldsymbol{\mu} (\mathbb{P}_{\textup H}), \boldsymbol{\Sigma}(\mathbb{P}_{\textup H}))) \leq \epsilon$, leading to the contradiction.
Therefore, for any $\mathbb{P}^{\prime} \in \mathcal{B}_{\epsilon}^{\textup G}(\boldsymbol{\mu}_0, \boldsymbol{\Sigma}_0)$, we have $\mathbb{P}^{\prime} \in \mathcal{D}_{\textup H}(\lambda^*)$, indicating that $\mathcal{B}_{\epsilon}^{\textup G}(\boldsymbol{\mu}_0, \boldsymbol{\Sigma}_0) \subseteq \mathcal{D}_{\textup H}(\lambda^*)$.

Second, we show that $\mathcal{B}_{\epsilon}^{\textup W}(\mathbb{P}_0) \subseteq
\mathcal{B}_{\epsilon}^{\textup G}(\boldsymbol{\mu}_0, \boldsymbol{\Sigma}_0)$.
For any $\mathbb{P}^{\prime} \in \mathcal{B}_{\epsilon}^{\textup W}(\mathbb{P}_0)$, we have $\mathcal{W}_2(\mathbb{P}_0, \mathbb{P}^{\prime}) \leq \epsilon$.
By Theorem 1 in \cite{nguyen2021mean}, we have
$\mathcal{G} ( ( \boldsymbol{\mu}_0, \boldsymbol{\Sigma}_0 ), ( \boldsymbol{\mu}(\mathbb{P}^{\prime}), \boldsymbol{\Sigma}(\mathbb{P}^{\prime}) ) ) \leq
\mathcal{W}_2(\mathbb{P}_0, \mathbb{P}^{\prime}) \leq \epsilon$, indicating that $\mathbb{P}^{\prime} \in \mathcal{B}_{\epsilon}^{\textup G}(\boldsymbol{\mu}_0, \boldsymbol{\Sigma}_0)$.
Therefore, we have $\mathcal{B}_{\epsilon}^{\textup W}(\mathbb{P}_0) \subseteq
\mathcal{B}_{\epsilon}^{\textup G}(\boldsymbol{\mu}_0, \boldsymbol{\Sigma}_0)$.
Consequently, we have
\begin{align}
    \mathcal{B}_{\epsilon}^{\textup W}(\mathbb{P}_0) 
    \subseteq
    \mathcal{B}_{\epsilon}^{\textup G}(\boldsymbol{\mu}_0, \boldsymbol{\Sigma}_0)
    \subseteq
    \mathcal{D}_{\textup H}(\lambda^*). \label{set_cover}
\end{align}

Finally, by Theorem 2 in \cite{fournier2015rate}, we have $\mathcal{P}(\mathbb{P} \in \mathcal{B}_{\epsilon}^{\textup W}(\mathbb{P}_0)) \geq 1-\beta$.
By \eqref{set_cover}, we further have
$\mathcal{P}( \mathbb{\mathbb{P}} \in \mathcal{D}_{\textup H}(\lambda^*) ) \geq \mathcal{P}(\mathbb{P} \in \mathcal{B}_{\epsilon}^{\textup W}(\mathbb{P}_0)) \geq 1-\beta$.
Moreover, by Proposition \ref{prop: ambiguity_set_size}, we have $\mathcal{D}_{\textup H}(\lambda^*) \subseteq \mathcal{D}_{\textup H}(\lambda)$ for any $\lambda \geq \lambda^*$.
It follows that $\mathcal{P} ( \mathbb{P} \in \mathcal{D}_{\textup H}(\lambda) ) \geq \mathcal{P} ( \mathbb{P} \in \mathcal{D}_{\textup H}(\lambda^*) ) $ for any $\lambda \geq \lambda^*$.
Furthermore, note that given any $\lambda \in [0,1]$, the definition of $\mathcal{D}_{\textup H}(\lambda)$ implies that $\mathbb{P} \in \mathcal{D}_{\textup H}(\lambda)$ is equivalent to the existence of a $\mathbb{P}_{\textup M} \in \Pi_{\boldsymbol{\xi}} \mathcal{D}$ such that $\mathbb{P} = (1-\lambda)\mathbb{P}_0 + \lambda \mathbb{P}_{\textup M}$.
By the definition of a mixture distribution, $\lambda$ also represents the probability that the mixture distribution $\mathbb{P}$ is $\mathbb{P}_{\textup M}$.
That is, given any $\lambda \in [0,1]$, we have $\mathcal{P}( \mathbb{P} = (1-\lambda) \mathbb{P}_0 + \lambda \mathbb{P} ) = \lambda$.
Since $\mathbb{P} \in \Pi_{\boldsymbol{\xi}} \mathcal{D}$, we have $\mathcal{P}(\mathbb{P} \in \mathcal{D}_{\textup H}(\lambda)) \geq \mathcal{P}( \mathbb{P} = (1-\lambda) \mathbb{P}_0 + \lambda \mathbb{P} ) = \lambda \geq 1-\beta$ for any $\lambda \in [1-\beta, 1]$.
\Halmos 


\subsection{Bisection Search Algorithm} \label{app: Bisection Search Algorithm}

Let $g(\lambda) = \min_{\mathbb{P}_{\textup H} \in \partial\mathcal{D}_{\textup H}(\lambda)} \mathcal{G}( (\boldsymbol{\mu}_0, \boldsymbol{\Sigma}_0), (\boldsymbol{\mu}(\mathbb{P}_{\textup H}), \boldsymbol{\Sigma}(\mathbb{P}_{\textup H})))$ for any $\lambda \in [0,1]$.
Algorithm \ref{alg: bisection} presents the details of the bisection search algorithm used for determining $\lambda^*$.

\begin{algorithm}[!htb]
\SingleSpacedXI
\footnotesize
\caption{Bisection Search Algorithm} \label{alg: bisection}
\begin{algorithmic}[1]
\Require
$\underline{\lambda}=0$, $\overline{\lambda} = 1$, $\Delta = 10^{-6}$, $\epsilon$.
\Do
\State Set $\hat{\lambda} = (\underline{\lambda} + \overline{\lambda})/2$. 
\If{$g(\hat{\lambda}) \geq \epsilon$}
\State Set $\overline{\lambda} = \hat{\lambda}$.
\Else
\State Set $\underline{\lambda} = \hat{\lambda}$.
\EndIf
\doWhile{$\overline{\lambda} - \underline{\lambda} \geq \Delta$}
\Ensure
$\lambda^* = \hat{\lambda}$.
\end{algorithmic}
\end{algorithm}

\subsection{Proof of Proposition \ref{prop: harmonizing_converge}}
For any $\mathbf{x} \in \mathcal{X}$ and $\lambda \in [0,1]$, we have
\begin{align}
\left| F_{\lambda} \left( \mathbf{x} \right) - F \left( \mathbf{x} \right) \right| 
& \leq \left| F_{\lambda} \left( \mathbf{x} \right) - F_N \left( \mathbf{x} \right) \right| + \left| F_N \left( \mathbf{x} \right) - F \left( \mathbf{x} \right) \right| \nonumber \\
& = \lambda \left| \max_{\mathbb{P} \in \mathcal{D}} \mathbb{E}_{\mathbb{P}}\left[ f \left( \mathbf{x}, \boldsymbol{\xi} \right) \right] - \mathbb{E}_{\mathbb{P}_0} \left[ f \left( \mathbf{x}, \boldsymbol{\xi} \right) \right] \right| + \left| F_N \left( \mathbf{x} \right) - F \left( \mathbf{x} \right) \right|. \label{ineq: proof_harmonizing_converge_1}
\end{align} 
Since $\mathbb{E}_{\mathbb{P}} [ |f ( \mathbf{x}, \boldsymbol{\xi} )| ] < \infty$ for any $\mathbf{x} \in \mathcal{X}$ with any given $\mathbb{P}$, we have $| \max_{\mathbb{P} \in \mathcal{D}} \mathbb{E}_{\mathbb{P}} [ f(\mathbf{x}, \boldsymbol{\xi}) ] - \mathbb{E}_{\mathbb{P}_0} [ f ( \mathbf{x}, \boldsymbol{\xi} ) ]| < \infty$ for any $\mathbf{x} \in \mathcal{X}$.
Thus, for any $\mathbf{x} \in \mathcal{X}$ and $\epsilon_1 > 0$, there exists $N_1(\mathbf{x}, \epsilon_1) = (C | \max_{\mathbb{P} \in \mathcal{D}} \mathbb{E}_{\mathbb{P}} [ f(\mathbf{x}, \boldsymbol{\xi}) ] - \mathbb{E}_{\mathbb{P}_0} [ f ( \mathbf{x}, \boldsymbol{\xi} ) ]| / \epsilon_1)^2$ such that for any $N > N_1(\mathbf{x}, \epsilon_1)$,
\begin{align}
    \lambda \left| \max_{\mathbb{P} \in \mathcal{D}} \mathbb{E}_{\mathbb{P}}\left[ f \left( \mathbf{x}, \boldsymbol{\xi} \right) \right] - \mathbb{E}_{\mathbb{P}_0} \left[ f \left( \mathbf{x}, \boldsymbol{\xi} \right) \right] \right| 
    & = 
    \frac{C}{\sqrt{N}} \left| \max_{\mathbb{P} \in \mathcal{D}} \mathbb{E}_{\mathbb{P}}\left[ f \left( \mathbf{x}, \boldsymbol{\xi} \right) \right] - \mathbb{E}_{\mathbb{P}_0} \left[ f \left( \mathbf{x}, \boldsymbol{\xi} \right) \right] \right| \nonumber \\
    & < \frac{C}{\sqrt{N_1(\mathbf{x}, \epsilon_1)}} \left| \max_{\mathbb{P} \in \mathcal{D}} \mathbb{E}_{\mathbb{P}}\left[ f \left( \mathbf{x}, \boldsymbol{\xi} \right) \right] - \mathbb{E}_{\mathbb{P}_0} \left[ f \left( \mathbf{x}, \boldsymbol{\xi} \right) \right] \right|
    = \epsilon_1. \label{ineq: proof_harmonizing_converge_2}
\end{align}
Furthermore, $F_N(\mathbf{x})$ converges to $F(\mathbf{x})$ w.p. 1 as $N \rightarrow \infty$, uniformly on $\mathcal{X}$.
That is, for any $\mathbf{x} \in \mathcal{X}$ and $\epsilon_2 > 0$, there exists $N_2(\mathbf{x},\epsilon_2)$ such that
\begin{align}
  \left| F_N \left( \mathbf{x} \right) - F \left( \mathbf{x} \right) \right| < \epsilon_2, \ \forall \, N > N_2(\mathbf{x},\epsilon_2). \label{ineq: proof_harmonizing_converge_3}
\end{align}
By \eqref{ineq: proof_harmonizing_converge_1}--\eqref{ineq: proof_harmonizing_converge_3}, we have for any $\mathbf{x} \in \mathcal{X}$, $\epsilon_1, \epsilon_2 > 0$, there exists $N_3(\mathbf{x},\epsilon_1,\epsilon_2) = \max\{N_1(\mathbf{x},\epsilon_1), N_2(\mathbf{x},\epsilon_2)\}$ such that
\begin{align}
    \left| F_{\lambda} \left( \mathbf{x} \right) - F \left( \mathbf{x} \right) \right| < \epsilon_1 + \epsilon_2, \ \forall \, N > N_3(\mathbf{x},\epsilon_1,\epsilon_2). \label{ineq: proof_harmonizing_converge_4}
\end{align}
Additionally, we have
\begin{align*}
    \Gamma(\lambda) - V^* 
    & = \min_{\mathbf{x}_1 \in \mathcal{X}} F_{\lambda} \left( \mathbf{x}_1 \right) - \min_{\mathbf{x}_2 \in \mathcal{X}} F \left( \mathbf{x}_2 \right) \nonumber = \min_{\mathbf{x}_1 \in \mathcal{X}} F_{\lambda} \left( \mathbf{x}_1 \right) - F \left( \mathbf{x}_2^* \right) \leq F_{\lambda} \left( \mathbf{x}_2^* \right) - F \left( \mathbf{x}_2^* \right)  \\
    & \leq \max_{\mathbf{x} \in \mathcal{X}} \left\{ F_{\lambda} \left( \mathbf{x} \right) - F \left( \mathbf{x} \right) \right\} \leq \max_{\mathbf{x} \in \mathcal{X}} \left\{ \left| F_{\lambda} \left( \mathbf{x} \right) - F \left( \mathbf{x} \right) \right| \right\},
\end{align*}
where $\mathbf{x}_2^*$ is the optimal solution of $\min_{\mathbf{x}_2 \in \mathcal{X}} F( \mathbf{x}_2 )$.
Similarly, we can have $V^* - \Gamma(\lambda) \leq \max_{\mathbf{x} \in \mathcal{X}} \{ | F_{\lambda} ( \mathbf{x} ) - F ( \mathbf{x} ) | \}$. 
Let $\mathbf{x}^* = \argmax_{\mathbf{x} \in \mathcal{X}} \{ | F_{\lambda} ( \mathbf{x} ) - F ( \mathbf{x} ) | \}$. 
For any $\epsilon_1, \epsilon_2 > 0$, by \eqref{ineq: proof_harmonizing_converge_4}, we have
\begin{align*}
    \left| \Gamma(\lambda) - V^* \right| \leq \max_{\mathbf{x} \in \mathcal{X}} \left\{ \left| F_{\lambda} \left( \mathbf{x} \right) - F \left( \mathbf{x} \right) \right| \right\} = \left| F_{\lambda} \left( \mathbf{x}^* \right) - F \left( \mathbf{x}^* \right) \right| < \epsilon_1 + \epsilon_2, \ \forall \, N > N_3(\mathbf{x}^*,\epsilon_1,\epsilon_2).
\end{align*}
This completes the proof.
\Halmos 


\subsection{Proof of Proposition \ref{prop: harmonizing_bias}}
We have 
\begin{align*}
\Gamma(\lambda) - V_N 
& = \min_{\mathbf{x}_{1} \in \mathcal{X}} F_{\lambda} \left( \mathbf{x}_1 \right) - \min_{\mathbf{x}_2 \in \mathcal{X}} F_{N} \left( \mathbf{x}_2 \right) 
= \min_{\mathbf{x}_{1} \in \mathcal{X}} F_{\lambda} \left( \mathbf{x}_1 \right) - F_{N} \left( \mathbf{x}_2^* \right)
\leq F_{\lambda} \left( \mathbf{x}_2^* \right) - F_{N} \left( \mathbf{x}_2^* \right) \\
& \leq \max_{\mathbf{x} \in \mathcal{X}} \left\{ F_{\lambda} \left( \mathbf{x} \right)-F_{N} \left( \mathbf{x} \right) \right\}
= \frac{C}{\sqrt{N}} \max_{\mathbf{x} \in \mathcal{X}} \left\{ \max_{\mathbb{P} \in \mathcal{D}} \mathbb{E}_{\mathbb{P}}\left[ f \left( \mathbf{x}, \boldsymbol{\xi} \right) \right] - \mathbb{E}_{\mathbb{P}_0} \left[ f \left( \mathbf{x}, \boldsymbol{\xi} \right) \right] \right\},
\end{align*}
where $\mathbf{x}_2^*$ is the optimal solution of $\min_{\mathbf{x}_2 \in \mathcal{X}} F_N( \mathbf{x}_2 )$.
Similarly, we also have
\begin{align*}
V_N - \Gamma(\lambda) 
\leq \frac{C}{\sqrt{N}} \max_{\mathbf{x} \in \mathcal{X}} \left\{ \mathbb{E}_{\mathbb{P}_0} \left[ f \left( \mathbf{x}, \boldsymbol{\xi} \right) \right] - \max_{\mathbb{P} \in \mathcal{D}} \mathbb{E}_{\mathbb{P}}\left[ f \left( \mathbf{x}, \boldsymbol{\xi} \right) \right] \right\}.
\end{align*}
Due to the assumption that $\mathbb{E}_{\mathbb{P}} [ |f ( \mathbf{x}, \boldsymbol{\xi} )| ] < \infty$ for any $\mathbf{x} \in \mathcal{X}$ with any given $\mathbb{P}$, 
there exists a finite constant $\overline{C}_1 > 0$ such that $|\Gamma(\lambda) - V_N| \leq \overline{C}_1 / \sqrt{N}$.
Similarly, there also exists a finite constant $\overline{C}_2 > 0$ such that $|F_{\lambda}(\mathbf{x}) - F_{N}(\mathbf{x}) | \leq \overline{C}_2 / \sqrt{N}$ for any $\mathbf{x} \in \mathcal{X}$.
It follows that
\begin{align}
& \Gamma(\lambda) = V_N + O\left( \frac{1}{\sqrt{N}} \right), \label{asymptotic_HRO_SAA_1} \\
& F_{N}(\mathbf{x}) = F_{\lambda}(\mathbf{x}) + O\left( \frac{1}{\sqrt{N}} \right), \ \forall \, \mathbf{x} \in \mathcal{X}. \label{asymptotic_HRO_SAA_2}
\end{align}
Equation \eqref{asymptotic_HRO_SAA_1} (resp. \eqref{asymptotic_HRO_SAA_2}) indicates the relationship between the optimal values (resp. objective functions) of model \eqref{model_harmonizing} and SAA model \eqref{model: SAA}.
From Theorem 5.7 in \cite{shapiro2021lectures}, we obtain the relationship between the optimal value of the SAA model \eqref{model: SAA} (i.e., $V_N$) and its objective function (i.e., $F_{N}(\mathbf{x}$)) on the optimal solution set of primal model \eqref{model0} (i.e., $\mathcal{X}^*$), as introduced below.
\begin{align}
    V_N = \inf_{\mathbf{x} \in \mathcal{X}^*} F_{N}(\mathbf{x}) + o_p \left( \frac{1}{\sqrt{N}} \right), \label{asymptotic_HRO_SAA_shapiro_1}
\end{align}
where $ o_p (\cdot) $ refers to convergence in probability to 0. 
By substituting $V_N$ with $\Gamma(\lambda)$ from \eqref{asymptotic_HRO_SAA_1} and $F_{N}(\mathbf{x})$ with $F_{\lambda}(\mathbf{x})$ from \eqref{asymptotic_HRO_SAA_2}, we can transform \eqref{asymptotic_HRO_SAA_shapiro_1} into $\Gamma ( \lambda ) = \inf_{\mathbf{x} \in \mathcal{X}^*} F_{\lambda}(\mathbf{x}) + O  ( 1/\sqrt{N} )$, i.e., the first part in \eqref{eqn: asymptotic_order_2}.

In addition, from Theorem 5.7 in \cite{shapiro2021lectures}, we obtain the relationship between the optimal value of the original model \eqref{model0} (i.e., $V^*$) and that of the SAA model \eqref{model: SAA} (i.e., $V_N$), as detailed below.
\begin{align}
& \sqrt{N} \left( V_N - V^* \right) \xrightarrow{\mathcal{D}}  \inf_{\mathbf{x} \in \mathcal{X}^*} Y \left( \mathbf{x} \right), \label{asymptotic_HRO_SAA_shapiro_2} \\
& \sqrt{N} \left( V_N - V^* \right) \xrightarrow{\mathcal{D}}  \mathcal{N} \left(0, \sigma^2\left( \mathbf{x}^* \right) \right), \ \text{if} \ \mathcal{X}^* = \{\mathbf{x}^*\} \ \text{is a singleton}. \label{asymptotic_HRO_SAA_shapiro_3}
\end{align}

By \eqref{asymptotic_HRO_SAA_1}, we have $\Gamma(\lambda) \rightarrow V_N$, which, together with \eqref{asymptotic_HRO_SAA_shapiro_2} and \eqref{asymptotic_HRO_SAA_shapiro_3}, leads to 
\begin{align*}
& \sqrt{N} \left( \Gamma \left( \lambda \right)  - V^* \right) \xrightarrow{\mathcal{D}}  \inf_{\mathbf{x} \in \mathcal{X}^*} Y \left( \mathbf{x} \right), \\
& \sqrt{N} \left( \Gamma \left( \lambda \right)  - V^* \right) \xrightarrow{\mathcal{D}}  \mathcal{N} \left(0, \sigma^2\left( \mathbf{x}^* \right) \right), \ \text{if} \ \mathcal{X}^* = \{\mathbf{x}^*\} \ \text{is a singleton},
\end{align*}
which are exactly the second part of \eqref{eqn: asymptotic_order_2} and \eqref{eqn: asymptotic_order_3}, respectively.
\Halmos 


\subsection{
Details of Parameter Estimation
} \label{sec: details of parameter estimation}

We describe three different methods of choosing $C$. 
\begin{enumerate}[label=(\roman*)]
\item \textbf{$K$-fold cross-validation}. 
Ideally, we should choose $C^*$ such that the optimal solution of model \eqref{model_harmonizing}, denoted by $\mathbf{x}(C^*)$, exhibits the best performance under the true distribution $\mathbb{P}$ over all possible values of $C$. 
However, it is impossible to find such a $C^*$ because $\mathbb{P}$ is unknown. 
Here we adopt the $K$-fold cross-validation to estimate such a $C^*$ using the training data. 
Specifically, we divide data samples $\tilde{\boldsymbol{\xi}}_1, \ldots, \tilde{\boldsymbol{\xi}}_N$ into $K$ subsets $\mathcal{S}_1, \ldots, \mathcal{S}_K$, by which we run $K$ iterations. 
In each iteration $k \in [K]$, we choose subset $\mathcal{S}_k$ as the training set and the remaining subsets $\cup_{i \in [K] \setminus \{k\}} \mathcal{S}_i$ as the validation set. 
Given $\mathcal{S}_k$, we consider a large number of candidates of $C$. 
For each candidate of $C$, we solve the corresponding model \eqref{model_harmonizing} and obtain an optimal solution $\mathbf{x}(C)$.
Then, we evaluate the out-of-sample results of all these solutions using the validation set and identify the best solution, $x(C)$, along with its corresponding candidate of $C$, denoted by $C_k$.
After $K$ iterations, we set $C^*=\sum_{k \in [K]} C_k / K$.

\item \textbf{Tightening the confidence interval in Proposition \ref{prop: interval}}.
In Proposition \ref{prop: interval}, we introduce that the gap between the objective value of the SAA model, i.e., $F_N(\mathbf{x})$, and the objective value of the original model, i.e., $F(\mathbf{x})$, is $z_{\alpha/2} \hat{\sigma}(\mathbf{x}) / \sqrt{N}$.
Note that when $N$ is large enough, our proposed model \eqref{model_harmonizing} becomes almost the same as the SAA model, by which the gap between the objective value of model \eqref{model_harmonizing} and $F(\mathbf{x})$ is approximately $z_{\alpha/2} \hat{\sigma}(\mathbf{x}) / \sqrt{N}$.
Thus, we can find a $C^{\dag}$ such that the optimal solution $\mathbf{x}(C^{\dag})$ of model \eqref{model_harmonizing} minimizes the gap over all possible values of $C$.
Specifically, given that $\mathbf{x}_{\text{SAA}}$ and $\mathbf{x}_{\text{DRO}}$ are optimal solutions of models \eqref{model: SAA} and \eqref{DRO}, respectively, 
we use $\mathbf{x}(C) = (1-C/\sqrt{N})\mathbf{x}_{\text{SAA}} + C/\sqrt{N} \mathbf{x}_{\text{DRO}}$ to approximate the solution of \eqref{model_harmonizing} and set $C$ as a variable, by which we identify $C^{\dag}$ and the corresponding $\mathbf{x}(C^{\dag})$ that minimizes the gap $z_{\alpha/2} \hat{\sigma}(\mathbf{x}) / \sqrt{N}$.
Same as the above method (i), we divide data samples $\tilde{\boldsymbol{\xi}}_1, \ldots, \tilde{\boldsymbol{\xi}}_N$ into $K$ subsets $\mathcal{S}_1, \ldots, \mathcal{S}_K$, construct the training and validation sets, and run $K$ iterations.
In each iteration $k \in [K]$, we use the training set to solve models \eqref{model: SAA} and \eqref{DRO} to obtain their optimal solutions $\mathbf{x}_{\text{SAA}}$ and $\mathbf{x}_{\text{DRO}}$, respectively.
We then use the validation set and apply golden-section search method to solve $\min_{C} \{z_{\alpha/2} \hat{\sigma}(\mathbf{x}(C)) / \sqrt{N} \ | \ \mathbf{x}(C) \in \mathcal{X} \}$ to obtain the optimal solution $C_k$.
After $K$ iterations, we set $C^{\dag}=\sum_{k \in [K]} C_k / K$.

\item \textbf{Straightforward estimation}. 
We set $C = \sqrt{M_0}$, where $M_0$ denotes the smallest number of samples we may have; that is, $\lambda = \sqrt{M_0} / \sqrt{N}$.
When the number of considered samples is the smallest, i.e., $N = M_0$, our HO model is the same as the DRO model, ensuring the robustness of the obtained solution.
\end{enumerate}

\subsection{
Computationally Tractable Forms of Model \eqref{model_harmonizing_reform}
} \label{sec: tractable reform}

In this section, we demonstrate that model \eqref{model_harmonizing_reform} is a computationally tractable program for several ambiguity sets of practical interests.

\begin{proposition} \label{prop: reform_moment}
Incorporating the mean-covariance ambiguity set $\mathcal{D}_{\textup T}$, i.e., $\mathcal{D} = \mathcal{D}_{\textup T}$, model \eqref{model_harmonizing_reform} shares the same optimal value with the following SDP formulation:
\begin{align}
    \min\limits_{\mathbf{x}, \mathbf{w}, s, \mathbf{q}, \mathbf{Q}} \
    & \left( 1-\lambda \right) \frac{1}{N} \sum_{j=1}^N w_j + \lambda \left( s + \gamma_2 \mathbf{I} \bullet \mathbf{Q} + \sqrt{\gamma_1}\left \| \mathbf{q} \right \|_2 \right) \label{model: harmonizing_moment_reform} \\
    {\normalfont \text{s.t.}} \ & \mathbf{x} \in \mathcal{X}, \nonumber \\
    & w_j \geq \alpha_k(\mathbf{x})^{\top} \tilde{\boldsymbol{\xi}}_j + \beta_k(\mathbf{x}), \ \forall \, j \in [N], \, k \in [K], \nonumber \\
    & \begin{bmatrix}
    s - \beta_k \left( \mathbf{x} \right) - \alpha_k \left( \mathbf{x} \right)^{\top}\boldsymbol{\mu}
    & \hspace{0.1 in}  \frac{1}{2}  \left( \mathbf{q} - \left(\mathbf{U}\boldsymbol{\Lambda}^{{\frac{1}{2}}}\right)^{\top} \alpha_k \left( \mathbf{x} \right) \right)^{\top} \\ 
    \frac{1}{2} \left( \mathbf{q} - \left(\mathbf{U}\boldsymbol{\Lambda}^{{\frac{1}{2}}}\right)^{\top} \alpha_k \left( \mathbf{x} \right) \right)
    & \mathbf{Q}
\end{bmatrix} \succeq 0, \ \forall \, k \in [K], \nonumber
\end{align}
where $\mathbf{U} \in \mathbb{R}^{m \times m}$ is an orthogonal transformation matrix, $\boldsymbol{\Lambda} \in \mathbb{R}^{m \times m}$ is a diagonal matrix, and they are obtained by an eigenvalue decomposition on $\boldsymbol{\Sigma}$, i.e., $\boldsymbol{\Sigma} = \mathbf{U} \boldsymbol{\Lambda} \mathbf{U}^{\top} = \mathbf{U} \boldsymbol{\Lambda}^{1/2}(\mathbf{U} \boldsymbol{\Lambda}^{1/2})^{\top}$.
\end{proposition}

\begin{proof}{Proof.}
When using $\mathcal{D}_{\textup T}$, model \eqref{model_harmonizing} shares the same optimal value with
\begin{align} \label{model: harmonizing_moment}
    \min_{\mathbf{x} \in \mathcal{X}} \left\{ \left( 1 - \lambda \right) \mathbb{E}_{\mathbb{P}_0} \left[ f \left( \mathbf{x}, \boldsymbol{\xi} \right) \right] + \lambda \max_{\mathbb{P}_{\boldsymbol{\xi}} \in \Pi_{\boldsymbol{\xi}} \mathcal{D}_{\textup T}} \mathbb{E}_{\mathbb{P}_{\boldsymbol{\xi}}} \left[ f \left( \mathbf{x}, \boldsymbol{\xi} \right) \right] \right\},
\end{align}
where
\begin{align*}
\Pi_{\boldsymbol{\xi}} \mathcal{D}_{\textup T} = 
\left\{
    \mathbb{P}_{\boldsymbol{\xi}} \in \mathcal{D}_0 \left( \mathbb{R}^m \right)
   \ \middle| \
    \begin{array}{l}
    \mathbb{P}_{\boldsymbol{\xi}} \left( \boldsymbol{\xi} \in \mathcal{S} \right) = 1 \\
    \left( \mathbb{E}_{\mathbb{P}_{\boldsymbol{\xi}}} \left[ \boldsymbol{\xi} \right] -\boldsymbol{\mu} \right)^{\top} \boldsymbol{\Sigma}^{-1} \left( \mathbb{E}_{\mathbb{P}_{\boldsymbol{\xi}}}\left[ \boldsymbol{\xi} \right]-
    \boldsymbol{\mu} \right)\leq \gamma_1 \\
    \mathbb{E}_{\mathbb{P}_{\boldsymbol{\xi}}} \left[ \left( \boldsymbol{\xi} - \boldsymbol{\mu} \right) \left( \boldsymbol{\xi} - \boldsymbol{\mu} \right)^{\top} \right]  \preceq \gamma_2\boldsymbol{\Sigma}
    \end{array}
\right\}.
\end{align*}
By Proposition 1 in \cite{cheramin2022computationally}, we can reformulate model \eqref{model: harmonizing_moment} as \eqref{model: harmonizing_moment_reform}.
\end{proof}

\begin{proposition} \label{prop: reform_mad}
Incorporating the MAD ambiguity set $\mathcal{D}_{\textup D}$, i.e., $\mathcal{D}=\mathcal{D}_{\textup D}$, model \eqref{model_harmonizing_reform} shares the same optimal value with the following LP formulation:
\begin{align} \label{model: harmonizing_mad_reform}
    \min\limits_{\mathbf{x}, \mathbf{w}, s, \mathbf{q}, \boldsymbol{\pi}} \ &  \left( 1-\lambda \right) \frac{1}{N} \sum_{j=1}^N w_j + \lambda \left( s + \boldsymbol{\delta}^{\top} \mathbf{q} \right) \\
    {\normalfont \text{s.t.}} \ & \mathbf{x} \in \mathcal{X}, \nonumber \\
    & w_j \geq \alpha_k(\mathbf{x})^{\top} \tilde{\boldsymbol{\xi}}_j + \beta_k(\mathbf{x}), \ \forall \, j \in [N], \, k \in [K], \nonumber \nonumber \\
    & \alpha_k\left( \mathbf{x} \right)^{\top} \boldsymbol{\mu} + \beta_k \left( \mathbf{x} \right) \leq s, \ \forall \, k \in [K], \nonumber \\
    & | \alpha_k\left( \mathbf{x} \right) + \boldsymbol{\pi} | \leq \mathbf{q}, \ \forall \, k \in [K], \nonumber \\
    & \mathbf{q} \geq \mathbf{0}. \nonumber
\end{align}

\end{proposition}

\begin{proof}{Proof.}
When using $\mathcal{D}_{\textup D}$, model \eqref{model_harmonizing} shares the same optimal value with
    \begin{align} \label{model: harmonizing_mad}
        \min_{\mathbf{x} \in \mathcal{X}} \left\{ \left( 1 - \lambda \right) \mathbb{E}_{\mathbb{P}_0} \left[ f \left( \mathbf{x}, \boldsymbol{\xi} \right) \right] + \lambda \max_{\mathbb{P}_{\boldsymbol{\xi}} \in \Pi_{\boldsymbol{\xi}} \mathcal{D}_{\textup D}} \mathbb{E}_{\mathbb{P}_{\boldsymbol{\xi}}} \left[ f \left( \mathbf{x}, \boldsymbol{\xi} \right) \right] \right\},
\end{align}
where
\begin{align} \label{mad_ambiguity_set}
    \Pi_{\boldsymbol{\xi}} \mathcal{D}_{\textup D} = 
    \left\{
        \mathbb{P}_{\boldsymbol{\xi}} \in \mathcal{D}_0 \left( \mathbb{R}^m \right)
       \ \middle| \
        \begin{array}{l}
        \mathbb{P}_{\boldsymbol{\xi}} \left( \boldsymbol{\xi} \in \mathcal{S} \right) = 1 \\
        \mathbb{E}_{\mathbb{P}_{\boldsymbol{\xi}}} \left[ \boldsymbol{\xi} \right] = \boldsymbol{\mu} \\
        \mathbb{E}_{\mathbb{P}_{\boldsymbol{\xi}}} \left[ | \boldsymbol{\xi} - \boldsymbol{\mu} | \right]  \leq \boldsymbol{\delta}
        \end{array}
    \right\}.
\end{align}

Introducing dual variables $s \in \mathbb{R}$, $\boldsymbol{\pi} \in \mathbb{R}^m$, and $\mathbf{q} \in \mathbb{R}_+^m$ with respect to the three constraints on $\mathbb{P}_{\boldsymbol{\xi}}$ in \eqref{mad_ambiguity_set}, 
we then present the Lagrange dual form of $\max_{\mathbb{P}_{\boldsymbol{\xi}} \in \Pi_{\boldsymbol{\xi}} \mathcal{D}_{\textup D}} \mathbb{E}_{\mathbb{P}_{\boldsymbol{\xi}}} [ f ( \mathbf{x}, \boldsymbol{\xi} ) ]$ in model \eqref{model: harmonizing_mad} as
\begin{align}
    \min\limits_{s, \mathbf{q}, \boldsymbol{\pi}} \ & s + \boldsymbol{\delta}^{\top} \mathbf{q} \label{model: mad_innermax} \\
    {\normalfont \text{s.t.}} \ & f \left( \mathbf{x}, \boldsymbol{\xi} \right) + \boldsymbol{\pi}^{\top} \left( \boldsymbol{\xi} - \boldsymbol{\mu} \right) - \mathbf{q}^{\top} |\boldsymbol{\xi} - \boldsymbol{\mu} | \leq s, \ \forall \, \boldsymbol{\xi} \in \mathbb{R}^m, \label{ineq: mad_cons_dual} \\
    & \mathbf{q} \geq \mathbf{0}. \nonumber
\end{align}
By the assumption of $f(\mathbf{x}, \boldsymbol{\xi}) = \max_{k \in [K]} \{ \alpha_k(\mathbf{x})^{\top} \boldsymbol{\xi} + \beta_k(\mathbf{x}) \}$ (see the beginning of this section) 
and $\mathbf{q}^{\top} |\boldsymbol{\xi} - \boldsymbol{\mu} | = \max_{|\mathbf{z}| \leq \mathbf{q}} \mathbf{z}^{\top} (\boldsymbol{\xi} - \boldsymbol{\mu})$, 
we have
\begin{align}
    \eqref{ineq: mad_cons_dual}
    & \Leftrightarrow 
    \max_{k \in [K]} \left \{\alpha_k \left( \mathbf{x} \right)^{\top} \boldsymbol{\xi} + \beta_k \left( \mathbf{x} \right) \right\} + \boldsymbol{\pi}^{\top} \left( \boldsymbol{\xi} - \boldsymbol{\mu} \right) - \mathbf{q}^{\top} |\boldsymbol{\xi} - \boldsymbol{\mu} | \leq s, \ \forall \, \boldsymbol{\xi} \in \mathbb{R}^m \nonumber \\
    & \Leftrightarrow 
    \alpha_k \left( \mathbf{x} \right)^{\top} \boldsymbol{\xi} + \beta_k \left( \mathbf{x} \right) + \boldsymbol{\pi}^{\top} \left( \boldsymbol{\xi} - \boldsymbol{\mu} \right) - \mathbf{q}^{\top} |\boldsymbol{\xi} - \boldsymbol{\mu} | \leq s, \ \forall \, \boldsymbol{\xi} \in \mathbb{R}^m, k \in [K]  \nonumber \\
    & \Leftrightarrow 
    \alpha_k \left( \mathbf{x} \right)^{\top} \boldsymbol{\xi} + \beta_k \left( \mathbf{x} \right) + \boldsymbol{\pi}^{\top} \left( \boldsymbol{\xi} - \boldsymbol{\mu} \right) - \max_{|\mathbf{z}_k| \leq \mathbf{q}} \left\{ \mathbf{z}_k^{\top} \left( \boldsymbol{\xi} - \boldsymbol{\mu} \right) \right\} \leq s, \ \forall \, \boldsymbol{\xi} \in \mathbb{R}^m, k \in [K]  \nonumber \\
    & \Leftrightarrow 
    \max_{\boldsymbol{\xi} \in \mathbb{R}^m} \min_{|\mathbf{z}_k| \leq \mathbf{q}} \ \alpha_k \left( \mathbf{x} \right)^{\top} \boldsymbol{\xi} + \beta_k \left( \mathbf{x} \right) + \boldsymbol{\pi}^{\top} \left( \boldsymbol{\xi} - \boldsymbol{\mu} \right) - \mathbf{z}_k^{\top} \left( \boldsymbol{\xi} - \boldsymbol{\mu} \right) \leq s, \ \forall \, k \in [K]  \nonumber \\
    & \Leftrightarrow 
    \min_{|\mathbf{z}_k| \leq \mathbf{q}} \max_{\boldsymbol{\xi} \in \mathbb{R}^m} \ \alpha_k \left( \mathbf{x} \right)^{\top} \boldsymbol{\xi} + \beta_k \left( \mathbf{x} \right) + \boldsymbol{\pi}^{\top} \left( \boldsymbol{\xi} - \boldsymbol{\mu} \right) - \mathbf{z}_k^{\top} \left( \boldsymbol{\xi} - \boldsymbol{\mu} \right) \leq s, \ \forall \, k \in [K]  \label{eqn: maxmin} \\
    & \Leftrightarrow 
    \exists \, \mathbf{z}_k, \ {\normalfont \text{s.t.}} \ |\mathbf{z}_k| \leq \mathbf{q}, \ 
    \max_{\boldsymbol{\xi} \in \mathbb{R}^m} \ \alpha_k \left( \mathbf{x} \right)^{\top} \boldsymbol{\xi} + \beta_k \left( \mathbf{x} \right) + \boldsymbol{\pi}^{\top} \left( \boldsymbol{\xi} - \boldsymbol{\mu} \right) - \mathbf{z}_k^{\top} \left( \boldsymbol{\xi} - \boldsymbol{\mu} \right) \leq s, \ \forall \, k \in [K]  \nonumber \\
    & \Leftrightarrow 
    \exists \, \mathbf{z}_k, \ {\normalfont \text{s.t.}} \ |\mathbf{z}_k| \leq \mathbf{q}, \ \max_{\boldsymbol{\xi} \in \mathbb{R}^m} \ 
    \left( \alpha_k \left( \mathbf{x} \right) + \boldsymbol{\pi} - \mathbf{z}_k \right)^{\top} \boldsymbol{\xi} \leq s - \beta_k \left( \mathbf{x} \right) + \boldsymbol{\pi}^{\top}\boldsymbol{\mu} - \mathbf{z}_k^{\top}\boldsymbol{\mu}, \ \forall \, k \in [K]  \label{eqn: max_xi} \\
    & \Leftrightarrow 
    \exists \, \mathbf{z}_k, \ {\normalfont \text{s.t.}} \ |\mathbf{z}_k| \leq \mathbf{q}, \ 0 \leq s - \beta_k \left( \mathbf{x} \right) + \boldsymbol{\pi}^{\top}\boldsymbol{\mu} - \mathbf{z}_k^{\top}\boldsymbol{\mu}, \ \alpha_k \left( \mathbf{x} \right) + \boldsymbol{\pi} - \mathbf{z}_k = \mathbf{0}, \ \forall \, k \in [K], \label{ineq: mad_cons_dual_reform}
\end{align}
where equivalence \eqref{eqn: maxmin} holds by the Sion’s minimax theorem \citep{sion1958general} because function $\alpha_k ( \mathbf{x} )^{\top} \boldsymbol{\xi} + \beta_k ( \mathbf{x} ) + \boldsymbol{\pi}^{\top} ( \boldsymbol{\xi} - \boldsymbol{\mu} ) - \mathbf{z}_k^{\top} ( \boldsymbol{\xi} - \boldsymbol{\mu} )$ is concave (specifically, linear) on $\boldsymbol{\xi}$ and convex (specifically, linear) on $\mathbf{z}_k$, and the feasible region defined by $|\mathbf{z}_k| \leq \mathbf{q}$ is compact and convex for any finite $\mathbf{q} \geq \mathbf{0}$.
Equivalence \eqref{ineq: mad_cons_dual_reform} holds due to $\boldsymbol{\xi} \in \mathbb{R}^m$, which implies that $\alpha_k ( \mathbf{x} ) + \boldsymbol{\pi} - \mathbf{z}_k$ has to be $\mathbf{0}$ for any $k \in [K]$, otherwise the left-hand side of \eqref{eqn: max_xi} goes to infinity.

By replacing \eqref{ineq: mad_cons_dual} with \eqref{ineq: mad_cons_dual_reform} and $\mathbf{z}_k$ with $\alpha_k (\mathbf{x}) + \boldsymbol{\pi}$ for any $k \in [K]$, we can reformulate \eqref{model: mad_innermax} as
\begin{align} \label{model: mad_innermax_reform}
    \min\limits_{s, \mathbf{q}, \boldsymbol{\pi}} \ & s + \boldsymbol{\delta}^{\top} \mathbf{q} \\
    {\normalfont \text{s.t.}} \ 
    & \alpha_k\left( \mathbf{x} \right)^{\top} \boldsymbol{\mu} + \beta_k \left( \mathbf{x} \right) \leq s, \ \forall \, k \in [K], \nonumber \\
    & | \alpha_k\left( \mathbf{x} \right) + \boldsymbol{\pi} | \leq \mathbf{q}, \ \forall \, k \in [K], \nonumber \\
    & \mathbf{q} \geq \mathbf{0}. \nonumber
\end{align}

\noindent
By further integrating \eqref{model: mad_innermax_reform} with the outer minimization problem of \eqref{model: harmonizing_mad}, we then obtain \eqref{model: harmonizing_mad_reform}.
\end{proof}

\section{Supplement to Section \ref{sec: Numerical Experiments}}

\subsection{Parameter Settings in Numerical Experiments} \label{app: Parameter Settings in Numerical Experiments}

In the numerical experiments for the mean-risk portfolio optimization problem, we set $m=10$, $a = 0.2$, and $\rho = 10$. 
For any asset $i = 1, \ldots, m$, its uncertain return $\xi_i$ can be decomposed into a systematic risk factor $\phi \in \mathbb{R}$, which is common to all assets, and an idiosyncratic risk factor $\epsilon_i \in \mathbb{R}$: $\xi_i = \phi + \epsilon_i$. 
Here $\phi \sim \mathcal{N}\left( 0, 0.02 \right)$ and $\epsilon_i \sim \mathcal{N} \left( i \times 0.03, i \times 0.025 \right)$ for any $i = 1, \ldots, m$, by which we draw the training and test samples. 
Under this setting, assets with higher indices promise higher mean returns at a higher risk.

In the numerical experiments for the lot sizing problem, we set $m = 30$, $a_i \sim U(0.5, 1.5)$, $c_i = 5\sum_{j \in [m]} b_{j,i}$ for any $i \in [m]$, $Y_{i,j} = 1$ for any $i, j \in [m]$ if $i \neq j$ and $Y_{i,j} = 0$ otherwise, $\mu_i \sim U(300, 420)$, $\underline{d}_i \sim U(60, \mu_i-60)$, $\overline{d}_i \sim U(\mu_i+60, 660)$, $\xi_i \sim U(\underline{d}_i, \overline{d}_i)$, and $K_i = \overline{d}_i$ for any $i \in [m]$.
For any $i, j \in [m]$, we set $b_{i,j} = 0$ if $i=j$ and $b_{i,j} \sim U(\underline{b}_{i,j}, \underline{b}_{i,j} + 1)$ if $i \neq j$, where
\begin{align*}
    & \underline{b}_{i,j} = \begin{cases} 
          1+0.5k, \ & \text{ if }  |i-j| \in [4k+1, 4(k+1)]_{\mathbb{Z}}, \ k \in [0, 6]_{\mathbb{Z}} \\
        4.5, \ & \text { otherwise }
    \end{cases}.
\end{align*}
We draw both the training and test samples based on the above setting.
In $\mathcal{D}_{\textup D}$, the support $\mathcal{S} = \{\boldsymbol{\xi} \ | \ \underline{\mathbf{d}} \leq \boldsymbol{\xi} \leq \overline{\mathbf{d}} \}$ and the mean vector $\boldsymbol{\mu}$ is estimated from all the $N$ training samples.

\subsection{Setup for Mean-risk Portfolio Optimization Problem} \label{app: setup portfolio}

By Proposition \ref{prop: reform_moment}, we can reformulate problem \eqref{model: portfolio_2} with $\mathcal{D}$ being $\mathcal{D}_{\textup T}$ as
\begin{align*}
        \min\limits_{\mathbf{x}, \mathbf{w}, \tau, s, \mathbf{q}, \mathbf{Q}} \
        & \left( 1-\lambda \right) \frac{1}{N} \sum_{j=1}^N w_j + \lambda \left( s + \gamma_2 \mathbf{I} \bullet \mathbf{Q} + \sqrt{\gamma_1}\left \| \mathbf{q} \right \|_2 \right) \\
        \text{s.t.} \ & \sum_{i=1}^m x_i = 1, \\
        & x_i \geq 0, \ \forall \, i \in [m], \\
        & w_j \geq \alpha_k\mathbf{x}^{\top} \tilde{\boldsymbol{\xi}}_j + \beta_k \tau, \ \forall \, j \in [N], \, k \in [K], \\
        & \begin{bmatrix} 
            s - \beta_k \tau - \alpha_k \mathbf{x}^{\top}\boldsymbol{\mu}
            & \hspace{0.1 in}  \frac{1}{2}  \left( \mathbf{q} - \left(\mathbf{U}\boldsymbol{\Lambda}^{{\frac{1}{2}}}\right)^{\top} \alpha_k \mathbf{x} \right)^{\top} \\ 
            \frac{1}{2} \left( \mathbf{q} - \left(\mathbf{U}\boldsymbol{\Lambda}^{{\frac{1}{2}}}\right)^{\top} \alpha_k \mathbf{x}  \right)
            & \mathbf{Q}
    \end{bmatrix} \succeq 0, \ \forall \, k \in [K].
\end{align*}

\noindent
Also, we reformulate problem \eqref{model: portfolio_2} with $\mathcal{D}$ being $\mathcal{D}_{\textup D}$ as
\begin{align*}
    \min\limits_{\mathbf{x}, \mathbf{w}, \tau, s, \mathbf{q}, \boldsymbol{\pi}} \ &  \left( 1-\lambda \right) \frac{1}{N} \sum_{j=1}^N w_j + \lambda \left( s + \boldsymbol{\delta}^{\top} \mathbf{q} \right) \\
    \text{s.t.} \ & \sum_{i=1}^m x_i = 1, \\
    & x_i \geq 0, \ \forall \, i \in [m], \\
    & w_j \geq \alpha_k\mathbf{x}^{\top} \tilde{\boldsymbol{\xi}}_j + \beta_k \tau, \ \forall \, j \in [N], \, k \in [K], \\
    & \alpha_k \mathbf{x}^{\top} \boldsymbol{\mu} + \beta_k \tau \leq s, \ \forall \, k \in [K], \nonumber \\
    & | \alpha_k \mathbf{x} + \boldsymbol{\pi} | \leq \mathbf{q}, \ \forall \, k \in [K], \nonumber \\
        & \mathbf{q} \geq \mathbf{0}. \nonumber
\end{align*}

\subsection{Setup for Lot Sizing on a Network} \label{app: lot sizing}

We apply Algorithm 1 in \cite{long2023supermodularity} to solve the two-stage HO model \eqref{model: lot sizing} with $\mathcal{D}$ being $\mathcal{D}_{\textup D}$.
The goal of this algorithm is to identify the worst-case distribution $\mathbb{P}_{\boldsymbol{\xi}}^* \in \Pi_{\boldsymbol{\xi}}\mathcal{D}_{\textup D}$ of model \eqref{model: lot sizing}.
Once we obtain $\mathbb{P}_{\boldsymbol{\xi}}^*$, we can then solve model \eqref{model: lot sizing} by solving the following model:
\begin{align} \label{model: lot sizing_algorithm}
    \min_{\mathbf{x}} \ \left\{ \mathbf{a}^{\top} \mathbf{x} + \left( 1 - \lambda \right) \mathbb{E}_{\mathbb{P}_0} \left[ f \left( \mathbf{x}, \boldsymbol{\xi} \right) \right] + \lambda  \mathbb{E}_{\mathbb{P}_{\boldsymbol{\xi}}^*}\left[ f \left( \mathbf{x}, \boldsymbol{\xi} \right) \right] \ | \
    0 \leq x_i \leq K_i, \ \forall \, i \in [m] \right\}.
\end{align}

\noindent
Before applying this algorithm, we need to initially find a $\mathbb{P}_{\boldsymbol{\xi}}^{\dag} \in \arg\sup_{\mathbb{P}_{\boldsymbol{\xi}} \in \Pi_{\boldsymbol{\xi}}\mathcal{D}_{\textup D}} \mathbb{E}_{\mathbb{P}_{\boldsymbol{\xi}}}[ f ( \mathbf{x}, \boldsymbol{\xi} ) ]$ such that its marginal distribution in $i$ is independent of $\mathbf{x}$ for all $i \in [m]$.
By Proposition 1 in \cite{long2023supermodularity}, we have
\begin{align} \label{eqn: P_calculate}
    \mathbb{P}_{\boldsymbol{\xi}}^{\dag}\left( \xi_i=v \right) = 
    \begin{cases}
        \frac{\hat{\delta}_i}{2\left(\mu_i-\underline{d}_i\right)}, & \text{ if } v=\underline{d}_i \\ 
        1-\frac{\hat{\delta}_i\left(\overline{d}_i-\underline{d}_i\right)}{2\left(\overline{d}_i-\mu_i\right)\left(\mu_i-\underline{d}_i\right)}, & \text { if } v=\mu_i \\ 
        \frac{\hat{\delta}_i}{2\left(\overline{d}_i-\mu_i\right)}, & \text { if } v=\overline{d}_i \\ 
        0, & \text { otherwise }
    \end{cases},
\end{align}
where $\hat{\delta}_i = \min \{ \delta_i, \ 2(\overline{z}_i-\mu_i)(\mu_i-\underline{z}_i) / (\overline{z}_i-\underline{z}_i) \}$ for all $i \in [m]$ with $\overline{z}_i \geq \underline{z}_i$.
With \eqref{eqn: P_calculate}, we can then use Algorithm 1 in \cite{long2023supermodularity} to obtain the $\mathbb{P}_{\boldsymbol{\xi}}^*$.
\begin{algorithm}[h]
\SingleSpacedXI
\footnotesize
\caption{Algorithm 1 in \cite{long2023supermodularity}}
\begin{algorithmic}[1]
\Require
$\mathcal{D}_{\textup D}$ with given $\boldsymbol{\mu}$, $\boldsymbol{\delta}$, $\underline{\mathbf{d}}$, and $\overline{\mathbf{d}}$.
\State Denote $\mathbb{P}_{\boldsymbol{\xi}}^{\dag}$ obtained by \eqref{eqn: P_calculate} as the worst-case distribution and calculate $\mathbb{P}_{\boldsymbol{\xi}}^{\dag}( \xi_i=v )$ for $v \in \{\underline{d}_i, \mu_i, \overline{d}_i\}$ for any $i \in [m]$ using \eqref{eqn: P_calculate}.
\State Set $\hat{\boldsymbol{\xi}}^1 = \underline{\mathbf{d}}$, $\mathbf{q}^1 = ( \mathbb{P}_{\boldsymbol{\xi}}^{\dag}( \xi_1=\underline{d}_1), \mathbb{P}_{\boldsymbol{\xi}}^{\dag}(\xi_2=\underline{d}_2), \ldots, \mathbb{P}_{\boldsymbol{\xi}}^{\dag}(\xi_m=\underline{d}_m) )$, $p_1 = \min \{q_1^1, \ldots, q_m^1\}$, and $r=1$.
\For{$r \leq 2m$}
\State Set $l_r = \min \{i \in [m] \ | \ q_i^r = p_r  \}$.
\State Set $\hat{\boldsymbol{\xi}}^{r+1} = \hat{\boldsymbol{\xi}}^r$, $\mathbf{q}^{r+1} = \mathbf{q}^r - p_r \mathbf{1}$.
\State Set $\hat{\xi}_{l_r}^{r+1} = \mu_{l_r}$ if its existing value is $\underline{d}_{l_r}$ and $\hat{\xi}_{l_r}^{r+1} = \overline{d}_{l_r}$ if its existing value is $\mu_{l_r}$.
\State Set $q_{l_r}^{r+1} = \mathbb{P}_{\boldsymbol{\xi}}^{\dag}(\xi_{l_r} = \hat{\xi}_{l_r}^{r+1})$.
\State Set $p_{r+1} = \min \{ q_1^{r+1}, q_2^{r+1}, \ldots, q_m^{r+1} \}$.
\State Set $r = r+1$.
\EndFor
\Ensure
$\hat{\boldsymbol{\xi}}^1, \hat{\boldsymbol{\xi}}^2, \ldots, \hat{\boldsymbol{\xi}}^{2m+1}$ and $\mathbf{p} = (p_1, p_2, \ldots, p_{2m+1})^{\top}$.
\end{algorithmic}
\end{algorithm}

With obtained $\hat{\boldsymbol{\xi}}^1, \hat{\boldsymbol{\xi}}^2, \ldots, \hat{\boldsymbol{\xi}}^{2m+1}$ and $\mathbf{p}$, we have $\mathbb{P}_{\boldsymbol{\xi}}^* = \sum_{j=1}^{2m+1}p_j \delta_{\hat{\boldsymbol{\xi}}^j}$ and reformulate model \eqref{model: lot sizing_algorithm} as
\begin{align*}
    \min_{\mathbf{x}} \ \left\{ \mathbf{a}^{\top} \mathbf{x} + \left( 1 - \lambda \right) \mathbb{E}_{\mathbb{P}_0} \left[ f \left( \mathbf{x}, \boldsymbol{\xi} \right) \right] + \lambda  \sum_{j\in[2m+1]} p_j f \left( \mathbf{x}, \hat{\boldsymbol{\xi}}^j \right)  \ | \
    0 \leq x_i \leq K_i, \ \forall \, i \in [m] \right\},
\end{align*}
which is a linear programming (LP) model and can be solved easily.

Next, we introduce a local search algorithm designed for the scenario reduction problem \eqref{scenario_reduction_problem} as described in \cite{rujeerapaiboon2022scenario}.
Here, for any set $\tilde{\mathcal{S}} \in \mathcal{S}_0(M)$ and $M < N$, we define $G_l^{\prime} \left( \mathbb{P}_{\textup N}, \tilde{\mathcal{S}} \right) = \min _{\mathbb{Q}} \left\{ d_l \left( \mathbb{P}_{\textup N}, \mathbb{Q} \right) : \mathbb{Q} \in \mathcal{D}_0 \left( \tilde{\mathcal{S}} \right) \right\}$, which measures the type-$l$ Wasserstein distance between $\mathbb{P}_{\textup N}$ and its closest discrete distribution supported on $\tilde{\mathcal{S}}$.

\begin{algorithm}[h]
\SingleSpacedXI
\footnotesize
\caption{Local Search Algorithm in \cite{rujeerapaiboon2022scenario}}
\begin{algorithmic}[1]
\State Initialize the reduced set $\tilde{\mathcal{S}} \subseteq \{ \tilde{\boldsymbol{\xi}}_1, \ldots, \tilde{\boldsymbol{\xi}}_N \}$ with $|\tilde{\mathcal{S}}| = M$, arbitrarily.
\State Select the next exchange to be applied to $\tilde{\mathcal{S}}$ as
\begin{align*}
    \left( \tilde{\boldsymbol{\zeta}}, \tilde{\boldsymbol{\zeta}}^{\prime} \right) 
    \in \arg\min \left\{ 
    G_l^{\prime} \left( \mathbb{P}_{\textup N}, \tilde{\mathcal{S}} \cup \left\{ \boldsymbol{\zeta} \right\} \setminus \left\{ \boldsymbol{\zeta}^{\prime} \right\} \right) 
    \ : \ 
    \left( \boldsymbol{\zeta}, \boldsymbol{\zeta}^{\prime} \right) \in
    \left( \left\{ \tilde{\boldsymbol{\xi}}_1, \ldots, \tilde{\boldsymbol{\xi}}_N \right\} \setminus \tilde{\mathcal{S}} \right)
    \times \tilde{\mathcal{S}}
    \right\},
\end{align*}
and set $\tilde{\mathcal{S}} = \tilde{\mathcal{S}} \cup \{ \tilde{\boldsymbol{\zeta}} \} \setminus \{ \tilde{\boldsymbol{\zeta}}^{\prime} \}$ if $G_l^{\prime} ( \mathbb{P}_{\textup N}, \tilde{\mathcal{S}} \cup \{ \tilde{\boldsymbol{\zeta}} \} \setminus \{ \tilde{\boldsymbol{\zeta}}^{\prime} \} ) < G_l^{\prime} ( \mathbb{P}_{\textup N}, \tilde{\mathcal{S}} )$. \label{alg:local search_selection step}
\State Repeat Step \ref{alg:local search_selection step} until no further improvement is possible.
\end{algorithmic} \label{alg:local search}
\end{algorithm}

We initialize $\tilde{\mathcal{S}}$ using the results from applying $k$-means clustering algorithm to  samples $\tilde{\boldsymbol{\xi}}_1, \ldots, \tilde{\boldsymbol{\xi}}_N$.
We denote the latest reduced set obtained after Algorithm \ref{alg:local search} terminates as  $\tilde{\mathcal{S}}^*=\{ \tilde{\boldsymbol{\zeta}}_1^*, \ldots, \tilde{\boldsymbol{\zeta}}_M^* \}$.
Following the steps introduced in \cite{rujeerapaiboon2022scenario}, we can recover the distribution $\mathbb{Q}^*$ on the reduced set $\tilde{\mathcal{S}}^*$ as $\mathbb{Q}^* = \sum_{j=1}^M \omega_j \delta_{\tilde{\boldsymbol{\zeta}}_j^*}$ with the probability $\omega_j = |I_j| / N$ for any $j \in [M]$.
The sets $I_j \subseteq \{ \tilde{\boldsymbol{\xi}}_1, \ldots, \tilde{\boldsymbol{\xi}}_N \}$ ($\forall j \in [M]$) constitute a partition of $\{ \tilde{\boldsymbol{\xi}}_1, \ldots, \tilde{\boldsymbol{\xi}}_N \}$, i.e., $\cup_{j \in [M]} I_j = \{ \tilde{\boldsymbol{\xi}}_1, \ldots, \tilde{\boldsymbol{\xi}}_N \}$ and $I_i \cap I_j = \emptyset$ for any $i \neq j$, such that $I_j$ contains all elements of $\{ \tilde{\boldsymbol{\xi}}_1, \ldots, \tilde{\boldsymbol{\xi}}_N \}$ closest to $\tilde{\boldsymbol{\zeta}}_j^*$, in terms of the Euclidean norm.
 
With reduced set $\tilde{\mathcal{S}}^*$ and its distribution $\mathbb{Q}^*$, we can formulate model \eqref{model: lot sizing} under the stochastic programming framework as 
\begin{align} \label{model: lot sizing_local search}
    \min_{\mathbf{x}} \ \left\{ \mathbf{a}^{\top} \mathbf{x} + \sum_{j \in [M]} \omega_j f \left( \mathbf{x}, \tilde{\boldsymbol{\zeta}}_j^* \right) \ | \
    0 \leq x_i \leq K_i, \ \forall \, i \in [m] \right\}.
\end{align}

The ``Local Search" approach introduced in Section \ref{sec: lot sizing} first applies Local Search Algorithm \ref{alg:local search} to obtain $\tilde{\mathcal{S}}^*$ and $\mathbb{Q}^*$, and then solve model \eqref{model: lot sizing_local search}.
In our experiment, we set $l=1$ and $\eta_i = 1/N$ for any $i \in [N]$ for the ``Local Search" approach.
The ``Random" approach obtains $\tilde{\mathcal{S}}^{\prime} = \{\tilde{\boldsymbol{\zeta}}_1^{\prime}, \ldots, \tilde{\boldsymbol{\zeta}}_M^{\prime}\}$ by randomly selecting these $M$ scenarios from $\{ \tilde{\boldsymbol{\xi}}_1, \ldots, \tilde{\boldsymbol{\xi}}_N \}$ and establish the empirical distribution on $\tilde{\mathcal{S}}^{\prime}$, i.e., $\tilde{\mathbb{P}}_0 (\tilde{\mathcal{S}}^{\prime}) = \sum_{j=1}^M \delta_{\tilde{\boldsymbol{\zeta}}_j^{\prime}} / M$, and then solve the following model:
\begin{align*}
    \min_{\mathbf{x}} \ \left\{ \mathbf{a}^{\top} \mathbf{x} + \sum_{j \in [M]} \frac{1}{M} f \left( \mathbf{x}, \tilde{\boldsymbol{\zeta}}_j^{\prime} \right) \ | \
    0 \leq x_i \leq K_i, \ \forall \, i \in [m] \right\}.
\end{align*}

\subsection{Computational Performance of Different Scenario Reduction Approaches} \label{app:computational_performance}

Tables \ref{tab:objgap_n_30_N_1000}--\ref{tab:preparetime_n_30_N_1000} provide the performance of various scenario reduction approaches when $N = 1000$.
\begin{table}[htbp]
    \centering
    \scriptsize
    \caption{Approximation Error (\%) When $N=1000$} 
    \scalebox{0.95}{
    \begin{tabular}{cccccc}
    \toprule
    $M$ & \textbf{MAD\_$\sqrt{M_0}$} & \textbf{MAD\_Gap} & \textbf{MAD\_Cross} & \textbf{Random} & \textbf{ \tabincell{c}{Local \\ Search} } \\
    \midrule
    10    & 4.48  & 4.48  & 7.67  & 184.08  & 350.28  \\
    20    & 4.48  & 4.46  & 5.76  & 40.08  & 148.55  \\
    30    & 4.48  & 4.45  & 3.71  & 23.80  & 80.50  \\
    40    & 4.47  & 4.40  & 4.55  & 15.16  & 55.62  \\
    50    & 4.47  & 4.31  & 3.57  & 7.90  & 38.05  \\
    \bottomrule
    \end{tabular}%
    }
  \label{tab:objgap_n_30_N_1000}%
\end{table}

\vspace{-6mm}

\begin{table}[H]
  \begin{minipage}{0.5\textwidth}
    \centering
    \scriptsize
    \caption{Computational Time (s) When $N=1000$} 
    \scalebox{0.95}{
    \begin{tabular}{cccccc}
    \toprule
    $M$ & \textbf{MAD\_$\sqrt{M_0}$} & \textbf{MAD\_Gap} & \textbf{MAD\_Cross} & \textbf{Random} & \textbf{ \tabincell{c}{Local \\ Search} } \\
    \midrule
    10    & 367.31  & 443.83  & 417.55  & 10.16  & 7.47  \\
    20    & 497.07  & 506.16  & 511.40  & 27.93  & 17.71  \\
    30    & 805.34  & 773.85  & 631.82  & 79.91  & 79.73  \\
    40    & 913.05  & 744.54  & 756.72  & 150.99  & 210.81  \\
    50    & 883.98  & 964.92  & 885.89  & 242.44  & 316.55  \\
    \bottomrule
    \end{tabular}%
	}
	\label{tab:time_n_30_N_1000}
  \end{minipage}
  %
  \begin{minipage}{0.5\textwidth}
    \centering
    \scriptsize
    \caption{Preparation Time (s) When $N=1000$} 
    \scalebox{0.95}{
    \begin{tabular}{cccccc}
    \toprule
    $M$ & \textbf{MAD\_$\sqrt{M_0}$} & \textbf{MAD\_Gap} & \textbf{MAD\_Cross} & \textbf{Random} & \textbf{ \tabincell{c}{Local \\ Search} } \\
    \midrule
    10    & 0     & 4,168.07  & 9,855.43  & 0     & 3,600  \\
    20    & 0     & 0     & 0     & 0     & 3,600  \\
    30    & 0     & 0     & 0     & 0     & 3,600  \\
    40    & 0     & 0     & 0     & 0     & 3,600  \\
    50    & 0     & 0     & 0     & 0     & 3,600  \\
    \bottomrule
    \end{tabular}%
    }
  \label{tab:preparetime_n_30_N_1000}%
  \end{minipage}
\end{table}

\end{APPENDICES}

\end{document}